\documentclass{amsart}

\usepackage{amsfonts, amssymb, amsmath}
\usepackage{amsthm}                               % AMS symbols
\usepackage{setspace}                                                         % Spacing
\usepackage{geometry}                                                         % Spacing
\usepackage{color}                                                            % Text coloring
\usepackage{graphicx}                                                         % Graphics
\usepackage{slashed}                                                          % Slashed notations
\usepackage{hyperref}
\usepackage{float}
\usepackage{verbatim}
\usepackage{upgreek}
\usepackage{mathtools}
\newtheorem{theorem}{Theorem} 	      	      	                              % Theorem environment
\newtheorem{corollary}[theorem]{Corollary}     	      	      	      	      % Corollary environment
\newtheorem{lemma}[theorem]{Lemma}     	       	      	      	      	      % Lemma environment
\newtheorem{proposition}[theorem]{Proposition} 	      	      	      	      % Proposition environment
\newtheorem{definition}[theorem]{Definition} 	      	      	                % Definition environment
\newtheorem{problem}[theorem]{Problem} 	      	      	                      % Problem environment
\newtheorem{remark}[theorem]{Remark}                                          % Remark environment
\newcommand{\thistheoremname}{}
\newtheorem*{theorem*}{Theorem}
\newtheorem*{genericthm*}{\thistheoremname}
\newenvironment{namedthm*}[1]
  {\renewcommand{\thistheoremname}{#1}%
   \begin{genericthm*}}
  {\end{genericthm*}}

\numberwithin{equation}{section}                                              % Equation numbering
\numberwithin{theorem}{section}                                               % Theorem numbering
\numberwithin{figure}{section}                                                % Figure numbering

\newcommand{\mf}[1]{\mathfrak{#1}}                                            % Mathfrak
\newcommand{\mc}[1]{\mathcal{#1}}                                             % Mathcal
\newcommand\numberthis{\addtocounter{equation}{1}\tag{\theequation}}

\newcommand{\N}{\mathbb{N}}                                                   % Natural numbers
\newcommand{\R}{\mathbb{R}}                                                   % Real numbers
\newcommand{\Sph}{\mathbb{S}}                                                 % Unit sphere

\newcommand{\nasla}{\slashed{\nabla}}                                         % Slashed nabla

\allowdisplaybreaks
\begin{document}

\title[Carleman estimate and interior control]{ Carleman estimate for ultrahyperbolic operators\\ and improved interior control for wave equations }

\author{Vaibhav Kumar Jena}
\address{School of Mathematical Sciences\\ Queen Mary University of London\\
London E1 4NS\\ United Kingdom}
\email{v.k.jena@qmul.ac.uk}

\begin{abstract}
In this article, we present a novel Carleman estimate for ultrahyperbolic operators, in \( \mathbb{R}^m_t \times \mathbb{R}^n_x \). Then, we use a special case of this estimate to obtain improved observability results for wave equations with time-dependent lower order terms. The key improvements are: (1) we obtain smaller observation regions when compared with  standard Carleman estimate results, and (2) we also address the case when observability is centered around a point inside the domain. Finally, as a corollary of the observability result, we obtain improved interior controllability for wave equations with same features.
\end{abstract}


\setcounter{tocdepth}{2}

\maketitle

\section{Introduction}

\subsection{Controllability}

The controllability problem for a general PDE is the following: \emph{Is it possible to drive the solution of the system from a given initial state to a final state, using a suitable control?}

There can be different choices of control, each of which gives to a rise a different type of controllability problem. For example
\begin{itemize}
\item \textit{Boundary controllability}: The control acts on the boundary of the domain. These are further divided into different cases, depending on the type of the prescribed boundary data.

\item \textit{Interior controllability}: The control acts on a subset of the domain.
\end{itemize}
In this article, we will address the interior controllability problem for wave equations.
\subsubsection{Wave equation}
Let \(T>0\) and \(n \in \N\). Also, let \(\Omega\) be a bounded and open subset of \(\R^n\). On the domain \( (-T,T) \times \Omega \) we will consider a wave equation of the following type 
\begin{equation}
\label{eq.intro_wave} \begin{cases} - \partial_{tt}^2 y + \Delta_x y - \nabla_\mc{X} y + q y = F \textbf{1}_W \qquad \hspace{0.18cm} \text{in } (-T,T) \times \Omega \text{,} \\
y(-T) =  y_0^- , \quad \partial_t y (-T)=  y_1^-, \quad \qquad \hspace{0.25cm} \text{in } \Omega, \\
 y = 0, \qquad \qquad \qquad \quad \qquad \qquad \qquad \hspace{0.50cm} \text{on } (-T,T) \times \partial \Omega,
\end{cases}
\end{equation}
where \( \mc{X} := \mc{X}(t,x) \) is a vector field, \( q:= q(t,x) \) is the potential, and \( W \) is a non-empty subset of \( (-T,T) \times \Omega \). In the above equation, \( \nabla_\mc{X} y = \mc{X}^t \partial_t y + \mc{X}^x \cdot \nabla_x y \), where \( \mc{X}^t, \mc{X}^x \) are the Cartesian components of \(\mc{X}\). The function \(F\) is known as the control function and it only acts on the subset \(W\). Then, the (exact) controllability problem for the above system can be stated as follows:
\begin{center}
\textit{Given any initial and final data \( (y_0^-, y_1^-),(y_0^+, y_1^+) \in H^1_0(\Omega) \times L^2(\Omega) \) find a control \(F \in L^2(W)\) such that the (weak) solution of \eqref{eq.intro_wave} satisfies}
\end{center}
\begin{equation}\label{eq:controlgoal}
y(T) = y_0^+, \quad \partial_t y(T) = y_1^+.
\end{equation}
We point out a couple more facts about the controllability problem for wave equations.
\begin{itemize}
\item[(a)] If \( (y_0^+, y_1^+) = (0,0)\), then this becomes the null controllability problem. Due to time reversibility of wave equations, we see that null and exact controllability are equivalent. 

\item[(b)] The finite speed of propagation of wave equation implies that the time of control \(T\) should be large enough, so that the effects of the control can reach all points in the domain. 

\end{itemize}

\begin{remark}
For the above two facts, contrast this with the heat equation, which is not time reversible. Thus, null and exact controllability are not equivalent. Also, since heat equation has infinite speed of propagation, there is no such requirement for the time of control. In general, if the heat equation is null controllable for some time \(T\), then it is null controllable for all positive time.
\end{remark}

To solve the described control problem, we will use a technique described by the work of Dolecki-Russell \cite{dolec_russe:obs_control}. This result shows that using a duality argument, controllability is equivalent to obtaining unique continuation properties for the adjoint system of \eqref{eq.intro_wave}. Also, Lions \cite{lionj:control_hum} describes the process to obtain such a control using the \emph{Hilbert Uniqueness Method (HUM)}. By \emph{HUM}, to solve the controllability problem one needs to prove an observability inequality of the type\footnote{The notation \( a \lesssim b\) means that there is some constant \(C>0\) such that \(a \leqslant C b\).}
\begin{equation}
||(\phi_0, \phi_1)||_{L^2(\Omega) \times H^{-1}(\Omega) } \lesssim ||\phi||_{L^2( W )},
\end{equation}
where \(\phi\) is the solution of the adjoint system
\begin{equation}\label{eq.intro_obs}
\begin{cases}
-\partial^2_{tt} \phi + \Delta_x \phi + \nabla_\mathcal{X} \phi + V \phi= 0 , \qquad \text{in}\ (-T,T) \times \Omega,\\
( \phi, \partial_t\phi)|_{t=-T}= ( \phi_0, \phi_1), \qquad \qquad \quad  \hspace{0.07cm}\text{in } \Omega, \\
 \phi = 0 \qquad \qquad \qquad \qquad \qquad \qquad \hspace{0.57cm} \text{on }  (-T,T) \times \partial \Omega,
\end{cases}
\end{equation}
for some potential \(V\). If the observability inequality holds, then we say that the adjoint system is observable. Henceforth, our goal will be to obtain this observability inequality.

\subsubsection{General techniques} 
There are different methods to show that the adjoint system is observable. We discuss some of them briefly here:
\begin{itemize}

\item Multiplier estimates: These are integral estimates obtained by multiplying a suitably chosen multiplier with the adjoint equation and then using integration by parts. See the articles by Ho \cite{ho:obs_wave} and Lions \cite{lionj:ctrlstab_hum} for some of the initial results using multiplier methods.

\item Carleman estimates: These are weighted estimates which can be used for showing unique continuation properties for PDEs. This has been extensively used for solving observability/controllability problem for various PDEs. Some related results for wave equations (or hyperbolic equations, in general) can be found in \cite{Carleman}, \cite{FYZ}, \cite{furs_iman:ctrl_evol}, \cite{Arick}, \cite{lasie_trigg_zhang:wave_global_uc}, \cite{tat:paper}, and \cite{Zhang}. In particular, we will adapt some techniques from the works of Fu-Yong-Zhang \cite{FYZ}, Shao \cite{Arick}, and Zhang \cite{Zhang}.

\item Microlocal analysis: These are used for obtaining unique continuation results. A major work in this direction is the article by Bardos, Lebeau, and Rauch \cite{BLR}, that describes a certain optimal property known as \emph{Geometric Control Condition (GCC)}. For \(\omega \subset \Omega\), we say that \((\omega,T)\) satisfies GCC if \emph{every ray of geometric optics enters \(\omega\) in time less than \(T\)}. Other related works are \cite{burq:control_wave}, \cite{laur_leaut:obs_unif}. These results assume that the coefficients in the PDE are time analytic. 

\end{itemize}
Out of the above methods, Carleman estimates are suitable for solving the observability problem for more general type of PDEs. They can be applied to systems which have lower order terms that are time dependent. Moreover, we do not need to assume any time analyticity for these terms. In contrast, although microlocal methods obtain optimal results in terms of the control/observation region, the coefficients can at best be assumed to be time analytic. As our goal is to obtain controllability for a wider class of PDEs, namely with \( \mc{X}, V\) varying in both time and space, we will use Carleman estimates.

\subsection{Main Results}
In this section we present our main result and also explain some of the major ideas used in the proofs. Let \(x_0 \in \R^n\), and \(\nu\) be the outward unit normal of \(\Omega\). Define \( \Gamma_+ \subset \partial \Omega\) as
\[ \Gamma_+ := \{x \in \partial \Omega\ |\ \nu.(x-x_0) > 0 \}.\]
Let \( \sigma >0 \) and define \(\omega\) as
\[ \omega := \mc{O}_\sigma (\Gamma_+) \cap \Omega, \qquad \mc{O}_\sigma (\Gamma_+) := \{ y \in \R^n : |y-x| < \sigma, \text{ for some } x \in \Gamma_+ \}. \]
The usual observation region for existing Carleman estimate results (and hence control region as well) is given by \( (-T,T) \times \omega\). However, our result improves this particular feature by restricting the observation region to a smaller set. To elaborate on this discussion, let us define the following
\begin{align*}
R := \sup_{ x \in \Omega} |x-x_0|, \qquad \mc{D} := \{ |x-x_0|^2 > t^2 \} \subset \R^{1+n} .
\end{align*}
The region \(\mc{D}\) is the exterior of the null cone centered at the point \((0,x_0)\). Then, we have the following observability result:
\begin{theorem}\label{intro_thm_obs}
Let \( T > R\), and \( x_0 \in \R^n\). Fix \( \mc{X} \in C^\infty \left( [-T,T] \times \bar{\Omega}; \R^{1+n}\right)\) and  \(V  \in C^\infty \left( [-T,T] \times \bar{\Omega}\right) \). Let \( \phi \) be solution of the adjoint system \eqref{eq.intro_obs}. Then, we have the following:
\begin{itemize}
\item If \( x_0 \notin \bar\Omega \), then there exists a constant \( C := C(\Omega, T, V, \mc{X})\) such that the following holds
\begin{equation}\label{eq.intro_thm_obs}
|| \phi_0 ||_{L^2(\Omega)}^2 + || \phi_1 ||_{H^{-1}(\Omega)}^2 \leqslant C \int_{ ((-T,T) \times \omega ) \cap \mc{D}} |\phi(t,x)|^2,
\end{equation} 
for all \( (\phi_0,\phi_1) \in L^2(\Omega) \times H^{-1}(\Omega)\).

\item If \( x_0 \in \bar{\Omega} \), then for any open set \(W\) in \((-T,T) \times \Omega\) satisfying
\[ \overline{((-T,T) \times \omega ) \cap \mc{D}} \subset W \subset (-T,T) \times \Omega, \]
there is a constant \(C := C(\Omega, T, V, \mc{X},W)\), such that the following holds
\begin{equation}\label{eq.intro_thm_obs2}
|| \phi_0 ||_{L^2(\Omega)}^2 + || \phi_1 ||_{H^{-1}(\Omega)}^2 \leqslant C \int_W |\phi(t,x)|^2,
\end{equation}
for all \( (\phi_0,\phi_1) \in L^2(\Omega) \times H^{-1}(\Omega)\).
\end{itemize}
\end{theorem}
\begin{remark}
Note that, we also consider the case when the observation point \(x_0\) lies inside the domain \(\Omega\). This interior observability case has not been addressed by Carleman estimate methods before, the only exception being the work in \cite{Arick}.
\end{remark}

\begin{remark}
From the statement of the above theorem, we can see that in both the cases the observation region is better than those obtained by earlier Carleman based results. When \(x_0 \notin \bar{\Omega}\), our observation region is restricted to the exterior of the null cone. When \(x_0 \in \bar{\Omega}\), we get a neighbourhood \(W,\) but the difference between \( ((-T,T) \times \omega ) \cap \mc{D} \) and \( W \) can be made arbitrarily small.
\end{remark}

From our earlier discussion about the equivalence of observability and controllability, a corollary to the above theorem is the following controllability result. 
\begin{corollary}
Let \(T>R\), and \( x_0 \in \R^n\). Fix \( \mc{X} \in C^\infty \left( [-T,T] \times \bar{\Omega}; \R^{1+n}\right)\) and  \( q  \in C^\infty \left( [-T,T] \times \bar{\Omega}\right) \). Let \(y\) be the solution of system \eqref{eq.intro_wave}. Then, we have the following:
\begin{itemize}
\item If \( x_0 \notin \bar\Omega \), then there exists a function \( F \in L^2((-T,T) \times\omega  )\), such that the control goal \eqref{eq:controlgoal} is satisfied with the control \(F\) acting on the region \(( (-T,T) \times \omega  ) \cap \mc{D}\).

\item If \( x_0 \in \bar\Omega \), then for any open set \(W\) in \((-T,T) \times \Omega\) satisfying
\[ \overline{((-T,T) \times \omega ) \cap \mc{D}} \subset W \subset (-T,T) \times \Omega, \]
there exists a function \( F \in L^2(W)\), such that the control goal \eqref{eq:controlgoal} is satisfied with the control \(F\) acting on the region \(W\).

\end{itemize}

\end{corollary}

\subsubsection{Improvements to established results} 
The key features of this article are the following:
\begin{enumerate}

\item We present a novel Carleman estimate for ultrahyperbolic operators in \( \R^m_t \times \R^n_x \).

\item For the wave equation observability, the observation regions we obtain are smaller than the observation regions obtained from standard Carleman methods.

\item We also prove observability when the observation point is in the interior of the domain. Classical Carleman estimate results do not handle this particular case.

\end{enumerate}

Finally, we point out that, we consider the wave equation with general lower order terms that are time-dependent.

\subsubsection{Proof ideas}

As mentioned earlier we will use Carleman estimates to obtain the above observability result. A Carleman estimate result that is related to our problem is the article by Shao \cite{Arick}, where the author proved a boundary estimate for (\( \R^1_t \times \R^n_x \)) wave operators on a domain with moving boundary. We need to generalise this particular result to ultrahyperbolic operators in \( \R^m_t \times \R^n_x \), for proving Theorem \ref{intro_thm_obs}. Since we will adapt the corresponding Carleman weight to our setting, we point out two of its notable features:
\begin{itemize}
\item The weight function vanishes near the boundary of the null cone.
\item This helps us to localise the observation region to the exterior of the null cone.
\end{itemize}
Now let us see the motivation behind obtaining an estimate for ultrahyperbolic operators. For convenience, we will use \( (m,n) \) to denote the \( \R^m_t \times \R^n_x \) case. From Theorem \ref{intro_thm_obs}, we can see that we want the \(L^2 \times H^{-1} \) energy on the LHS, which can be obtained from a \( L^2 \) Carleman estimate. The usual application of Carleman estimates for proving such an observability estimate involves the following steps:
\begin{itemize}
\item[(a)] Find a \(H^1_0\) Carleman estimate for a more regular quantity, say \(z\).
\item[(b)] Use the \(H^1_0\) estimate to obtain a \(L^2\) estimate for \(\phi\).
\item[(c)] Absorb the non-essential terms and use energy estimates for \(\phi\).
\end{itemize}
The standard process for obtaining step (b) above involves dealing with the reciprocal of the Carleman weight (see, for example, \cite[Section 7]{FYZ}). As we mentioned, the weight we use vanishes near the null cone boundary. Thus, when we take the reciprocal it will blow up at these points which creates an issue. Hence, we take a different approach to overcome this problem. We use a method described by Zhang \cite{Zhang}, where a new function \(z\) is defined as
\begin{equation}\label{eq_intro_z}
z(t_1,t_2,x) := \int_{t_1}^{t_2} \phi(s, x) ds.
\end{equation}
Now, \( z \) is dependent on two time variables, and satisfies an ultrahyperbolic equation (obtained from the wave equation of \( \phi \)). Also note that, \( z \)  has better regularity than \( \phi \). Then, we can apply a suitable \( H^1_0 \) Carleman estimate for \( z \), in \( (2,n) \) setting. Finally, to get back to \(\phi\) we have to integrate out the extra time integral. Thus, it is clear that we must obtain a \( (2,n) \) estimate, that is, a \emph{two time variable} Carleman estimate. The proof of this estimate is geometric, and can be generalised to \(m\)-time dimensions through a clever choice of coordinate system. This is the key motivation for the main \( (m,n) \) estimate in this article. Once the \( (m,n) \) estimate is proved, we just use the special case of \( m=2 \) for the above \( z \) to solve our observability problem.

Now, let us see the idea for the choice of the coordinate system that allows us to prove such a result. The proof of the \((1,n)\) Carleman estimate in \cite{Arick} uses null coordinates in the standard setting of \( \R^{1+n}\):
\[ u = \frac{1}{2} (t - |x|), \quad v = \frac{1}{2} (t + |x|). \]
To generalise this result to \((m,n)\) setting we need to be able to construct null coordinates in \( \R^{m+n}\). Hence, we base our proof on the idea of using polar coordinate representation for the \(m\)-time components which allows us to define null coordinates in \(\R^{m+n}\). This way we can generalise the required wave equation features to ultrahyperbolic equations. 

Another point to mention is that we first obtain a boundary Carleman estimate in \( (m,n) \) setting. Then, for deriving the interior Carleman estimate we use the \emph{hidden regularity} argument adapted to our setting. In this method, the boundary term present in the boundary estimate is estimated by data on an interior subset of the domain.

To go from the Carleman estimate to the observability inequality, we will use energy estimate results. While doing this, we have to switch back from \( z \) to the original function \( \phi\), and one has to be particularly careful with the order of operations. This is because we have a Carleman estimate for \( z \) but the energy estimate is only applicable for \( \phi \). Thus, we have to be very precise while dealing with the extra time integral introduced in \eqref{eq_intro_z} and doing the usual absorption of terms.

\subsection{Future work and comments} We discuss some future problems and directions in this section.

\noindent (1) This article only deals with static domains, that is time independent domains. We did not make any mention of time dependent domains (also called as domains with moving boundary). However, the Carleman estimate presented in \cite{Arick} is actually a result for time dependent domains. Analogously, the \((m,n)\) Carleman estimate of this article also works for time dependent domains.\footnote{Presenting the moving boundary case involves defining some more technical ideas. We do not present that here, as our goal is to obtain controllability on static domains.} However, solving the observability problem in this case is still open. This is mainly because for time dependent domains the definition of \(z\) in \eqref{eq_intro_z} turns out to be
\begin{align*}
z(t_1,t_2,x) := \int_{t_1}^{t_2} \phi(s, x(s)) ds,
\end{align*}
the \(x(s)\) inside the integral represents the fact that the domain is now time dependent. For this case, we get second spatial derivatives of \(\phi\) which causes serious issues. This seems to suggest that this might not be the optimal method for solving the interior controllability in the time dependent case. To obtain controllability for this problem, we need more knowledge of wave equations in general geometries. We leave this problem to a future work. 

\medskip

\noindent (2) On time dependent domains, interior control for wave equations with initial data in the weaker space \( L^2(\Omega) \times H^{-1}(\Omega)\) can be solved by directly using the \(H^1_0\) Carleman estimate for the corresponding adjoint system. For this problem, using \emph{HUM} shows that the control is a distribution and lies in the space \( ( H^1(-T,T; L^2(\omega)) )' \), where \('\) denotes the dual space. This is an upcoming result.

\subsection{Outline}
The rest of the paper is divided as follows:
\begin{itemize}
\item Section \ref{sec Setting} describes the setting of the adjoint system and associated observability problem. We also present the geometric setting for the Carleman estimate here.

\item Section \ref{sec Carl_est} contains the statement and proof of the \((m,n)\) interior Carleman estimate. We first state a boundary Carleman estimate in the \((m,n)\) case which is a precursor to the interior estimate; the proof of the boundary estimate is provided in Section \ref{sec Carl_est} as well.

\item Section \ref{sec Carl_obs} describes the process of using the Carleman estimate to prove the observability result Theorem \ref{intro_thm_obs}. For this purpose, we will use a special case of the interior Carleman estimate, and take \(m=2\).

\end{itemize}

\subsection{Acknowledgements} The author is grateful to his supervisor, Arick Shao, for his guidance and the many helpful discussions on this research work.

\section{Setting}\label{sec Setting}

\subsection{Controllability and Observability problem}
Fix \(n \in \N\), and let \(\Omega \subset \R^n \) be an open and bounded subset, with a smooth boundary. The space-time domain we consider is given by:
\begin{equation}
\mc{U} := \R \times \Omega,
\end{equation}
where the domain for the time component is \(\R\). Let \( \mc{X} \in C^\infty(\bar{\mc{U}}; \R^{1+n})\) be a vector field and \( V \in C^\infty (\bar{\mc{U}}) \). We consider \( \mc{X}\) and \( V \) to be time dependent. On \(\bar{\mc{U}}\), we consider the following wave equation with initial and boundary condition
\begin{equation}\label{eq.wave_obs}
\begin{cases}
-\partial^2_{tt} \phi + \Delta_x \phi + \nabla_\mathcal{X} \phi + V \phi= 0 , \qquad \text{in}\  (-T,T) \times  \Omega ,\\
( \phi, \partial_t\phi)|_{t=-T}= ( \phi_0, \phi_1), \qquad \qquad \quad  \hspace{0.06cm}\text{in } \Omega, \\
 \phi = 0 \qquad \qquad \qquad \qquad \qquad \qquad \hspace{0.56cm} \text{on }  (-T,T) \times \partial \Omega .
\end{cases}
\end{equation}
It can be shown by using usual functional analytic methods that the above system is well posed in \(L^2(\Omega)  \times H^{-1}(\Omega) \) (see \cite{lionj_mage:bvp1}). The solution \(\phi\) is known as the transposition (or ultra-weak) solution.

\noindent The observability problem in this context is the following:
\begin{problem}
Let \(T>0\) be large enough. Let \( \phi\) be the solution  of \eqref{eq.wave_obs} obtained by the transposition method. Does there exist a subdomain \( W \subset (-T,T) \times \Omega \), such that for some constant \( C:=C(\mc{X}, V)>0\) we have
\begin{equation} \label{eq.observe_ineq}
|| \phi_0 ||_{L^2(\Omega)}^2 + || \phi_1 ||_{H^{-1}(\Omega)} \leqslant C \int_{ W } |\phi|^2 dx dt, \quad \forall\ (\phi_0, \phi_1) \in L^2(\Omega) \times H^{-1}(\Omega) \text{?}
\end{equation}
\end{problem}
As already discussed in the introduction, the above observability problem is equivalent to solving the controllability problem for the following system:
\begin{equation}
\begin{cases}
-\partial^2_{tt} y + \Delta_x y - \nabla_\mathcal{X} y + \left( V - \nabla_\alpha \mc{X}^\alpha \right) y= F \textbf{1}_W , \qquad \text{in}\  (-T,T) \times \Omega ,\\
( y,\partial_t y)|_{t=-T}= ( y_0, y_1), \qquad \qquad \qquad \quad \hspace{0.04cm} \text{in } \Omega,\\
 y = 0, \qquad \qquad \qquad \qquad \qquad \qquad \qquad \hspace{0.31cm} \text{on }  (-T,T) \times  \partial \Omega  ,
\end{cases}
\end{equation}
with initial data \( (y_0, y_1) \in H^1_0(\Omega) \times L^2(\Omega) \) and the control function \( F \in L^2(W) \).

\subsection{Geometry for the Carleman estimate} As we briefly mentioned in the discussion centred around equation \eqref{eq_intro_z}, to solve the observability problem one needs to obtain a \emph{two time variable} Carleman estimate. This idea can be generalised to the concept of a \(m\)-time variable Carleman estimate. That is, we can derive a Carleman estimate for operators acting on a space-time domain with \( (t,x) \in \R^m \times \R^n \). In this section, we will present the geometric setting in which our main \(\R^m \times \R^n\) Carleman estimate works. 

\begin{definition}\label{def_setting}
Fix \(m,n\in \N\). On \( \R^{m+n} \), we define the following:
\begin{itemize}

\item Let \( t:=(t_1, \ldots, t_m) \) and \( x = (x_1, \ldots,  x_n)\) denote the Cartesian coordinates on \(   \R^{m+n}  \).

\item Let \( g\) denote the flat pseudo-Riemannian metric on \( ( \R^{m+n}, g) \):
\begin{equation}\label{eq.gmetric}
g = -dt_1^2 - \cdots - dt_m^2 + dx_1^2 + \cdots + dx_n^2.
\end{equation}

\item Let \(r := |x| \) and \( \tau := |t|\) denote the spatial and temporal radial functions, respectively.

\item The null coordinates \( u\) and \(v\) are given by:
\begin{equation} \label{eq.uv}
u:= \frac{1}{2}(\tau- r), \qquad v:= \frac{1}{2} (\tau + r).
\end{equation}

\item Define the function \(f\) as:
\begin{equation}\label{eq.f}
 f= -uv = \frac{1}{4}(r^2 - \tau^2) = \frac{1}{4}\left(|x|^2-|t|^2\right) .
\end{equation}

\end{itemize}

\end{definition}

\begin{remark}\label{rmk.t_polar}
A key idea of the work presented here is the following: to deal with the \(m\)-time coordinates, we represent the time components in polar coordinates as well. This allows us to define the notion of null coordinates \( (u,v) \) in our setting.\footnote{Compare this with the case of Lorentzian geometry, where the usual null coordinates are better suited for some calculations.} This also helps us to analogously define the function \(f\) to be used later in the Carleman weight.

Moreover, using polar coordinate representation for \(t\) allows us to treat the spatial and temporal components at the same level, in some sense. 
\end{remark}

On \( \{ \tau \neq 0, r \neq 0 \}\), standard polar coordinates are \( (\tau, r, \omega_x, \omega_t) \) and null coordinates are \(( u, v, \omega_x, \omega_t )\), where \(\omega_x\) is the spatial angular coordinate with values in \(\mathbb{S}^{n-1}\) and \(\omega_t\) is the temporal angular coordinate with values in \(\mathbb{S}^{m-1}\). Moreover, let \( \partial_\tau, \partial_r, \partial_u, \partial_v \) denote the coordinate vector fields with respect to these coordinate systems. We can write the metric \(g\) in terms of polar and null coordinates as
\begin{equation}\label{eq.psr_met}
g = -d\tau^2 + dr^2 + r^2 \mathring{\gamma}_{\mathbb{S}^{n-1}}- \tau^2 \mathring{\gamma}_{\mathbb{S}^{m-1}}  = -4dudv + r^2 \mathring{\gamma}_{\mathbb{S}^{n-1}} - \tau^2 \mathring{\gamma}_{\mathbb{S}^{m-1}},
\end{equation}
where \( \mathring{\gamma}_{\mathbb{S}^{n-1}} \) (and \( \mathring{\gamma}_{\mathbb{S}^{m-1}} \) ) denotes the unit round metric on \( \mathbb{S}^{n-1} \) (and \(  \mathbb{S}^{m-1} \)).

Now we provide some definitions that are crucial for presenting the Carleman estimate.
\begin{definition}\label{def_coord}
We use the following conventions-
\begin{itemize}
\item \( (\alpha, \beta, \ldots) :\) Lower case Greek letters (ranging from 1 to \(m+n\)) denote spacetime components in \(\R^{m+n}\).
\item \( (a,b,\ldots) :\) Lower case Latin letters (ranging from 1 to \(n-1\)) denote spatial angular components corresponding to \(\omega_x \in \mathbb{S}^{n-1} \) in any of the above mentioned coordinate systems.
\item \( (A,B,\ldots) :\) Upper case Latin letters (ranging from 1 to \(m-1\)) denote temporal angular components corresponding to \(\omega_t \in \mathbb{S}^{m-1} \) in any of the above mentioned coordinate systems.
\end{itemize}
\end{definition}

\begin{definition}
We use the following notations for denoting objects corresponding to \(g\):
\begin{itemize}
\item \( \nabla \) denotes the Levi-Civita connection with respect to \(g\).
\item \(  \square  := g^{\alpha\beta}  \nabla_{\alpha\beta} \) denotes the wave operator with respect to  \(g\).
\item \(  \nasla  \) denotes the derivatives in the spatial angular components with respect to \(g\).
\item \( \tilde{\nabla} \) denotes the derivatives in the temporal angular components with respect to \(g\).
\end{itemize}
\end{definition}

\begin{remark}
We will use the Einstein summation notation here. That is, repeated indices in the subscript and superscript indicate summation over all possible indices. For instance, for the \( \square \) operator, we have
\begin{align*}
g^{\alpha \beta} \nabla_{\alpha\beta} & = \sum_{\mu, \lambda \in \{ 1, 2, \ldots, m+n \} } g^{\mu \lambda} \nabla_{\mu \lambda} \\
& = g^{t_1t_1} \nabla_{t_1 t_1} + \ldots + g^{t_m t_m} \nabla_{t_m t_m} + g^{x_1 x_1} \nabla_{x_1 x_1} + \ldots + g^{x_n x_n} \nabla_{x_n x_n} \\
& = - \nabla_{t_1 t_1} - \ldots - \nabla_{t_m t_m} + \nabla_{x_1 x_1} + \ldots + \nabla_{x_n x_n} ,
\end{align*}
where the other terms vanish from the summation because
\begin{align*}
g^{t_i t_j} & = 0, \quad \text{for } i \neq j, \\
g^{x_k x_l} & = 0, \quad \text{for } k \neq l, \\
g^{t_p x_q} & = 0, \quad \text{for any } p,q .
\end{align*}
\end{remark}

\noindent From \eqref{eq.f}, we can see that the region \( \{f=0\}\) is the null cone centred at the origin. With this idea, we make the following definition.
\begin{definition}
Define the region \(\mf{D} \subset \R^{m+n}\) as
\begin{align}
\mf{D} := \{ f > 0\} \label{eq.D} \text{.}
\end{align}
\end{definition}
\begin{remark} Note that, \(\mf{D}\) is the region exterior to the null cone centred at \( (0,0) \in \R^{m+n} \). That is, it denotes the region 
\[ \mf{D} = \{ (t,x) \in \R^{m+n} : |x|^2 > |t|^2\}.\]
The improved observability inequality that we will derive later, is obtained by working with this \( \mf{D} \). 
\end{remark}

\begin{remark}\label{rmk.uvfbound}
On the region $\mf{D} := \{f>0\} $, the following are satisfied:
\begin{equation}
\label{eq.uvf_bound1} 0 < - u < r \text{,} \qquad 0 < v < r \text{,} \qquad 0 < f < r^2 \text{.}
\end{equation}
This is clear from the definitions of the corresponding functions, see \eqref{eq.uv} and \eqref{eq.f}.
\end{remark}

\noindent We also recall the following definition from geometry. 

\begin{definition}
A vector \(X\) is said to be
\begin{enumerate}
\item spacelike, if \( g(X,X) > 0 \);

\item null, if \( g(X,X) = 0 \);

\item timelike, if \( g(X,X) < 0 \).
\end{enumerate}
Further, a hypersurface \( \mc{H} \) is said to be
\begin{enumerate}

\item spacelike, if the normal at each point \( p \in \mc{H} \) is timelike;

\item null, if the normal at each point \( p \in \mc{H} \) is null;

\item timelike, if the normal at each point \( p \in \mc{H} \) is spacelike.

\end{enumerate}

\end{definition}

\section{Carleman estimate}\label{sec Carl_est}
\noindent We have the following estimate for ultrahyperbolic operators in \(\R^{m+n}\).
\begin{theorem}[Boundary Carleman Estimate]\label{thm.carl_bdry}
Let  \(\mf{U} \subset \R^{m+n}\) be such that for some \( R>0 \)
\begin{equation}\label{eq.carleman_domain}
\mf{U} \cap \mf{D} \subseteq \{ r<R \}.
\end{equation}
Also, assume that the boundary of \(\mf{U}\), denoted by \( \partial \mf{U} \), is smooth and timelike. Let \( \varepsilon, a, b >0 \) be constants such that:
\begin{equation}\label{eq.carleman_choices}
a \geqslant (m+n)^2, \qquad \varepsilon \ll_{m,n} b \ll R^{-1}.
\end{equation}
Then, there exists \( C >0 \) such that for any \(z \in \mc{C}^2({\mf{U}})\cap \mc{C}^1(\bar{\mf{U}}) \) with
\begin{equation}\label{eq.carleman_dirichlet}
z|_{\partial\mf{U} \cap \mf{D}} = 0,
\end{equation}
we have 
\begin{align}\label{eq.carleman_est}
\notag C\varepsilon \int_{\mf{U}\cap \mf{D}}\zeta_{a,b;\varepsilon} & r^{-1}(|u \partial_u z|^2 + |v \partial_v z|^2 + f g^{ab}\slashed\nabla_a z \slashed\nabla_b z  - f g^{CD} \tilde{\nabla}_C z  \tilde{\nabla}_D z ) + Cba^2\int_{\mf{U}\cap \mf{D}}\zeta_{a,b;\varepsilon} f^{-\frac{1}{2}} z^2 \\
& \leqslant \frac{1}{a}\int_{\mf{U}\cap \mf{D}} \zeta_{a,b;\varepsilon} f |\square z|^2 +  C' \int_{\partial\mf{U}\cap \mf{D}}\zeta_{a,b;\varepsilon} [( 1 - \varepsilon r ) \mc{N} f + \varepsilon f \mc{N} r ] |\mathcal{N} z|^2, 
\end{align}
where \( \zeta_{a,b;\varepsilon} \) is the Carleman weight defined as
\begin{equation}
\label{eq.carleman_weight} \zeta_{ a, b; \varepsilon } := \left\{ \frac{ f }{ ( 1 + \varepsilon u ) ( 1 - \varepsilon v ) } \cdot \exp \left[ \frac{ 2 b f^\frac{1}{2} }{ ( 1 - \varepsilon u )^\frac{1}{2} ( 1 + \varepsilon v )^\frac{1}{2} } \right] \right\}^{2a},
\end{equation}
and \(\mc{N}\) is the outer-pointing unit normal of \(\mf{U}\) (with respect to \(g\)).
\end{theorem}

\begin{remark}
The action of \( \mc{N} \) on \( f \) is as follows
\[ \mc{N} f = \nabla_\mc{N} f = \mc{N}^\alpha \nabla_\alpha f . \]
\end{remark}

\begin{remark}
Due to \eqref{eq.psr_met}, \(g^{CD} \tilde{\nabla}_C z  \tilde{\nabla}_D z  \leq 0 \). Hence, the corresponding term on the LHS of \eqref{eq.carleman_est} has a good sign.
\end{remark}

The proof of the above theorem is analogous to the one presented in \cite{Arick}, the main difference is that now we have \(m\) time variables here. The key idea that we use, as mentioned in Remark \ref{rmk.t_polar}, is to use polar coordinates for the \(t\)-components as well. In section \ref{sec_mnCarl}, we present the notion of a warped geometry, which is the main idea behind proving the boundary Carleman estimate in Theorem \ref{thm.carl_bdry}. In section \ref{sec App_BCE}, we present a warped Carleman estimate result, which is then used to prove Theorem \ref{thm.carl_bdry}.

\subsection{Warped Carleman Estimate}\label{sec_mnCarl}
First, we define a warped geometry, and present some properties and calculations related to it. Then, we will prove that the quantities in the warped geometry are conformally related to the corresponding quantities in the original geometry.

\begin{definition}
Fix a constant \( \varepsilon \in \R\). This constant will be called the warping factor.

\end{definition}

\begin{definition}
Define the \(\varepsilon\)-warped metric on \( \R^{m+n}\setminus \{r=0\} \) as follows
\begin{equation}\label{eq.g_wp}
\bar{g} := -4 du dv + \bar{\rho}^2 \mathring{\gamma}_{\mathbb{S}^{n-1}} - \tau^2 \mathring{\gamma}_{\mathbb{S}^{m-1}}  ,
\end{equation}
where \(\bar{\rho}\) is the warped radius given by
\begin{equation}\label{eq.rho_wp}
\bar{\rho}:= r + 2 \varepsilon f,
\end{equation}
where \( f\) is defined by \eqref{eq.f}.
\end{definition}

In the original geometry, the level sets of \(f\) barely fail to be pseudoconvex. Hence, to derive the Carleman estimate in the original geometry we have to work with a different function, say \( f_* \). Using the standard technique to define such a \( f_* \) forces the \( f_* \) to be not well suited for the geometry, and this makes the computations significantly difficult. To solve this issue we define the warped geometry given by \eqref{eq.g_wp}-\eqref{eq.rho_wp}, which allows us to achieve the pseudoconvexity condition. In the warped geometry, the positive level sets of \( f \) are pseudoconvex (with respect to \( \bar{g} \)) in the sense of H\"ormander. Finally, note that pseudoconvexity is a conformally invariant property. Hence, we can go back to the original geometry using a suitable map and obtain pseudoconvexity in the original geometry, albeit with a modification of \( f \) (see \eqref{eq.carleman_weight}).

\begin{remark}
The polar coordinates are well defined only when \( \tau \neq 0\). But, the metric \(\bar{g}\) can be extended smoothly to include the \(\{\tau=0\}\) region as well.
\end{remark}

\begin{remark}
The corresponding result presented in \cite{Arick}, for the \( \R^{1+n}\) case, defines the warped metric as \( \bar{g} := -4 du dv + \bar{\rho}^2 \mathring{\gamma}_{\mathbb{S}^{n-1}}\). In our case, to deal with the extra time dimensions, we will use polar coordinates for the time components as well. This allows to define null coordinates on \(\R^m\times\R^n\).
\end{remark}

\begin{remark}
When \(\varepsilon = 0\), we get the original flat metric \(g\) as given by \eqref{eq.gmetric}.
\end{remark}
We use the following notations for denoting objects corresponding to \(\bar{g}\). 
\begin{itemize}
\item \(\bar{\nabla} \) denotes the Levi-Civita connection with respect to \(\bar{g}\).
\item \( \bar{\square} := \bar{g}^{\alpha\beta} \bar{\nabla}_{\alpha\beta} \) denotes the wave operator with respect to  \(\bar{g}\).
\item \( \bar{\nasla} \) denotes the derivatives in the spatial angular components with respect to \(\bar{g}\).
\item \( \hat{\nabla} \) denotes the derivatives in the temporal angular components with respect to \(\bar{g}\).
\end{itemize}

\noindent We will use the notations mentioned earlier in Definition \ref{def_coord}. 

\begin{lemma}\label{lemma_Christ}
For \(\varepsilon \in \mathbb{R}\) and \(\bar{g}\) as defined above, we have
\begin{itemize}
\item Non zero components of \(\bar{g}\) in the coordinates \((u,v,\omega_x,\omega_t)\) are
\begin{align}
\bar{g}_{uv}= -2 \text{,} & \qquad \bar{g}_{ab} = \bar{\rho}^2 \mathring{\gamma}_{ab}\text{,} \qquad \bar{g}_{CD} = \tau^2 \mathring{\gamma}_{CD}\text{,}  \\
\notag \bar{g}^{uv}= -\frac{1}{2}\text{,} & \qquad \bar{g}^{ab} = \bar{\rho}^{-2} \mathring{\gamma}^{ab}\text{,} \quad \bar{g}^{CD} = \tau^{-2} \mathring{\gamma}^{CD}\text{.}
\end{align}
\item The non-zero Christoffel symbols are
\begin{align}
\label{eql.Gamma_comp_wp} \bar{\Gamma}^u_{ab} = \frac{1}{ 2 \bar{\rho} } ( 1 - 2 \varepsilon u ) \bar{g}_{ab} \text{,} &\qquad \bar{\Gamma}^v_{ab} = - \frac{1}{ 2 \bar{\rho} } ( 1 + 2 \varepsilon v ) \bar{g}_{ab} \text{,} \\
\notag \bar{\Gamma}^a_{ub} = - \frac{1}{ \bar{\rho} }  ( 1 + 2 \varepsilon v ) \delta^a_b \text{,} &\qquad \bar{\Gamma}^a_{vb} = \frac{1}{ \bar{\rho} } ( 1 - 2 \varepsilon u ) \delta^a_b, \\
\notag \bar{\Gamma}^u_{CD} = \frac{1}{2 \tau} \bar{g}_{CD}, & \qquad  \bar{\Gamma}^v_{CD} = \frac{1}{2 \tau} \bar{g}_{CD} ,\\
\notag \bar{\Gamma}^C_{uD} = \frac{1}{\tau} \delta_C^D, & \qquad \bar{\Gamma}^C_{vD} = \frac{1}{\tau} \delta_C^D.
\end{align}
\end{itemize}

\end{lemma}

\noindent The proof of the above lemma involves elementary computations, hence we do not present the proof here. Now, we will present some propositions that will be used for deriving a Carleman estimate in this warped geometry. First, we will look at some properties of the function \(f\).
\begin{proposition}\label{thm.f_deriv_wp}
The function \(f\) satisfies the following:
\begin{itemize}
\item The \(\bar{g}\) gradient of \(f\) satisfies
\begin{equation}\label{eq.f_grad_wp}
\bar{\nabla}^\sharp f = \frac{1}{2}(u \partial u + v \partial v),\quad \bar\nabla^\alpha f \bar\nabla_\alpha f =f.
\end{equation}
\item The nonzero components of \(\bar\nabla^2 f\) are
\begin{equation}
\label{eq.f_hess_wp}\bar{\nabla}_{uv}f=-1, \quad \bar{\nabla}_{ab}f= \left(\frac{1}{2}+ \frac{\varepsilon f}{\bar\rho}\right)\bar{g}_{ab}, \quad \bar{\nabla}_{CD}f= \frac{1}{2} \bar{g}_{CD}.
\end{equation}
\item Moreover, we also have
\begin{equation}
\label{eq.f_box_wp} \bar\square f = \frac{n+m}{2} + \frac{(n-1)\varepsilon f}{\bar\rho}, \quad \bar{\nabla}^\alpha f \bar{\nabla}^\beta f \bar{\nabla}_{\alpha\beta}f = \frac{1}{2}f.
\end{equation}
\end{itemize}
\end{proposition}

\begin{proof}
For proving \eqref{eq.f_grad_wp}, we have
\begin{equation}
\notag \bar{\nabla}^\sharp f =  \bar{g}^{\alpha \beta} \partial_\alpha f \partial_\beta = \bar{g}^{vu} \partial_v f \partial_u + \bar{g}^{uv} \partial_u f \partial_v = \frac{1}{2} ( u \partial_u + v \partial_v ).
\end{equation} and 
\begin{align*}
\notag \bar{\nabla}^\alpha f \bar{\nabla}_\beta f = \bar{g}^{\alpha\beta}\partial_\alpha f \partial_\beta f &=  - (-v) (-u) = f.
\end{align*}
The first equation in \eqref{eq.f_hess_wp} holds because \( \bar{\Gamma}^\alpha_{uv} =0 \). For the second equation, by Lemma \ref{lemma_Christ}, we get
\begin{align*}
\bar{\nabla}_{ab} f &= - \bar{\Gamma}^u_{ab} \partial_u f - \bar{\Gamma}^v_{ab} \partial_v f = \left( \frac{1}{2} + \frac{ \varepsilon f }{ \bar{\rho} } \right) \bar{g}_{ab}.
\end{align*}
For the third equation,
\begin{align*}
\bar{\nabla}_{CD} f &= - \bar{\Gamma}^u_{CD} \partial_u f - \bar{\Gamma}^v_{CD} \partial_v f = \frac{v}{2\tau} \bar{g}_{CD} + \frac{u}{2 \tau } \bar{g}_{CD} = \frac{1}{2}\bar{g}_{CD}.
\end{align*}
To prove \eqref{eq.f_box_wp},  we calculate
\begin{align*}
\bar\square f & = 2\bar{g}^{uv}\bar\nabla_{uv}f+ \bar{g}^{ab}\bar\nabla_{ab}f + \bar{g}^{CD}\bar\nabla_{CD}f \\
& = 1 + (n-1)\left( \frac{1}{2} + \frac{\varepsilon f}{\bar\rho} \right) + \frac{1}{2} (m-1) \\
& = \frac{n+m}{2} + \frac{(n-1)\varepsilon f}{\bar{\rho}},
\end{align*}
and for the last equation, we see that
\[
\bar{\nabla}^\alpha f \bar{\nabla}^\beta f \bar{\nabla}_{\alpha \beta} f = \frac{1}{2} \bar{\nabla}^\alpha f \bar{\nabla}_\alpha ( \bar{\nabla}^\beta f \bar{\nabla}_\beta f ) = \frac{1}{2} \bar{\nabla}^\alpha f \bar{\nabla}_\alpha f = \frac{1}{2} f \text{.} \qedhere
\]
\end{proof}

Now we present some properties of the function \( f \bar{\rho}^{-1} \), which will be used later for obtaining the pseudoconvexity condition for level sets of \(f\).
\begin{proposition}\label{thm.f_rho_wp}
We have the following wave equation:
\begin{equation}
\label{eq.f_rho_box_wp} \bar{\Box} \left( \frac{f}{ \bar{\rho} } \right) = -\frac{(n-3)}{\bar{\rho}^3}f + \frac{(n+m-2)r}{2\bar{\rho}^2}.
\end{equation}
\end{proposition}

\begin{proof}
The (null coordinate) derivatives are as follows:
\begin{equation}
\label{eq.f_rho_deriv_wp} \partial_u \left( \frac{f}{ \bar{\rho} } \right) = -\frac{v^2}{\bar{\rho}^2}, \quad \partial_v \left( \frac{f}{ \bar{\rho} } \right) = \frac{u^2}{\bar{\rho}^2}.
\end{equation}
The second derivative is
\[
\partial_u \partial_v \left( \frac{f}{ \bar{\rho} } \right) = \frac{ 2 u ( v - u + 2 \varepsilon f ) }{ \bar{\rho}^3 } + \frac{ 2 u^2 ( 1 + 2 \varepsilon v ) }{ \bar{\rho}^3 } = \frac{ 2 u v }{ \bar{\rho}^3 } = - \frac{ 2 f }{ \bar{\rho}^3 } \text{.}
\]
For the angular components, we have
\begin{align*}
\bar{g}^{ab} \bar{\nabla}_{ab} \left( \frac{f}{ \bar{\rho} } \right) = - \bar{g}^{ab} \left[ \bar{\Gamma}^u_{ab} \partial_u \left( \frac{f}{ \bar{\rho} } \right) + \bar{\Gamma}^v_{ab} \partial_v \left( \frac{f}{ \bar{\rho} } \right) \right] = \frac{ n - 1 }{ 2 \bar{\rho}^3 } ( \bar{\rho}^2 - 2 f - 2 \varepsilon f \bar{\rho} ),
\end{align*}
and that
\begin{align*}
\bar{g}^{CD} \bar{\nabla}_{CD} \left( \frac{f}{ \bar{\rho} } \right) &= - \bar{g}^{CD} \left[ \bar{\Gamma}^u_{CD} \partial_u \left( \frac{f}{ \bar{\rho} } \right) + \bar{\Gamma}^v_{CD} \partial_v \left( \frac{f}{ \bar{\rho} } \right) \right] \\
& = \frac{(m-1)}{2\tau} \left( \frac{v^2}{\bar{\rho}^2} - \frac{u^2}{\bar{\rho}^2} \right) \\
& = \frac{(m-1)(v+u) (v-u) }{2 \tau \bar{\rho}^2} \\
& = \frac{(m-1)r}{2 \bar{\rho}^2}.
\end{align*}
Now we can calculate
\begin{align*}
\bar\square\bigg( \frac{f}{\bar{\rho}}\bigg) & = - \partial_u \partial_v \left( \frac{f}{ \bar{\rho} } \right) + \bar{g}^{ab} \bar{\nabla}_{ab} \left( \frac{f}{ \bar{\rho} } \right) + \bar{g}^{CD} \bar{\nabla}_{CD} \left( \frac{f}{ \bar{\rho} } \right) = -\frac{(n-3)}{\bar{\rho}^3}f + \frac{(n+m-2) r}{2\bar{\rho}^2}\text{.} \qedhere
\end{align*}
\end{proof}
\noindent The next proposition gives us bounds for the coordinate functions \(u,v\) and the function \(f\). 
\begin{proposition} \label{thm.uvf_bound}
On the region $\mf{D} := \{f>0\} $, the following are satisfied:
\begin{equation}
\label{eq.uvf_bound} 0 < - u < r \text{,} \qquad 0 < v < r \text{,} \qquad 0 < f < r^2 \text{.}
\end{equation}
If $\varepsilon \geq 0$, then we also have the following inequality on $\mf{D}$:
\begin{equation}
\label{eq.rho_est_wp} f^\frac{1}{2} < \bar{\rho} \text{.}
\end{equation}
\end{proposition}
\begin{proof}
Using \eqref{eq.uv}, \eqref{eq.f}, and \eqref{eq.D} shows that \eqref{eq.uvf_bound} is true. Next, due to \eqref{eq.rho_wp} 
\begin{align*}
f < r^2 = (\bar{\rho} - 2 \varepsilon f )^2.
\end{align*}
Since \( \varepsilon \geq 0\), and \( f > 0 \) on \(\mf{D}\), this proves \eqref{eq.rho_est_wp}.
\end{proof}

\begin{definition}
On \(\mf{D}\), define vector fields as follows
\begin{align}
\label{eq.TN} T := \frac{1}{2} f^{-\frac{1}{2}} ( - u \partial_u + v \partial_v ) \text{,} \qquad N := \frac{1}{2} f^{-\frac{1}{2}} ( u \partial_u + v \partial_v ).
\end{align}
\end{definition}

\begin{proposition}\label{thm.TN_wp}
\(N\) is everywhere normal to the level sets of \(f\), and \(T\) is everywhere tangent to the level sets of \(f\). We also have
\begin{equation}
\label{eq.f_TN_wp} \bar{\nabla}_{T a} f = \bar{\nabla}_{N a} f = \bar{\nabla}_{T N} f \equiv 0 \text{,} \qquad \bar{\nabla}_{T T} f = - \frac{1}{2} \text{,} \qquad \bar{\nabla}_{N N} f = \frac{1}{2} \text{.}
\end{equation}
\end{proposition}
\begin{proof}
We see that $N$ is normal to the level sets of $f$ because
\[
N = f^{ - \frac{1}{2} } \bar{\nabla}^\sharp f \text{.}
\]
Also, $T$ is tangent to the level sets of $f$ because
\[
T f = \frac{1}{2} f^{ - \frac{1}{2} } [ (-u) (-v) + v (-u) ] = 0 \text{.}
\]
Now, since $\bar{\nabla}_{u a} f = \bar{\nabla}_{v a} f = 0$ by \eqref{eq.f_hess_wp}, then
\(
\bar{\nabla}_{T a} f = \bar{\nabla}_{N a} f = 0 \text{.}
\)
We also have
\[ \bar{\nabla}_{T N} f = \frac{1}{4} f^{-1} [ (-u) v \bar{\nabla}_{u v} f + u v \bar{\nabla}_{u v} f ] = 0. \]
To check the other terms of $\bar{\nabla}^2 f$, we apply \eqref{eq.f_hess_wp} and \eqref{eq.TN}, to get
\begin{align*}
\bar{\nabla}_{T T} f &= \frac{1}{4} f^{-1} [ 2 (-u) v \bar{\nabla}_{u v} f ] = - \frac{1}{2} \text{,} \qquad \bar{\nabla}_{N N} f = \frac{1}{4} f^{-1} [ 2 u v \bar{\nabla}_{u v} f ] = \frac{1}{2} \text{,}
\end{align*}
which proves \eqref{eq.f_TN_wp}.
\end{proof}

\begin{definition}\label{def.pseudoconvex_wp}
Define the modified deformation tensor as follows
\begin{equation}\label{eq.pi_wp}
\bar\pi:= \bar{\nabla}^2f- \bar{h}\cdot\bar{g}, \qquad \bar{h}:=\frac{1}{2}+ \frac{\varepsilon f}{2\bar\rho}.
\end{equation}
Also, define the function \(\bar{w}\) as follows:
\begin{equation}\label{eq.w_wp}
\bar{w}:= \frac{1}{2}\bar{\square}f - \bar{h} = \frac{(n+m-2)}{4} + \frac{(n-2)\varepsilon f}{2\bar{\rho}}.
\end{equation}
\end{definition}

\begin{proposition} \label{thm.pseudoconvex_wp}
The non-zero components of $\bar{\pi}$ are
\begin{equation}
\label{eq.pseudoconvex_wp} \bar{\pi}_{T T} = \frac{ \varepsilon f }{ 2 \bar{\rho} } \text{,} \quad \bar{\pi}_{a b} = \frac{ \varepsilon f }{ 2 \bar{\rho} } \cdot \bar{g}_{a b}, \quad \bar{\pi}_{CD} = -\frac{ \varepsilon f }{ 2 \bar{\rho} } \cdot \bar{g}_{CD} \text{,} \quad \bar{\pi}_{N N} = - \frac{ \varepsilon f }{ 2 \bar{\rho} } \text{.}
\end{equation}
We also have the following wave equation for the quantity \(\bar{w}\)
\begin{equation}
\label{eq.w_box_wp} \bar{\Box} \bar{w} = - \frac{ (n - 2) \varepsilon }{ 2 \bar{\rho} } \left[ \frac{ (n - 3) f }{ \bar{\rho}^2 } - \frac{ (n + m - 2)r}{ 2 \bar{\rho} }  \right] \text{.}
\end{equation}
\end{proposition}

\begin{proof}
We compute the following:
\[ \bar{\pi}_{CD} = \bar{\nabla}_{CD} f - \bar{h} \cdot \bar{g}_{CD} = \left( \frac{1}{2} - \frac{1}{2} - \frac{ \varepsilon f }{ 2 \bar{\rho} } \right) \bar{g}_{CD} = - \frac{ \varepsilon f }{ 2 \bar{\rho} } \cdot \bar{g}_{CD} .\]
Similarly, the rest of the equations in \eqref{eq.pseudoconvex_wp} can be checked 
\begin{align*}
\bar{\pi}_{T T} &= \bar{\nabla}_{T T} f + \bar{h} = - \frac{1}{2} + \frac{1}{2} + \frac{ \varepsilon f }{ 2 \bar{\rho} } = \frac{ \varepsilon f }{ 2 \bar{\rho} } \text{,} \\
\bar{\pi}_{N N} &= \bar{\nabla}_{N N} f - \bar{h} = \frac{1}{2} - \frac{1}{2} - \frac{ \varepsilon f }{ 2 \bar{\rho} } = - \frac{ \varepsilon f }{ 2 \bar{\rho} } \text{,} \\
\bar{\pi}_{a b} &= \bar{\nabla}_{a b} f - \bar{h} \cdot \bar{g}_{a b} = \left( \frac{1}{2} + \frac{ \varepsilon f }{ \bar{\rho} } - \frac{1}{2} - \frac{ \varepsilon f }{ 2 \bar{\rho} } \right) \bar{g}_{ab} = \frac{ \varepsilon f }{ 2 \bar{\rho} } \cdot \bar{g}_{ab} \text{,} 
\end{align*}
The other components are zero because of \eqref{eq.f_TN_wp} and the fact that \( \bar{g}(T,a) = \bar{g}(N,a) = \bar{g}(T,N) = 0 \).
Moreover, for \( \bar{w} \)
\[ \bar{\Box} \bar{w} = \frac{(n-2)}{2}\bar\square\bigg( \frac{f}{\bar{\rho}}\bigg),\]
then using \eqref{eq.f_rho_box_wp} completes the proof.
\end{proof}

\noindent Note that for \(\varepsilon > 0 \), we have \(\bar{\pi}_{T T}, \bar{\pi}_{a b} > 0\). This property of the tensor \( \bar{\pi} \) gives us the pseudoconvexity condition for the function \(f\). 

Now we show that the warped geometry is conformally isometric to the (original) pseudo-Riemannian geometry. This is helpful because it allows us to switch between the two geometries using the defined isometry.

\begin{definition} \label{def.conformal}
Let $R > 0$ and choose $\varepsilon \in \R$ such that $| \varepsilon | \ll_{m,n} R^{-1}$. Let \(\xi\) be the function defined as
\begin{equation}
\label{eq.conformal_factor} \xi := ( 1 + \varepsilon u ) ( 1 - \varepsilon v ) \text{.}
\end{equation}
Also, define the map
\[
\bar{\Phi}: \mf{D} \cap \{ r < R \} \rightarrow \mf{D} \text{,}
\]
in terms of null coordinates as follows:
\begin{align}
\label{eq.conformal} \bar{\Phi} ( u, v, \omega_x, \omega_t) &:= ( \bar{u} ( u, v, \omega_x, \omega_t ), \bar{v} ( u, v, \omega_x, \omega_t ), \bar{\omega}_x ( u, v, \omega_x, \omega_t ), \bar{\omega}_t (  u, v, \omega_x, \omega_t ) ) \\
\notag &:= ( u ( 1 + \varepsilon u )^{-1}, v ( 1 - \varepsilon v )^{-1}, \omega_x, \omega_t ) \text{.}
\end{align}
\end{definition}

\begin{remark}\label{rmk_Phi}
Note that, the map \(\bar{\Phi}\) only affects the null coordinates. It leaves the (spatial and temporal) angular coordinates unchanged.
\end{remark}

\begin{remark}\label{rmk_xirf}
We can express \(\xi\) in terms of \(r, f\) as follows:
\[ \xi = ( 1 + \varepsilon u ) ( 1 - \varepsilon v ) = 1 - \varepsilon v + \varepsilon u - \varepsilon^2 uv = 1 - \varepsilon r + \varepsilon^2 f. \]
\end{remark}

\begin{proposition} \label{thm.met_conf}
In the setting of Definition \ref{def.conformal}, we have:
\begin{itemize}
\item The functions \( \tau, f, \bar{\rho}\) satisfy
\begin{equation}
\label{eq.rf_conf} \tau \circ \bar{\Phi} = \xi^{-1} \tau, \qquad f \circ \bar{\Phi} = \xi^{-1} f \text{,} \qquad \bar{\rho} \circ \bar{\Phi} = \xi^{-1} r.
\end{equation}
\item $\bar{\Phi}$ is a conformal isometry between $\mf{D} \cap \{ r < R \}$ and an open subset of $\mf{D}$. Specifically,
\begin{equation}
\label{eq.met_conf} \bar{\Phi}_\ast \bar{g} = \xi^{-2} g |_{ \mf{D} \cap \{ r < R \} } \text{,}
\end{equation}
where $\bar{\Phi}_\ast \bar{g} $ denotes the pull-back of $\bar{g}$ through $\bar{\Phi}$.
\end{itemize}
\end{proposition}
\begin{proof}
For \(\tau \circ \bar{\Phi}\), we have
\begin{align*}
\tau \circ \bar{\Phi} = u \circ \bar{\Phi} + v \circ \bar{\Phi} = \frac{u}{(1+ \varepsilon u)} + \frac{v}{(1-\varepsilon v)} = \frac{u+v}{(1+\varepsilon u)(1- \varepsilon v)} = \xi^{-1} \tau.
\end{align*}
Similarly, we get
\begin{align*}
f \circ \bar{\Phi} = (- uv) \circ \bar{\Phi} =  \frac{- uv}{(1+\varepsilon u)(1-\varepsilon v)} = \xi^{-1} f \text{,} \qquad \bar{\rho} \circ \bar{\Phi} = (r+2\varepsilon f)\circ \bar{\Phi} = \xi^{-1} r.
\end{align*}
We calculate the pull-back of the one forms \(du, dv\), and check that they satisfy
\[
\bar{\Phi}_\ast ( d u ) = ( 1 + \varepsilon u )^{-2} du \text{,} \qquad \bar{\Phi}_\ast ( d v ) = ( 1 - \varepsilon v )^{-2} dv \text{.}
\]
The proof is completed after using \eqref{eq.rf_conf} and Remark \ref{rmk_Phi}, to get
\begin{align*}
\bar{\Phi}_\ast \bar{g} &= \bar{\Phi}_\ast ( - 4 du dv + \bar{\rho}^2 \mathring{\gamma}_{\mathbb{S}^{n-1}} - \tau^2 \mathring{\gamma}_{\mathbb{S}^{m-1}}  ) \\
& = - \frac{ 4 du dv }{ ( 1 + \varepsilon u )^2 ( 1 - \varepsilon v )^2 } + ( \bar{\rho} \circ \bar{\Phi} )^2 \mathring{\gamma}_{\mathbb{S}^{n-1}} - (\tau \circ \bar{\Phi})^2 \mathring{\gamma}_{\mathbb{S}^{m-1}}  \\
& = \xi^{-2} ( - 4 du dv + r^2 \mathring{\gamma}_{\mathbb{S}^{n-1}} + \tau^2 \mathring{\gamma}_{\mathbb{S}^{m-1}} ). \qedhere
\end{align*}
\end{proof}

The map \( \bar{\Phi} \) does not alter the geometric quantities by a lot, as can be seen from the next two Propositions.
\begin{proposition} \label{thm.uvf_comp}
Assume the setting of Definition \ref{def.conformal}.
Then:
\begin{itemize}
\item On $\mf{D} \cap \{ r < R \}$, we have the following:
\begin{align}
\label{eq.conf_comp} ( 1 + \varepsilon u )^{m+n} \simeq 1 \text{,}& \qquad ( 1 - \varepsilon v )^{m+n} \simeq 1 \text{,} \qquad \xi^{m+n} \simeq 1 \text{;}\\
\label{eq.uvf_comp} - ( u \circ \bar{\Phi} ) \simeq - u \text{,}& \qquad v \circ \bar{\Phi} \simeq v \text{,} \qquad \qquad \quad f \circ \bar{\Phi} \simeq f \text{;}\\
\label{eq.conf_d_comp} | \partial_u \xi | \simeq \varepsilon \text{,}& \qquad | \partial_v \xi | \simeq \varepsilon \text{.}
\end{align}
\item Let $ W \subseteq \mf{D} \cap \{ r < R \}$ be open and $\bar{z} \in C^1 ( W )$. If we let \(z = \bar{z} \circ \bar{\Phi}\), then we have
\begin{equation}
\label{eq.deriv_comp} | \partial_u \bar{z} | \simeq | \partial_u z | \text{,} \quad | \partial_v \bar{z} | \simeq | \partial_v z | \text{,} \quad \bar{g}^{ab} \bar{\nasla}_a \bar{z} \bar{\nasla}_b \bar{z} \simeq g^{ab} \nasla_a z \nasla_b z, \quad \bar{g}^{CD} \hat{\nabla}_C \bar{z} \hat{\nabla}_D \bar{z} \simeq g^{CD} \tilde{\nabla}_C z \tilde{\nabla}_D z .
\end{equation}
\end{itemize}
\end{proposition}

\begin{proof} 
The first two estimates in \eqref{eq.conf_comp} hold because \( \varepsilon \ll R^{-1}\) and \( 0 < -u, v < R\). The last one is true because of \eqref{eq.conformal_factor}.

Using \eqref{eq.conformal} and \eqref{eq.rf_conf} proves \eqref{eq.uvf_comp}. For \eqref{eq.conf_d_comp}, we calculate the corresponding derivatives and then use \eqref{eq.conf_comp}. Finally, we use \eqref{eq.conformal} and \eqref{eq.conf_comp} to obtain the comparisons in \eqref{eq.deriv_comp}.
\end{proof}

\begin{proposition} \label{thm.gtc_conf}
Assume the setting of Definition \ref{def.conformal}, and suppose $\mf{U} \subset \R^{m+n}$ is open and satisfies
\[
\bar{\mf{U}} \cap \mf{D} \subseteq \{ r < R \} \text{.}
\]
Also assume that the boundary of \(\mf{U}\), denoted by \( \partial\mf{U}\), is a hypersurface. Then, $\bar{\Phi} ( \mf{U} \cap \mf{D} )$ is an open subset of $\mf{D}$, and its boundary in $\mf{D}$ is the hypersurface $\bar{\Phi} ( \partial \mf{U} \cap \mf{D} )$.
\end{proposition}
\begin{proof}
Since \(\bar{\Phi}\) is a conformal isometry and conformal isometries preserve the causal geometry, we conclude that the proposition holds.
\end{proof}

\noindent Now that we have proved that \(\bar{\Phi}\) is indeed a conformal isometry, our next goal is to look at how the wave operator transforms under \(\bar{\Phi}\).
\begin{proposition} \label{thm.conf_wave}
Assume the setting of Definition \ref{def.conformal}.
Let $W$ be an open subset of $\mf{D}$, let $\bar{z} \in C^2 ( W )$, and let $z = \bar{z} \circ \bar{\Phi}$.
Then, the following equation is satisfied:
\begin{equation}
\label{eq.conf_wave} \left[ \bar{\square} + \frac{ (n-1)(m+n-2)\varepsilon}{2 (\bar{\rho}\circ \bar{\Phi})} \right] ( \xi^{\frac{m+n}{2} - 1} \bar{z} ) = \xi^{\frac{m+n}{2}+1} \square z \text{.}
\end{equation}
\end{proposition}

\begin{proof}
The expression for how the scalar curvature changes under a conformal transformation can be found in \cite[Appendix D, p.446]{Wald}, and it is given by
\begin{align*}
\bar{\mc{R}} \circ \bar{\Phi} & = \xi^2 \{ -2( m+n-1)g^{\alpha\beta}\nabla_\alpha\nabla_\beta \text{log} (\xi^{-1}) - (m+n-1) (m+n-2) g^{\alpha\beta}\nabla_\alpha \text{log}(\xi^{-1}) \nabla_\beta \text{log}(\xi^{-1})\}\\
& = 2(m+n-1) \xi  \square \xi - (m+n)(m+n-1)  g^{\alpha\beta} \nabla_\alpha \xi \nabla_\beta\xi,
\end{align*}
where all the derivatives are in the original (pseudo-Riemannian) metric, namely the case when \(\varepsilon = 0\). Then, we compute the following
\begin{align*}
\square \xi & = g^{\alpha\beta}\partial_{\alpha\beta}\xi - g^{\alpha\beta}\Gamma^{\mu}_{\alpha\beta} \partial_\mu \xi = \frac{(m+n)}{2} \varepsilon^2 - \frac{(n-1)\varepsilon}{r} .
\end{align*}
Also,
\begin{align*}
g^{\alpha\beta} \nabla_\alpha \xi \nabla_\beta\xi = -\partial_u\xi \partial_v\xi = \varepsilon^2 (1-\varepsilon v) (1+ \varepsilon u) = \varepsilon^2 \xi.
\end{align*}
Using these in the expression for \( \bar{\mc{R}}\),
\begin{align*}
\bar{\mc{R}} \circ \bar{\Phi} & = 2(m+n-1) \xi \left( \frac{(m+n)}{2} \varepsilon^2 - \frac{(n-1)}{r} \varepsilon \right) - (m+n)(m+n-1) \varepsilon^2 \xi \\
& = \frac{-2(n-1)(m+n-1) \varepsilon}{\bar{\rho} \circ \bar{\Phi}},
\end{align*}
where we also used \eqref{eq.rf_conf} in last step.
The formula for the change of wave operator under conformal transformation (from \cite[Appendix D]{Wald}) is 
\begin{align*}
\left[ \bar{\square} - \frac{(m+n-2)}{4(m+n-1)} (\bar{\mc{R}} \circ \bar{\Phi}) \right] \xi^{\frac{(m+n)}{2}-1} = \xi^{\frac{(m+n)}{2}+1} \square.
\end{align*}
Using the expression for \(\bar{\mc{R}} \circ \bar{\Phi}\) in the LHS of the above equation, shows that
\begin{equation*}
\left[ \bar{\square} - \frac{(m+n-2)}{4(m+n-1)} (\bar{\mc{R}} \circ \bar{\Phi}) \right] \xi^{\frac{(m+n)}{2}-1} = \left[ \bar{\square} - \frac{(m+n-2)}{4(m+n-1)} \cdot \frac{(-2)(n-1)(m+n-1) \varepsilon}{\bar{\rho} \circ \bar{\Phi}} \right] \xi^{\frac{(m+n)}{2}-1}\text{,}
\end{equation*}
which concludes the proof.
\end{proof}

\subsection{Boundary Carleman estimate}\label{sec App_BCE}

Now we will present a Carleman estimate in the warped geometry. This warped estimate will then be used to prove the main boundary Carleman estimate \eqref{eq.carleman_est} provided in Theorem \ref{thm.carl_bdry}, by using the map \(\bar{\Phi}\). 

\begin{definition}
For given constants \(a, b >0\) define \(\zeta_{a, b} \), the warped Carleman weight, as
\begin{equation}
\label{eq.carleman_weight_wp} \zeta_{a, b} := f^{2 a} e^{ 4 a b f^\frac{1}{2} } \text{.}
\end{equation}
\end{definition}

\noindent We have the following result in the warped geometry.

\begin{theorem} \label{thm.carleman_est_wp}
Let \(\mf{U} \subset \R^{m+n}\) be open, such that for some \(R>0\)
\begin{equation}
\label{eq.carleman_domain_wp} \bar{\mf{U}} \cap \mf{D} \subseteq \{ r < R \},
\end{equation}
 and assume that \(\partial\mf{U}\), the boundary of \(\mf{U}\), is smooth and timelike. Let constants \(\varepsilon, a, b > 0\) be chosen such that
\begin{align}
\label{eq.carleman_choices_wp} a \geq (m+n)^2 \text{,} &\qquad \varepsilon \ll_{m,n} b \ll R^{-1}.
\end{align}
Then, for any uniformly $C^1$-bounded funtion \(z \in C^2 ( \bar{\Phi} ( \mf{U} \cap \mf{D} ) ) \cap C^1 ( \bar{\Phi} ( \bar{\mf{U}} \cap \mf{D} ) )\) with
\begin{equation}
\label{eq.carleman_dirichlet_wp} z |_{ \bar{\Phi} ( \partial \mf{U} \cap \mf{D} ) } = 0 \text{,}
\end{equation}
the following Carleman estimate is satisfied
\begin{align}
\notag \frac{ \varepsilon }{ 8 } \int_{ \bar{\Phi} ( \mf{U} \cap \mf{D} ) } \zeta_{a, b} \bar{\rho}^{-1} ( | u \cdot \partial_u z |^2 + & | v \cdot \partial_v z |^2 + f \bar{g}^{ab} \bar{\nasla}_a z \bar{\nasla}_b z  - f \bar{g}^{CD} \hat{\nabla}_C z  \hat{\nabla}_D z ) ) + \frac{ b a^2 }{4} \int_{ \bar{\Phi} ( \mf{U} \cap \mf{D} ) } \zeta_{a, b} f^{- \frac{1}{2} } z^2 \\
\label{eq.carleman_est_wp} & \leqslant \frac{1}{2 a} \int_{ \bar{\Phi} ( \mf{U} \cap \mf{D} ) } \zeta_{a, b} f | \bar{\Box} z |^2 + \int_{ \bar{\Phi} ( \partial \mf{U} \cap \mf{D} ) } \zeta_{a, b} \bar{\mc{N}} f | \bar{\mc{N}} z |^2,
\end{align}
where the integrals are with respect to volume forms induced by \(\bar{g}\), and \(\bar{\mc{N}}\) denotes the outer-pointing unit normal of $\bar{\Phi} ( \mf{U} \cap \mf{D} )$, in the warped geometry. 
\end{theorem}

The proof of Theorem \ref{thm.carleman_est_wp} is analogous to the corresponding proof presented in \cite[Theorem 3.22]{Arick}, for the \( \R^{1+n}\) case. Although the proof that we present here is self contained, we have omitted some details. 

\begin{proof}
Let us assume that the hypothesis of Theorem \ref{thm.carleman_est_wp} is true and let \(z \in C^2 ( \bar{\Phi} ( \mf{U} \cap \mf{D} ) ) \cap C^1 ( \bar{\Phi} ( \bar{\mf{U}} \cap \mf{D} ) )\). We define \(F\), a transformation of \(f\), as
\begin{align}
\label{eq.F} F := F (f) := - a ( \log f + 2 b f^\frac{1}{2} ),
\end{align}
and also define 
\begin{align*}
 \qquad \bar{S} := \bar{\nabla}^\sharp f \text{,} \qquad \bar{S}_w := \bar{S} + \bar{w}.
\end{align*}
We define the following conjugated function
\begin{equation}\label{eq.conjug}
\psi := e^{-F} z,
\end{equation}
and the conjugated wave operator \(\bar{\mc{L}}\), as
\begin{equation}
\label{eq.F_conj} \bar{\mc{L}} \psi := e^{-F} \bar{\Box} (e^F \psi) = e^{-F} \bar{\Box} z. 
\end{equation}
The symbol \(\prime\) denotes derivatives with respect to \(f\) (in this section). Note that, \(\psi \in C^2 ( \bar{\Phi} ( \mf{U} \cap \mf{D} ) )\). Using \eqref{eq.carleman_weight_wp} and calculating the derivatives of \(F\) (with respect to \(f\)), shows that
\begin{equation}
\label{eq.F_deriv} e^{-2 F} = \zeta_{a, b} \text{,} \qquad F' = - a ( f^{-1} + b f^{- \frac{1}{2} } ).
\end{equation}
Note that the above equation also shows that \( - F' > 0 \). 

As per the standard derivation of Carleman estimates, we will now work with \(\psi\) and later go back to \(z\). For the reader's convenience, we present a layout of the proof:
\begin{itemize}

\item \textbf{Step I:} We find a wave type equation for the pair \( (\bar{\mc{L}}, \psi)\). Due to \eqref{eq.F_deriv}, it means that we look at the usual wave operator conjugated with the Carleman weight. Then, we find estimates for the non-dominating terms to obtain a point-wise inequality which only contains the dominating terms. The resulting inequality will be, roughly speaking, a Carleman inequality for the conjugated pair.

\item \textbf{Step II:} We reverse the conjugation done earlier, to go from \( ( \bar{\mc{L}}, \psi)\) to the pair \((\bar{\square}, z)\). This will give us a point-wise Carleman estimate for \(z\).

\item \textbf{Step III:} Finally, we integrate the resulting estimate from Step II above to conclude the proof of the Theorem.

\end{itemize}

\begin{remark}
In comparison with \cite{Arick}, here the new information arising from the temporal angular components is included in the tensor \( \bar{\pi}\) (see Definition \ref{def.pseudoconvex_wp} and Proposition \ref{thm.pseudoconvex_wp} above). A major change is in the final step when we integrate the point-wise Carleman estimate for \(z\).
\end{remark}

\noindent \textbf{Step I.} We begin by defining the function \( \mc{A}\) and the 1-form \( \bar{P}:= \bar{P}[\psi] \) as follows:
\begin{align}
\label{eq.carleman_id_A} \mc{A}(f) & := f ( F' )^2 + ( f F' )' = a^2 f^{-1} + b a \left( 2 a - \frac{1}{2} \right) f^{- \frac{1}{2} } + b^2 a^2, \\
\label{eq.carleman_id_current} \bar{P}_\beta & := \bar{S} \psi \bar{\nabla}_\beta \psi - \frac{1}{2} \bar{\nabla}_\beta f \cdot \bar{\nabla}^\mu \psi \bar{\nabla}_\mu \psi + \bar{w} \cdot \psi \bar{\nabla}_\beta \psi + \frac{1}{2} ( \mc{A} \bar{\nabla}_\beta f - \bar{\nabla}_\beta \bar{w} ) \cdot \psi^2.
\end{align}
First, we calculate
\begin{align*}
\bar{\nabla}^\beta \bar{P}_\beta &= \bar{\nabla}^\beta \left( \bar{S} \psi \bar{\nabla}_\beta \psi - \frac{1}{2} \bar{\nabla}_\beta f \cdot \bar{\nabla}^\mu \psi \bar{\nabla}_\mu \psi + \bar{w} \cdot \psi \bar{\nabla}_\beta \psi + \frac{1}{2} ( \mc{A} \bar{\nabla}_\beta f - \bar{\nabla}_\beta \bar{w} ) \cdot \psi^2 \right)\\
& = \bar{\square} \psi \bar{S} \psi + \bar{\nabla}_{\alpha \beta} f \cdot \bar{\nabla}^\alpha \psi \bar{\nabla}^\beta \psi - \frac{1}{2} \bar{\Box} f \cdot \bar{\nabla}^\mu \psi \bar{\nabla}_\mu \psi + \bar{w} \cdot \psi \bar{\Box} \psi + \bar{w} \cdot \bar{\nabla}^\beta \psi \bar{\nabla}_\beta \psi - \frac{1}{2} \bar{\Box} \bar{w} \cdot \psi^2 \\
& \quad + \mc{A} \cdot \psi \bar{S}_w \psi + \frac{1}{2} ( f \mc{A} )' \cdot \psi^2 + \frac{ \varepsilon f }{2 \bar{\rho} } \mc{A} \cdot \psi^2,
\end{align*}
where we used the fact that 
\begin{align*}
\frac{1}{2} \bar{\nabla}^\beta ( \mc{A} \bar{\nabla}_\beta f \cdot \psi^2 ) = \mc{A} \cdot \psi \bar{S}_w \psi + \frac{1}{2} ( f \mc{A} )' \cdot \psi^2 + \frac{ \varepsilon f }{ 2\bar{\rho} } \mc{A} \cdot \psi^2.
\end{align*}
An application of \eqref{eq.pi_wp} and \eqref{eq.w_wp} shows that
\begin{equation}
\label{eql.carleman_id_ptwise_1a}\bar{\nabla}^\beta \bar{P}_\beta = \bar{\square} \psi \bar{S}_w \psi + \bar{\pi}_{\alpha\beta} \bar{\nabla}^\alpha \psi \bar{\nabla}^\beta \psi - \frac{1}{2} \bar{\Box} \bar{w} \cdot \psi^2 + \mc{A} \cdot \psi \bar{S}_w \psi + \frac{1}{2} ( f \mc{A} )' \cdot \psi^2 + \frac{ \varepsilon f }{2 \bar{\rho} } \mc{A} \cdot \psi^2. 
\end{equation}
Using Proposition \ref{thm.f_deriv_wp}, Definition \ref{def.pseudoconvex_wp}, and \eqref{eq.F}, we get
\begin{equation}
\label{eql.carleman_id_ptwise_1} \bar{\mc{L}} \psi \bar{S}_w \psi = \bar{\Box} \psi \bar{S}_w \psi + 2 F' \cdot | \bar{S}_w \psi |^2 + \left( \mc{A} + \frac{ \varepsilon f }{ \bar{\rho} } F' \right) \cdot \psi \bar{S}_w \psi \text{.}
\end{equation}
Then, using \eqref{eql.carleman_id_ptwise_1a} and \eqref{eql.carleman_id_ptwise_1} along with \eqref{eq.pseudoconvex_wp}, shows that \( (\bar{\mc{L}}, \psi)\) satisfies the following equation:
\begin{align}
\label{eq.carleman_id_ptwise} - \bar{\mc{L}} \psi \bar{S}_w \psi + \bar{\nabla}^\beta \bar{P}_\beta &= - 2 F' \cdot | \bar{S}_w \psi |^2 + \frac{ \varepsilon f }{ 2 \bar{\rho} } ( | T \psi |^2 + \bar{g}^{ab} \bar{\nasla}_a \psi \bar{\nasla}_b \psi - \bar{g}^{CD} \hat{\nabla}_C \psi \hat{\nabla}_D \psi  - | N \psi |^2 ) \\
\notag & \qquad - \frac{ \varepsilon f }{ \bar{\rho} } F' \cdot \psi \bar{S}_w \psi + \frac{1}{2} \left[ ( f \mc{A} )' + \frac{ \varepsilon f }{ \bar{\rho} } \mc{A} - \bar{\Box} \bar{w} \right] \cdot \psi^2.
\end{align}
Note that the above is an equation for the conjugated wave operator multiplied with the chosen multiplier function. 

Now, we will find estimates for some of the terms in \eqref{eq.carleman_id_ptwise} to obtain a point-wise estimate. As a first step in this direction, let us define the operator
\begin{equation}
\label{eq.carleman_est_tilde} \tilde{N} := f^{-\frac{m+n-2}{4} } N f^{ \frac{m+n-2}{4}}.
\end{equation}
This will help us to replace the \( -|N \psi|^2\) term on the RHS of \eqref{eq.carleman_id_ptwise}, which has a bad sign, with the term \( | \tilde{N} \psi |^2 \).
\begin{namedthm*}{\underline{Claim 1}}[]
The following is satisfied
\begin{align}
\label{eq.carleman_est_ptwise} \frac{1}{ 4 a } f | \bar{\mc{L}} \psi |^2 + \bar{\nabla}^\beta \bar{P}_\beta &\geq \frac{ \varepsilon f }{ 2 \bar{\rho} } ( | T \psi |^2 + \bar{g}^{ab} \bar{\nasla}_a \psi \bar{\nasla}_b \psi - \bar{g}^{CD} \hat{\nabla}_C \psi \hat{\nabla}_D \psi ) + \frac{1}{4} a | \tilde{N} \psi |^2 + \frac{1}{4} b a^2 f^{- \frac{1}{2} } \cdot \psi^2 .
\end{align}
\end{namedthm*}

\noindent\underline{\textit{Proof of claim}}: 
Using the definition of \( \varepsilon, N, \tilde{N}\), along with equation \eqref{eq.f_grad_wp}, we get the following estimate for the multiplier term
\[
\bar{S}_w \psi = f^\frac{1}{2} \tilde{N} \psi + \frac{ (n - 2) \varepsilon f }{ 2 \bar{\rho} } \cdot \psi \geq f^\frac{1}{2} \tilde{N} \psi - \varepsilon n C f^\frac{1}{2} \cdot | \psi |,
\]
for some constant \(C>0\). Now we find lower bounds for some of the terms in the RHS of \eqref{eq.carleman_id_ptwise}. We use equations \eqref{eq.uvf_bound}, \eqref{eq.rho_est_wp}, \eqref{eq.w_box_wp}, \eqref{eq.carleman_choices_wp}, \eqref{eq.F_deriv}, and the above estimate for \( \bar{S}_w \psi \), to show that
\begin{align*}
- F' \cdot  | \bar{S}_w \psi |^2 - &\frac{ \varepsilon f }{ 2 \bar{\rho} } | N \psi |^2 - \frac{ \varepsilon f }{ \bar{\rho} } F' \cdot \psi \bar{S}_w \psi - \frac{1}{2} \bar{\Box} \bar{w} \cdot \psi^2 \\ & \geq a ( 1 + b f^\frac{1}{2} ) \left( \frac{1}{2} | \tilde{N} \psi |^2 - \varepsilon^2 n^2 C \cdot \psi^2 \right) - \varepsilon C f^\frac{1}{2} \cdot | \tilde{N} \psi |^2 - \varepsilon C (m+n)^2 f^{ - \frac{1}{2} } \cdot \psi^2 \text{.}
\end{align*}
Due to \eqref{eq.uvf_bound}, \eqref{eq.carleman_domain_wp}, and the choice of constants \(b, R\) from \eqref{eq.carleman_choices_wp}, we have
\[ b \ll R^{-1} \Rightarrow b f^{\frac{1}{2}} \ll f^{\frac{1}{2}} R^{-1} \ll r R^{-1} \ll 1.  \]
Hence, we conclude that \( 1 + b f^\frac{1}{2} \simeq 1\). Moreover, using this along with \eqref{eq.carleman_choices_wp} again, reduces the estimate to
\begin{equation}
\label{eql.carleman_est_ptwise_2} - F' \cdot  | \bar{S}_w \psi |^2 - \frac{ \varepsilon f }{ 2 \bar{\rho} } | N \psi |^2 - \frac{ \varepsilon f }{ \bar{\rho} } F' \cdot \psi \bar{S}_w \psi - \frac{1}{2} \bar{\Box} \bar{w} \cdot \psi^2 \geq \frac{1}{4} a \cdot | \tilde{N} \psi |^2 - a^2 \varepsilon C \cdot f^{ - \frac{1}{2} } \psi^2 \text{.}
\end{equation}
For the remaining terms in \eqref{eq.carleman_id_ptwise}, using the definition of \(\mc{A}\) and equation \eqref{eq.carleman_choices_wp}, we get
\begin{equation}
\label{eql.carleman_est_ptwise_3} ( f \mc{A} )' + \frac{ \varepsilon f }{ \bar{\rho} } \mc{A} \geq \left(\frac{1}{2} b - \varepsilon C \right) a^2 f^{ - \frac{1}{2} }.
\end{equation}
Since \( -F'>0\), we use the estimate
\[
- \frac{1}{ 4 F' } | \bar{\mc{L}} \psi |^2 - F' | \bar{S}_w \psi |^2 \geqslant | \bar{\mc{L}} \psi | | \bar{S}_w \psi |,
\]
and equation \eqref{eq.carleman_id_ptwise}, to get
\begin{align}
\label{eql.carleman_est_ptwise_4} - \frac{1}{ 4 F' } | \bar{\mc{L}} \psi |^2 + \bar{\nabla}^\beta \bar{P}_\beta & \geq \frac{ \varepsilon f }{ 2 \bar{\rho} } ( | T \psi |^2 + \bar{g}^{ab} \bar{\nasla}_a \psi \bar{\nasla}_b \psi  - \bar{g}^{CD} \hat{\nabla}_C \psi \hat{\nabla}_D \psi ) + \frac{1}{4} a | \tilde{N} \psi |^2  \\
\notag & \qquad + \left( \frac{1}{2} b - \varepsilon C \right) a^2 f^{ - \frac{1}{2} } \cdot \psi^2 \text{.}
\end{align}
Finally, we use \eqref{eq.F_deriv} to obtain the coefficient of the \(| \bar{\mc{L}} \psi |^2\) term, and use the fact that \( \varepsilon \ll b \), to conclude that the claim is true.

\medskip

\noindent\textbf{Step II.} In this step we will go back from the (conjugated) function \(\psi\) to the (original) function \(z\). Define the 1-form \(\bar{P}^\star = \bar{P}^\star [z] \) on \( \bar{\Phi} ( \mf{U} \cap \mf{D} ) \) as
\begin{align}
\label{eq.carleman_rev_current} \bar{P}^\star_\beta &:= \bar{S} ( e^{ - F } z ) \bar{\nabla}_\beta ( e^{ - F } z ) - \frac{1}{2} \bar{\nabla}_\beta f \cdot \bar{\nabla}^\mu ( e^{ - F } z ) \bar{\nabla}_\mu ( e^{ - F } z ) + \bar{w} \cdot e^{ - F } z \bar{\nabla}_\beta ( e^{ - F } z ) \\
\notag &\qquad + \frac{1}{2} ( \mc{A} \bar{\nabla}_\beta f - \bar{\nabla}_\beta \bar{w} ) \cdot e^{ - 2 F } z^2.
\end{align}
Due to \eqref{eq.carleman_id_current}, we can see that \( \bar{P}^\star_\beta \) is related to \( \bar{P}_\beta\) by the conjugation \eqref{eq.conjug}. That is, \\ \( \bar{P}^\star_\beta (z)  = \bar{P}_\beta ( e^{-F} z )\). The next claim gives us the resulting point-wise Carleman estimate for \(z\). The idea is to substitute back the conjugation expression \eqref{eq.conjug} in terms of \(z\), and carefully estimate some terms that arise.
\begin{namedthm*}{\underline{Claim 2}}[]
On $\bar{\Phi} ( \mf{U} \cap \mf{D} )$, we have
\begin{align}
\label{eq.carleman_est_rev} \frac{1}{ 4 a } f \zeta_{a, b} | \bar{\Box} z |^2 + \bar{\nabla}^\beta \bar{P}^\star_\beta &\geq \frac{ \varepsilon }{ 16 \bar{\rho} } \zeta_{a, b} ( | u \cdot \partial_u z |^2 + | v \cdot \partial_v z |^2 + f \cdot \bar{g}^{ab} \bar{\nasla}_a z \bar{\nasla}_b z  - f \cdot \bar{g}^{CD} \hat{\nabla}_C z \hat{\nabla}_D z ) \\
\notag & \qquad + \frac{1}{8} b a^2 f^{- \frac{1}{2} } \zeta_{a, b} \cdot z^2.
\end{align}
\end{namedthm*}
\noindent\underline{\textit{Proof of claim}}: 
We start by putting back \(\psi = e^{-F} z\). Writing \eqref{eq.carleman_est_ptwise} in terms of \(z\), shows
\begin{align}
\label{eql.carleman_est_rev_1} \frac{1}{ 4 a } f e^{ -2 F } | \bar{\Box} z |^2 + \bar{\nabla}^\beta \bar{P}^\star_\beta &\geq \frac{ \varepsilon f }{ 2 \bar{\rho} } e^{ -2 F } ( | T z |^2 + \bar{g}^{ab} \bar{\nasla}_a z \bar{\nasla}_b z - \bar{g}^{CD} \hat{\nabla}_C z \hat{\nabla}_D z ) + \frac{1}{4} a | \tilde{N} ( e^{ - F } z ) |^2 \\
\notag & \qquad + \frac{1}{4} b a^2 f^{- \frac{1}{2} } e^{ - 2 F } \cdot z^2.
\end{align}
As we can see from \eqref{eq.carleman_est_wp}, the first order terms contain only the null coordinates and the angular coordinate derivatives. We already have the angular components in the above equation. Hence, we just have to express the remaining derivative terms in \eqref{eql.carleman_est_rev_1}, in \( (\partial_u,\partial_v)\). First, we want to replace the \(\tilde{N}\) on the RHS by \(N\). For this purpose, we use \eqref{eq.f_grad_wp}, \eqref{eq.TN}, \eqref{eq.carleman_choices_wp}, \eqref{eq.F_deriv}, and \eqref{eq.carleman_est_tilde}, to find that
\begin{align*}
a | \tilde{N} ( e^{-F} z ) |^2 + b a^2 e^{-2 F} f^{ - \frac{1}{2} } \cdot z^2 \gg \varepsilon f^\frac{1}{2} e^{-2 F} | N z |^2 \geqslant \frac{ \varepsilon f }{ 2 \bar{\rho} }  e^{-2 F} | N z |^2,
\end{align*}
where we used \eqref{eq.rho_est_wp} to get the last inequality. Then \eqref{eql.carleman_est_rev_1} reduces to 
\begin{align}
\label{eql.carleman_est_rev_2} \frac{1}{ 4 a } f e^{ -2 F } | \bar{\Box} z |^2 + \bar{\nabla}^\beta \bar{P}^\star_\beta & \geq \frac{ \varepsilon f }{ 2 \bar{\rho} } e^{ -2 F } ( | T z |^2 + | N z |^2 + \bar{g}^{ab} \bar{\nasla}_a z \bar{\nasla}_b z  - \bar{g}^{CD} \hat{\nabla}_C z \hat{\nabla}_D z ) \\
\notag & \qquad + \frac{1}{8} b a^2 f^{- \frac{1}{2} } e^{ - 2 F } \cdot z^2 \text{.}
\end{align}
Now, we will estimate the \(T, N\) terms by the (null coordinate) vector fields \(\partial_u, \partial_v\). From \eqref{eq.f} and \eqref{eq.TN}, we get
\[ | u \partial_u z |^2 + | v \partial_v z |^2 \leq 8 f ( | T z |^2 + | N z |^2 ) .\]
Using \eqref{eql.carleman_est_rev_2}, along with the above estimate, and \eqref{eq.F_deriv} (to get the Carleman weight) completes the proof of the claim. 

We now present some properties of \(\bar{P}^\star\) that will be used later when we integrate the point-wise estimate \eqref{eq.carleman_est_rev}; in particular, it comes up when dealing with the boundary integral terms.

\begin{namedthm*}{\underline{Claim 3}}[]
If \(z\) satisfies \eqref{eq.carleman_dirichlet_wp}, then 
\begin{equation}
\label{eq.carleman_rev_dirichlet} \bar{P}^\star ( \bar{\mc{N}} ) |_{ \bar{\Phi} ( \partial \mf{U} \cap \mf{D} ) } = \frac{1}{2} \zeta_{a, b} \cdot \bar{\mc{N}} f | \bar{\mc{N}} z |^2 |_{ \bar{\Phi} ( \partial \mf{U} \cap \mf{D} ) } \text{.}
\end{equation}
Also, on \(\bar{\Phi} ( \mf{U} \cap \mf{D} )\)
\begin{equation}
\label{eq.carleman_rev_currest} | \bar{P}^\star ( \bar{S} ) | \lesssim R^2 e^a f^{2 a} ( | \partial_u z |^2 + | \partial_v z |^2 + \bar{g}^{ab} \bar{\nasla}_a z \bar{\nasla}_b z  - \bar{g}^{CD} \hat{\nabla}_C z \hat{\nabla}_D z + a^2 f^{-1} z^2 ) \text{.}
\end{equation}
\end{namedthm*}

\noindent\underline{\textit{Proof of claim}}: 
Since \(z|_{ \bar{\Phi} ( \partial \mf{U} \cap \mf{D} ) } = 0\), we get for \(\bar{P}^\ast\) that
\[
\bar{P}^\star ( \bar{\mc{N}} ) |_{ \bar{\Phi} ( \partial \mf{U} \cap \mf{D} ) } = e^{ - 2 F } \left. \left( \bar{S} z \bar{\mc{N}} z - \frac{1}{2} \bar{\mc{N}} f \cdot | \bar{\mc{N}} z |^2 \right) \right|_{ \bar{\Phi} ( \partial \mf{U} \cap \mf{D} ) } \text{.}
\]
To check \eqref{eq.carleman_rev_dirichlet}, we note that \( \bar{S} z |_{ \bar{\Phi} ( \partial \mf{U} \cap \mf{D} ) } = \bar{\mc{N}} f \cdot \bar{\mc{N}} z |_{ \bar{\Phi} ( \partial \mf{U} \cap \mf{D} ) } \). To prove \eqref{eq.carleman_rev_currest}, we have to consider
\[ \bar{P}^\star (\bar{S}) = \bar{S} ( e^{ - F } z ) \bar{S} ( e^{ - F } z ) - \frac{1}{2} \bar{S} f \cdot \bar{\nabla}^\mu ( e^{ - F } z ) \bar{\nabla}_\mu ( e^{ - F } z ) + \bar{w} e^{ - F } z \cdot \bar{S} ( e^{ - F } z ) + \frac{1}{2} \left( \mc{A} \bar{S} f - \bar{S} \bar{w} \right) e^{ - 2 F } z^2. \]
For the last term in the above identity, using Proposition \ref{thm.f_rho_wp}, the definition of \(w\) from \eqref{eq.w_wp}, and \eqref{eq.carleman_choices_wp}, shows that
\[ | ( \mc{A} \bar{S} f - \bar{S} \bar{w} ) \cdot e^{ - 2 F } z^2 | \lesssim  R^2 e^{ -2 F } a^2 f^{-1} z^2. \]
For the remaining three terms of \( \bar{P}^\star ( \bar{S} ) \), we use equations \eqref{eq.f_grad_wp}, \eqref{eq.uvf_bound}, and \eqref{eq.carleman_domain_wp}, to conclude that all three of them are bounded from above by 
\[ R^2 e^{ - 2 F } ( | \partial_u z |^2 + | \partial_v z |^2 + \bar{g}^{ab} \bar{\nasla}_a z \bar{\nasla}_b z - \bar{g}^{CD} \hat{\nabla}_C z \hat{\nabla}_D z  + a^2 f^{-1} z^2 ).\]
Using \eqref{eq.carleman_choices_wp} and the fact that \( F=- a ( \log f + 2 b f^\frac{1}{2} ) \), shows that the claim is true.

\medskip

\noindent\textbf{Step III.} As mentioned earlier, we will integrate \eqref{eq.carleman_est_rev} over the region \( \bar{\Phi} ( \mf{U} \cap \mf{D} ) \).
Technically, we will integrate the pointwise Carleman estimate on domains that approximate the region \( \bar{\Phi} ( \mf{U} \cap \mf{D} ) \) in the limit.

\noindent For this purpose, let $\delta > 0$ be sufficiently small, and define the sets
\[
\mc{G}_\delta := \bar{\Phi} ( \mf{U} \cap \mf{D} ) \cap \{ f > \delta \} \text{,} \qquad \mc{H}_\delta := \bar{\Phi} ( \mf{U} \cap \mf{D} ) \cap \{ f = \delta \}.
\]
On \(\mc{G}_\delta, \mc{H}_\delta\), due to \eqref{eq.uvf_bound}, \eqref{eq.rho_est_wp}, and \eqref{eq.carleman_domain_wp},  we have that $r, \bar{\rho}, f$ are bounded from both above and below. 
Next, we integrate \eqref{eq.carleman_est_rev} over $\mc{G}_\delta$, take the limit as $\delta \searrow 0$, and use the monotone convergence theorem to get
\begin{align}
\label{eql.carleman_est_wp_1} \frac{1}{ 4 a } \int_{ \bar{\Phi} ( \mf{U} \cap \mf{D} ) } & \zeta_{a, b} f | \bar{\Box} z |^2 + \lim_{ \delta \searrow 0 } \int_{ \mc{G}_\delta } \bar{\nabla}^\alpha \bar{P}^\star_\alpha \\
\notag & \geq \frac{ \varepsilon }{ 16 } \int_{ \bar{\Phi} ( \mf{U} \cap \mf{D} ) } \zeta_{a, b} \bar{\rho}^{-1} ( | u \partial_u z |^2 + | v \partial_v z |^2 + f \bar{g}^{ab} \bar{\nasla}_a z \bar{\nasla}_b z  - f \bar{g}^{CD} \hat{\nabla}_C z \hat{\nabla}_D z ) \\
\notag & \quad + \frac{ b a^2 }{ 8 } \int_{ \bar{\Phi} ( \mf{U} \cap \mf{D} ) } \zeta_{a, b} f^{- \frac{1}{2} } z^2, 
\end{align}
if the limit term on the LHS exists.

Next, we will look at the limit term in the above equation. For this, note that the space-time boundary of \( \mc{G}_\delta\) consists of
\[ \partial \mc{G}_\delta = \big(\bar{\Phi} ( \partial \mf{U} \cap \mf{D} ) \cap \{ f > \delta \} \big) \ \bigcup \ \mc{H}_\delta. \]
For the first part, the outward unit normal is \( \bar{\mc{N}} \). For \( \mc{H}_\delta\), we see that the region \( \{f= \delta\} \) is a time-like hypersurface of $( \mf{D}, \bar{g} )$. The  outward unit normal of this hypersurface, towards \( \{f > \delta\}\), is given by $- f^{ - \frac{1}{2} } \bar{S}$. Hence, using the divergence theorem shows that
\begin{align*}
\int_{ \mc{G}_\delta } \bar{\nabla}^\alpha \bar{P}^\star_\alpha &= \int_{ \bar{\Phi} ( \partial \mf{U} \cap \mf{D} ) \cap \{ f > \delta \} } \bar{P}^\star ( \bar{\mc{N}} ) -  \int_{ \mc{H}_\delta } f^{ - \frac{1}{2} } \bar{P}^\star ( \bar{S} ) .
\end{align*}
Taking limit for the first term on the RHS of the above equation, and using \eqref{eq.carleman_rev_dirichlet} along with the fact that $z$ is uniformly $C^1$-bounded, shows
\[
\lim_{ \delta \searrow 0 } \int_{ \bar{\Phi} ( \partial \mf{U} \cap \mf{D} ) \cap \{ f > \delta \} } \bar{P}^\star ( \bar{\mc{N}} )  = \frac{1}{2} \int_{ \bar{\Phi} ( \partial \mf{U} \cap \mf{D} ) } \zeta_{a, b} \bar{\mc{N}} f \cdot | \bar{\mc{N}} z |^2 ,
\]
which gives us the boundary term present in \eqref{eq.carleman_est_wp}, with a factor of half.
Then, the limit term in \eqref{eql.carleman_est_wp_1} reduces to
\begin{align}
\label{eql.carleman_est_wp_2} \lim_{ \delta \searrow 0 } \int_{ \mc{G}_\delta } \bar{\nabla}^\alpha \bar{P}^\star_\alpha &=  \frac{1}{2} \int_{ \bar{\Phi} ( \partial \mf{U} \cap \mf{D} ) } \zeta_{a, b} \bar{\mc{N}} f \cdot | \bar{\mc{N}} z |^2  - \lim_{ \delta \searrow 0 } \delta^{ - \frac{1}{2} }\int_{ \mc{H}_\delta } \bar{P}^\star (\bar{S} ),
\end{align}
where we used the fact that \( f=\delta\) on \(\mc{H}_\delta\), for the final term. The integral term with \( \bar{P}^\star ( \bar{S} ) \) is estimated using \eqref{eq.carleman_rev_currest}, to get
\begin{align}
\notag \left| \lim_{ \delta \searrow 0 } \delta^{ - \frac{1}{2} } \int_{ \mc{H}_\delta } \bar{P}^\star ( \bar{S} ) \right| &\lesssim R^2 e^a \lim_{ \delta \searrow 0 } \delta^{ 2 a - \frac{1}{2} } \int_{ \mc{H}_\delta } ( | \partial_u z |^2 + | \partial_v z |^2 + \bar{g}^{ab} \bar{\nasla}_a z \bar{\nasla}_b z  - f \bar{g}^{CD} \hat{\nabla}_C z \hat{\nabla}_D z  ) \\
\notag &\qquad + R^2 e^a \lim_{ \delta \searrow 0 } \delta^{ 2 a - \frac{1}{2} } \int_{ \mc{H}_\delta } a^2 f^{-1} z^2  \\
&\lesssim C e^a a^2 \lim_{ \delta \searrow 0 } \delta^{ 2 a - \frac{3}{2} } \int_{ \mc{H}_\delta } 1 \text{,} \label{eql.carleman_est_wp_3}
\end{align}
where $C:= C(z,R) > 0$ is a constant.

Now, if we can show that the RHS of \eqref{eql.carleman_est_wp_3} is zero, then the proof will be complete. To show the limit is zero, first let \( \bar{g}_\mc{H}\) denote the metric induced by \(\bar{g}\) on \(\mc{H}_\delta\). Let \( \bar\partial_\tau\) denote the coordinate vector field along \(\tau\) on \( \mc{H}_\delta\). The vector fields \( \bar\partial_\tau\) and \( T \) point in the same direction, hence on \( \mc{H}_\delta\)
\[ \bar\partial_\tau = k T, \]
for some scalar \(k\). Since \( T=  f^{-\frac{1}{2}} (-u\partial_u + v \partial_v) \) and \( \tau = u+v\), we have
\[ 1 =\bar \partial_\tau \tau = k T \tau = k f^{-\frac{1}{2}} (-u+v)= k r  f^{-\frac{1}{2}}. \]
Thus, \( \bar\partial_\tau = r^{-1} (-u \partial_u + v \partial_v) \). Then
\[
\bar{g}_\mc{H} (\bar\partial_\tau, \bar\partial_\tau) = - 4 r^{-2} f \text{,}
\]
and the angular components remain unchanged. Hence, we have the following
\begin{align*}
\bar{g}_\mc{H} = -4 r^{-2} f d\tau^2 - \tau^2 \mathring{\gamma}_{\Sph^{m-1}} + \bar{\rho}^2 \mathring{\gamma}_{\Sph^{n-1}}.
\end{align*}
Now we calculate the integral over \(\mc{H}_\delta\) as follows
\begin{align*}
\lim_{ \delta \searrow 0 } \delta^{ 2 a - \frac{3}{2} } \int_{ \mc{H}_\delta } 1 & \lesssim \lim_{ \delta \searrow 0 } \delta^{ 2 a - \frac{3}{2} } \int_0^{ 2 R} d\tau \int_{\Sph^{m-1}} d \mathring{\gamma}_{\omega_t}   \int_{\Sph^{n-1}} \sqrt{|\bar{g}_\mc{H}|}\ d\mathring{\gamma}_{\omega_x} \\
& \lesssim \lim_{ \delta \searrow 0 } \delta^{ 2 a - \frac{3}{2} } \int_0^{ 2 R} d\tau \int_{\Sph^{m-1}} d \mathring{\gamma}_{\omega_t}  \int_{\Sph^{n-1}} 2 f^{\frac{1}{2}} r^{-1} \tau^{m-1} \bar{\rho}^{n-1}|_{(\tau, r= \sqrt{\tau^2 + 4\delta}, \omega_x, \omega_t)} d\mathring{\gamma}_{\omega_x}  \\
& \lesssim  R^{m+n-2} \lim_{ \delta \searrow 0 } \delta^{ 2 a - 1} \int_0^{ 2 R} r^{-1}|_{(\tau, r= \sqrt{\tau^2 + 4\delta}, \omega_x, \omega_t)}  d\tau \\
& \lesssim R^{m+n-2} \lim_{ \delta \searrow 0 } \delta^{ 2 a - 1 } \int_0^{2 R} \delta^{-\frac{1}{2}} d\tau\\
& \lesssim R^{m+n-1} \lim_{ \delta \searrow 0 } \delta^{ 2 a - \frac{3}{2} }\\
& =0.
\end{align*} 
Using this along with \eqref{eql.carleman_est_wp_1}-\eqref{eql.carleman_est_wp_3} shows that \eqref{eq.carleman_est_wp} is true and this completes the proof of Theorem \ref{thm.carleman_est_wp}. 
\end{proof}

\subsection{Proof of Theorem \ref{thm.carl_bdry}}

Now we will see how the warped Carleman estimate \eqref{eq.carleman_est_wp} can be used to obtain the Carleman estimate \eqref{eq.carleman_est} (in the original geometry). Mainly, we will be pulling back terms in \eqref{eq.carleman_est_wp} using the diffeomorphism \(\bar{\Phi}\).

\begin{proof}
Let us assume that the hypotheses of the Theorem holds. 
Let \(\bar{z}\) be the uniformly $C^1$-bounded function defined as
\begin{equation}
\label{eql.carleman_est_0} \bar{z} := z \circ \bar{\Phi}^{-1} \in C^2 ( \bar{\Phi} ( \mf{U} \cap \mf{D} ) ) \cap C^1 ( \bar{\Phi} ( \bar{\mf{U}} \cap \mf{D} ) ).
\end{equation}
Then, \(\bar{z}\) satisfies the Dirichlet boundary condition \eqref{eq.carleman_dirichlet_wp}. Applying Theorem \ref{thm.carleman_est_wp}, with \(z\) in \eqref{eq.carleman_est_wp} replaced by $\xi^\frac{m+n-2}{2} \bar{z}$, we obtain
\begin{align*}
\frac{ b a^2 }{4} \int_{ \bar{\Phi} ( \mf{U} \cap \mf{D} ) } & \zeta_{a, b} f^{- \frac{1}{2} } \xi^{m+n-2} \bar{z}^2 \cdot d \bar{g}\\
& \quad + \frac{ \varepsilon }{ 8 } \int_{ \bar{\Phi} ( \mf{U} \cap \mf{D} ) } \zeta_{a, b} \bar{\rho}^{-1} \big[ | u \cdot \partial_u ( \xi^\frac{m+n-2}{2} \bar{z} ) |^2 + | v \cdot \partial_v ( \xi^\frac{m+n-2}{2} \bar{z} ) |^2  \\
& \qquad \qquad \ \qquad \qquad \qquad + f \bar{g}^{ab} \xi^{m+n-2} \bar{\nasla}_a \bar{z} \bar{\nasla}_b \bar{z} - f\bar{g}^{CD}  \xi^{m+n-2}  \hat{\nabla}_C z \hat{\nabla}_D z  \big] \cdot d \bar{g} \\
\numberthis \label{eql.carleman_est_00} & \qquad \qquad \leqslant \frac{1}{2 a} \int_{ \bar{\Phi} ( \mf{U} \cap \mf{D} ) } \zeta_{a, b} f | \bar{\Box} ( \xi^\frac{m+n-2}{2} \bar{z} ) |^2 \cdot d \bar{g} + \int_{ \bar{\Phi} ( \partial \mf{U} \cap \mf{D} ) } \zeta_{a, b} \bar{\mc{N}} f | \bar{\mc{N}} ( \xi^\frac{m+n-2}{2} \bar{z} ) |^2 \cdot d \bar{g}
\end{align*}
\textbf{RHS Terms:} The first step is to estimate each term in the RHS of the above equation by terms with integral over the region \( \mf{U} \cap \mf{D}\). For this purpose, note that adding and subtracting the same term, we get
\[
\bar{\Box} ( \xi^\frac{m+n-2}{2} \bar{z} ) =\left[ \bar{\square} + \frac{ (n-1)(m+n-2)\varepsilon}{2 (\bar{\rho}\circ \bar{\Phi})} \right] ( \xi^{\frac{m+n-2}{2}} \bar{z} ) -\frac{ (n-1)(m+n-2)\varepsilon}{2 (\bar{\rho}\circ \bar{\Phi})}  ( \xi^{\frac{m+n-2}{2}} \bar{z} ) \text{.}
\]
Then, using the change of wave operator formula from \eqref{eq.conf_wave} gives the following inequality for the wave operator term in \eqref{eql.carleman_est_00}
\begin{align*}
\numberthis \label{eql.carleman_est_12} \int_{ \bar{\Phi} ( \mf{U} \cap \mf{D} ) } \zeta_{a, b} f | \bar{\Box} ( \xi^\frac{m+n-2}{2} \bar{z} ) |^2 \cdot d \bar{g}
& \lesssim   \int_{ \mf{U} \cap \mf{D} } \zeta_{a, b; \varepsilon} (\xi^{-1} f) \xi^{m+n+2}|\Box z|^2 \cdot \xi^{-m-n} dg \\
& \qquad + (m+n)^4 \varepsilon^2 \int_{ \mf{U} \cap \mf{D} } \zeta_{a, b; \varepsilon} (\xi^{-1} f) \xi^2 r^{-2} \xi^{m+n-2} z^2  \cdot \xi^{-m-n} dg \\
& \lesssim \int_{ \mf{U} \cap \mf{D} } \zeta_{a, b; \varepsilon} f | \Box z |^2 \cdot d g + (m+n)^4 \varepsilon^2 \int_{ \mf{U} \cap \mf{D} } \zeta_{a, b; \varepsilon} r^{-2} f z^2 \cdot d g \text{,}
\end{align*}
where we used \eqref{eq.rf_conf} to make the following substitutions \( \bar{\rho} \mapsto \xi^{-1}r,\ f \mapsto \xi^{-1}f,\ \zeta_{a, b} \mapsto \zeta_{a, b; \varepsilon} \). The factor \(\xi^{-m-n}\) arises due to the change of volume forms; any power of \(\xi\) that survives is taken care by using \eqref{eq.conf_comp}.

Now we will deal with the boundary integral term in \eqref{eql.carleman_est_00}. Using Proposition \ref{thm.gtc_conf}, we get
\begin{align*}
\bar{\Phi} ( \partial \mf{U} \cap \mf{D} ) & \xmapsto{\bar{\Phi}^{-1}} \partial \mf{U} \cap \mf{D}, \\
\bar{\mc{N}} & \xmapsto{\bar{\Phi}^{-1}} \xi \mc{N},
\end{align*}
where \(\mc{N}\) is the outer pointing unit normal of \( \mf{U}\) (with respect to \(g\)). Hence, the boundary term satisfies the following equality
\begin{align}
\label{eql.carleman_est_13} \int_{ \bar{\Phi} ( \partial \mf{U} \cap \mf{D} ) } \zeta_{a, b} \bar{\mc{N}} f | \bar{\mc{N}} ( \xi^\frac{m+n-2}{2} \bar{z} ) |^2 \cdot d \bar{g} &= \int_{ \partial \mf{U} \cap \mf{D} } \zeta_{a, b; \varepsilon} ( \xi \mc{N} ) ( \xi^{-1} f ) | ( \xi \mc{N} ) ( \xi^\frac{m+n-2}{2} z ) |^2 \cdot \xi^{-m-n+1} d g \\
\notag &= \int_{ \partial \mf{U} \cap \mf{D} } \zeta_{a, b; \varepsilon} \xi^2 \mc{N} ( \xi^{-1} f ) | \mc{N} z |^2 \cdot d g \text{,}
\end{align}
where we used the assumption \( z |_{ \partial \mf{U} \cap \mf{D} } = 0\) for the final equality (hence, the only term that survives from \(\mc{N} ( \xi^\frac{m+n-2}{2} z )\) is when \( \mc{N}\) hits \(z\)). Also, the factor \( \xi^{-m-n+1} \) comes in because the region \( \partial \mf{U} \cap \mf{D} \) has dimension \( m+n-1 \), one less than the dimension of \( \mf{U} \cap \mf{D}\).

\medskip

\noindent \textbf{LHS Terms:} Now we will show that the terms in the LHS of \eqref{eql.carleman_est_00} are bounded from below by terms with integral over \( \mf{U}\cap \mf{D}\). First, for the zeroth-order term and the angular derivative terms, we use the same argument as we did for showing \eqref{eql.carleman_est_12}, to obtain
\begin{align*}
\numberthis \label{eql.carleman_est_10} \int_{ \bar{\Phi} ( \mf{U} \cap \mf{D} ) } &  \zeta_{a, b} f^{- \frac{1}{2} } \xi^{m+n-2} \bar{z}^2 \cdot d \bar{g} + \int_{ \bar{\Phi} ( \mf{U} \cap \mf{D} ) } \zeta_{a, b} \bar{\rho}^{-1} f \bar{g}^{ab} \xi^{m+n-2} \bar{\nasla}_a \bar{z} \bar{\nasla}_b \bar{z} \cdot d \bar{g} \\
& - \int_{ \bar{\Phi} ( \mf{U} \cap \mf{D} ) } \zeta_{a, b} \bar{\rho}^{-1} f \bar{g}^{CD}  \xi^{m+n-2}  \hat{\nabla}_C z \hat{\nabla}_D z \cdot d \bar{g} \\
& \qquad \quad \gtrsim \int_{ \mf{U} \cap \mf{D} } \zeta_{a, b; \varepsilon} f^{ - \frac{1}{2} } z^2 \cdot d g + \int_{ \mf{U} \cap \mf{D} } \zeta_{a, b; \varepsilon} r^{-1} f g^{ab} \nasla_a z \nasla_b z \cdot d g \\
& \qquad \qquad \quad - \int_{ \mf{U} \cap \mf{D} } \zeta_{a, b; \varepsilon} r^{-1} f g^{CD} \tilde{\nabla}_C z \tilde{\nabla}_D z   \cdot d g.
\end{align*}
Next, we consider the integral term with the null coordinate derivatives. Using \eqref{eq.conformal}, \eqref{eq.conf_comp}, \eqref{eq.conf_d_comp}, and \eqref{eq.deriv_comp} shows that for the \(u\)-derivative term we get the estimate
\begin{align}
\label{eql.carleman_est_14} \int_{ \bar{\Phi} ( \mf{U} \cap \mf{D} ) } \zeta_{a, b} \bar{\rho}^{-1} | u \cdot \partial_u ( \xi^\frac{m+n-2}{2} \bar{z} ) |^2 \cdot d \bar{g} &\geq  C_1 \int_{ \mf{U} \cap \mf{D} } \zeta_{a, b; \varepsilon} r^{-1} | u \cdot \partial_u z |^2 \cdot d g \\
\notag &\qquad - C_2 \varepsilon^2 (m+n)^2 \int_{ \mf{U} \cap \mf{D} } \zeta_{a, b; \varepsilon} r^{-1} u^2 z^2 \cdot d g \text{,}
\end{align}
for some universal constants $C_1, C_2 > 0$.
Similarly, for the \(v\)-derivative, we have
\begin{align}
\label{eql.carleman_est_15} \int_{ \bar{\Phi} ( \mf{U} \cap \mf{D} ) } \zeta_{a, b} \bar{\rho}^{-1} | v \cdot \partial_v ( \xi^\frac{m+n-2}{2} \bar{z} ) |^2 \cdot d \bar{g} &\geq C_1 \int_{ \mf{U} \cap \mf{D} } \zeta_{a, b; \varepsilon} r^{-1} | v \cdot \partial_v z |^2 \cdot d g \\
\notag &\qquad - C_2 \varepsilon^2 (m+n)^2 \int_{ \mf{U} \cap \mf{D} } \zeta_{a, b; \varepsilon} r^{-1} v^2 z^2 \cdot d g \text{.}
\end{align}
Using \eqref{eql.carleman_est_12}--\eqref{eql.carleman_est_15} shows that \eqref{eql.carleman_est_00} reduces to
\begin{align*}
\numberthis \label{eql.carleman_est_20} C \varepsilon & \int_{ \mf{U} \cap \mf{D} } \zeta_{a, b; \varepsilon} r^{-1} [ | u \cdot \partial_u z |^2 + | v \cdot \partial_v z |^2 + f g^{ab} \nasla_a z \nasla_b z - f g^{CD} \tilde{\nabla}_C z \tilde{\nabla}_D z ] \cdot d g \\
& \qquad + C b a^2 \int_{ \mf{U} \cap \mf{D} } \zeta_{a, b; \varepsilon} f^{- \frac{1}{2} } z^2 \cdot d g - \frac{ C'' (m+n)^4 \varepsilon^2 }{a} \int_{ \mf{U} \cap \mf{D} } \zeta_{a, b; \varepsilon} r^{-2} f z^2 \cdot d g \\
& \qquad - C'' \varepsilon^3 (m+n)^2 \int_{ \mf{U} \cap \mf{D} } \zeta_{a, b; \varepsilon} r^{-1} ( u^2 + v^2 ) z^2 \cdot d g \\
 & \qquad \qquad \leqslant \frac{1}{a} \int_{ \mf{U} \cap \mf{D} } \zeta_{a, b; \varepsilon} f | \Box z |^2 \cdot d g + C' \int_{ \partial \mf{U} \cap \mf{D} } \zeta_{a, b; \varepsilon} \xi^2 \mc{N} ( \xi^{-1} f ) | \mc{N} z |^2 \cdot d g \text{,}
\end{align*}
for some positive constants $C, C', C''$.
Next, we will absorb some terms in the LHS of the above equation. By \eqref{eq.carleman_domain}, \eqref{eq.carleman_choices}, and \eqref{eq.uvf_bound}, we get
\begin{align*}
\frac{ \varepsilon^2 (m+n)^4 }{a} \cdot \frac{f}{ r^2 } \ll b a^2 f^{ - \frac{1}{2} } \text{,}\qquad \frac{ \varepsilon^3 (m+n)^2 ( u^2 + v^2 ) }{r}  \ll b a^2 f^{ - \frac{1}{2} } \text{.}
\end{align*}
Thus, both the negative terms in the LHS of \eqref{eql.carleman_est_20} can be absorbed into the second term, which is positive. We also calculate that
\[
\xi^2 \mc{N} ( \xi^{-1} f ) = \xi \mc{N} f - f \mc{N} \xi = ( 1 - \varepsilon r ) \mc{N} f + \varepsilon f \mc{N} r \text{,}
\]
where we used \( \xi = 1- \varepsilon r + \varepsilon^2 f \) from Remark \ref{rmk_xirf}. Finally, putting together the above expression and \eqref{eql.carleman_est_20}, and also accounting for the absorption of the negative terms, we get
\begin{align*}
C \varepsilon & \int_{ \mf{U} \cap \mf{D} } \zeta_{a, b; \varepsilon} \cdot r^{-1} [ | u \cdot \partial_u z |^2 + | v \cdot \partial_v z |^2 + f g^{ab} \nasla_a z \nasla_b z - f g^{CD} \tilde{\nabla}_C z \tilde{\nabla}_D z  ] \cdot d g \\
& \quad + C b a^2 \int_{ \mf{U} \cap \mf{D} } \zeta_{a, b; \varepsilon} \cdot f^{- \frac{1}{2} } z^2 \cdot d g \\
& \qquad \qquad \leqslant \frac{1}{a} \int_{ \mf{U} \cap \mf{D} }  \zeta_{a, b; \varepsilon} \cdot f | \Box z |^2 \cdot d g + C' \int_{ \partial \mf{U} \cap \mf{D} } \zeta_{a, b; \varepsilon} [( 1 - \varepsilon r ) \mc{N} f + \varepsilon f \mc{N}r] | \mc{N} z |^2 \cdot d g \text{,}
\end{align*}
for some constants $C, C' > 0$. This completes the proof of Theorem \ref{thm.carl_bdry}.
\end{proof}

\subsection{Interior Carleman Estimate}
From now on, we assume that our domain is static in time.
We use a type of \emph{hidden regularity} argument adapted to our setting, to derive an interior Carleman estimate from Theorem \ref{thm.carl_bdry}. 

First, let \( \Omega \subset \R^n \) be an open and bounded subset, with smooth boundary \(\partial \Omega\). Define \( \Gamma_+ \subset \partial \Omega \) as follows
\begin{equation}\label{eq.thm.carl_int_1}
\Gamma_+ := \{\nu r>0 \} = \{ x \in \partial \Omega | x\cdot \nu >0 \},
\end{equation}
where \(\nu\) denotes the outward pointing unit normal of \(\Omega\). 

\begin{theorem}[Interior Carleman Estimate]\label{thm.carl_int}
Let \( \mf{U}\) be defined as follows
\[\mf{U} := \R^m \times \Omega,
\]
and fix \( R>0 \) such that
\begin{equation}\label{eq.carl_int_domain}
\Omega \subseteq \{ r<R \} .
\end{equation}
Let \( \varepsilon, a, b >0 \) be constants such that:
\begin{equation}\label{eq.carl_int_choices}
a \geqslant (m+n)^2, \qquad a \gg R, \qquad \varepsilon \ll_{m,n} b \ll R^{-1}.
\end{equation} 
Fix \(\sigma >0 \) and define \( \omega \), a subset of \( \Omega\), as
\begin{equation}
    \omega := \mathcal{O}_\sigma (\Gamma_+) \cap \Omega, \label{eq:3}
\end{equation}
where \[ \mc{O}_\sigma(\Gamma_+) := \{ y \in \R^n : |y-x|< \sigma, \text{ for some } x \in \Gamma_+  \} .\] Then, there exists \(C>0\) such that for any \( z \in \mc{C}^2({\mf{U}})\cap \mc{C}^1(\bar{\mf{U}}) \) satisfying \(z|_{\partial\mf{U} \cap \mf{D}} = 0,\) we have the following estimate
\begin{align*} 
C\varepsilon \int_{\mf{U} \cap \mf{D}} & \zeta_{a,b;\varepsilon}r^{-1}(|u\partial_u z |^2 + |v\partial_v z |^2 + f g^{ab}\slashed\nabla_a z \slashed\nabla_b z - f g^{CD} \tilde\nabla_C z \tilde\nabla _D z  ) + Cba^2\int_{\mf{U} \cap \mf{D}}\zeta_{a,b;\varepsilon} f^{-\frac{1}{2}} z^2 \\
& \leqslant \frac{1}{a} \int_{\mf{U}\cap \mf{D}} \zeta_{a,b;\varepsilon} f |\square z|^2 + a^2 R^3 \int_{( \R^m \times \omega ) \cap \mf{D}} \zeta_{a,b;\varepsilon} f^{-2} |\nabla_t z|^2  \numberthis \label{eq.carl_int}\\
& \qquad + a^4 R^4  \int_{( \R^m \times \omega ) \cap \mf{D}} \zeta_{a,b;\varepsilon} f^{-3} z^2,
\end{align*}
where \( \zeta_{a,b;\varepsilon} \) is defined as in \eqref{eq.carleman_weight}.
\end{theorem}

\begin{remark} Note that, Theorem \ref{thm.carl_bdry} is applicable to a domain with moving boundary as well. But, we assume in Theorem \ref{thm.carl_int} that the domain must have a static boundary. This is primarily done to make the result more suitable for solving the observability problem. However, Theorem \ref{thm.carl_int} can be easily generalised to moving boundary domains.
\end{remark}

\noindent Now we present a lemma which will be used in the proof of the above theorem.
\begin{lemma}\label{thm.der_carl_est}
Assume the hypothesis of Theorem \ref{thm.carl_int}. In the region \(\mf{U}\cap \mf{D}\) we have the following estimate for the Carleman weight
\begin{equation}\label{prf.der_zeta}
|\nabla^\alpha\zeta_{a,b;\varepsilon}| \lesssim a R \zeta_{a,b;\varepsilon} f^{-1},
\end{equation}
where \(\alpha\) represents derivatives in the Cartesian coordinates.
\end{lemma}
\begin{proof}
Using the definitions of \(u,v,f\), we get \( (1 + \varepsilon u) (1-\varepsilon v) = (1- \varepsilon r + \varepsilon^2 f ) \). Then, \eqref{eq.carleman_weight} shows\footnote{Even though \( u,v \) are not smooth at \( t=0 \), \( f \) and \( f (1- \varepsilon r + \varepsilon^2 f )^{-1}\) are smooth.}
\begin{align} 
\label{eq.zeta_der_0}\nabla^\alpha\zeta_{a,b;\varepsilon} & = \nabla^\alpha\Bigg\{\frac{f}{(1- \varepsilon r + \varepsilon^2 f )}\cdot\text{exp} \Bigg[\frac{2bf^{\frac{1}{2}}}{(1- \varepsilon r + \varepsilon^2 f )^\frac{1}{2}}\Bigg] \Bigg\}^{2a}\\
\notag  & = 2a\ \zeta_{a,b;\varepsilon} \Bigg\{ \frac{(1- \varepsilon r + \varepsilon^2 f )}{f} \nabla^\alpha\bigg( \frac{f}{(1- \varepsilon r + \varepsilon^2 f )} \bigg) + \nabla^\alpha \Bigg[\frac{2bf^{\frac{1}{2}}}{(1- \varepsilon r + \varepsilon^2 f )^\frac{1}{2}}\Bigg]\Bigg\}. 
\end{align}
Let us evaluate the terms inside the brackets. For the first term,
\begin{align*}
\bigg|\nabla^\alpha\bigg( \frac{f}{(1- \varepsilon r + \varepsilon^2 f )} \bigg)\bigg| = \bigg|\frac{\nabla^\alpha f}{(1- \varepsilon r + \varepsilon^2 f )} - \frac{f \nabla^\alpha (1- \varepsilon r + \varepsilon^2 f )}{(1- \varepsilon r + \varepsilon^2 f )^2}\bigg| 
\end{align*}
Note that \( | \nabla^\alpha r | \leqslant 1 \), and due to \eqref{eq.f}, \eqref{eq.D}, and \eqref{eq.uvf_bound1} we get \( |\nabla^\alpha f| \lesssim R \). The second term is estimated analogously. Then, the fact that \( \varepsilon \ll_{m,n} b \ll R^{-1}\) shows that \eqref{eq.zeta_der_0} gives us
\begin{align*}
|\nabla^\alpha \zeta_{a,b;\varepsilon}| & \lesssim 2 a \zeta_{a,b;\varepsilon} \left( \frac{(1- \varepsilon r + \varepsilon^2 f ) R}{f}  + \frac{ b R }{ f^\frac{1}{2} } \right) \lesssim a R \zeta_{a,b;\varepsilon} ( f^{-1} + b f^{-\frac{1}{2}}) \lesssim a R  \zeta_{a,b;\varepsilon} f^{-1}. \qedhere
\end{align*}

\end{proof}

\begin{proof}[Proof of Theorem \ref{thm.carl_int}]
Let us assume that the hypothesis of Theorem \ref{thm.carl_int} holds. That is, we choose \( \mf{U}, R, a, b, \varepsilon, \omega, z \) as in the statement. Note that, this choice of the domain and constants also satisfies the hypothesis of Theorem \ref{thm.carl_bdry}, where we note that \eqref{eq.carl_int_domain} implies \eqref{eq.carleman_domain}. Hence, we apply Theorem \ref{thm.carl_bdry} to the above chosen \( \mf{U}, R, a, b, \varepsilon, z \). Recall that \(\mc{N}\) denotes the outer pointing unit normal of \(\mf{U}\). Then, under the present hypothesis, the coefficient present in the boundary term on the RHS of \eqref{eq.carleman_est} reduces to\footnote{In the general case, \( \mc{N} = (\nu^t, \nu^x)= ( \nu^{t_1}, \ldots, \nu^{t_m}, \nu^{x_1}, \ldots, \nu^{x_n}) \). As we are dealing with the static boundary case \(\nu^t=0\), and we get \( \mc{N} = (0, \nu^x) \).}
\[ ( 1 - \varepsilon r ) \mc{N} f + \varepsilon f \mc{N} r \leqslant r \nu r. \]
Using the above observation, we bound the boundary term present in equation \eqref{eq.carleman_est} as follows
\begin{align} 
\notag \int_{\partial\mf{U}\cap \mf{D}}\zeta_{a,b;\varepsilon} r \nu r |\mc{N} z|^2 & \leqslant\int_{( \R^m \times \Gamma_+ )\cap \mf{D}} \zeta_{a,b;\varepsilon} r \nu r |\mc{N} z|^2  \lesssim R \int_{( \R^m \times \Gamma_+ )\cap \mf{D}}\zeta_{a,b;\varepsilon} |\mc{N} z|^2. \label{Carl}
\end{align} 
Our goal is to find a bound for
\[
 R \int_{( \R^m \times \Gamma_+ )\cap \mf{D}}\zeta_{a,b;\varepsilon} |\mc{N} z|^2,
\]
in terms of integral over \( \R^m \times \omega \).
Let \( h\in C^1(\Bar{\Omega};\R^n) \) be a vector field such that \(h = \nu \) on \(\partial\mf{U}\cap \mf{D}\). Define cut-off function \(\rho\in C^2(\Bar{\Omega};[0,1]), \) by
\begin{equation}
\rho(x)
= \begin{cases}
1,\hspace{1cm} x\in\mathcal{O}_{\sigma/3}(\Gamma_+)\cap \Omega,\\
0, \hspace{1cm} x\in\Omega \setminus \mathcal{O}_{\sigma/2}(\Gamma_+).
\end{cases} \label{eq:3.1}
\end{equation} 
Then, using integration by parts shows
\begin{align*}
\int_{\mf{U}\cap \mf{D}}\square  z \rho \zeta_{a,b;\varepsilon} h  z & = \int_{\mf{U}\cap \mf{D}}\nabla^\alpha\nabla_\alpha z (\rho \zeta_{a,b;\varepsilon} h  z) = - \int_{\mf{U}\cap \mf{D}}\nabla_\alpha z\nabla^\alpha (\rho \zeta_{a,b;\varepsilon} h z) + \int_{\partial\mf{U}\cap \mf{D}}\mathcal{N} z \rho \zeta_{a,b;\varepsilon} h z,
\end{align*}
where the integral over \( \mf{U} \cap \partial\mf{D} \) on the RHS vanishes because \(\zeta_{a,b;\varepsilon}\) vanishes on boundary of the null cone. This implies
\begin{equation} \label{eq:c1}
\int_{\partial\mf{U}\cap \mf{D}}\rho\zeta_{a,b;\varepsilon} |\mathcal{N} z|^2 = \int_{\mf{U}\cap \mf{D}}\square z \rho\zeta_{a,b;\varepsilon} h z + \int_{\mf{U}\cap \mf{D}} \nabla_\alpha z\nabla^\alpha ( \rho \zeta_{a,b;\varepsilon} h z). 
\end{equation} 
For the second term in the RHS, we see that
\begin{align*}
 \label{eq:g1} \numberthis & \int_{\mf{U}\cap \mf{D}}  \nabla_\alpha z\nabla^\alpha ( \rho \zeta_{a,b;\varepsilon} h z) \\
& =  \int_{\mf{U} \cap \mf{D}} \rho \zeta_{a,b;\varepsilon} \nabla_\alpha z \nabla^\alpha(h^\beta \nabla_\beta z) + \int_{\mf{U} \cap \mf{D}} \nabla_\alpha z \nabla^\alpha(\rho \zeta_{a,b;\varepsilon}) \cdot hz \\
& = \int_{\mf{U} \cap \mf{D}} \rho \zeta_{a,b;\varepsilon} \nabla_\alpha z \nabla^\alpha h^\beta \nabla_\beta z + \int_{\mf{U} \cap \mf{D}} \rho \zeta_{a,b;\varepsilon} \nabla_\alpha z \cdot h^\beta {\nabla^\alpha}_\beta z + \int_{\mf{U} \cap \mf{D}} \nabla_\alpha z \nabla^\alpha(\rho \zeta_{a,b;\varepsilon}) \cdot hz \\
& = \int_{\mf{U} \cap \mf{D}} \rho \zeta_{a,b;\varepsilon} \nabla^\alpha h^\beta \nabla_\alpha z \nabla_\beta z + \frac{1}{2} \int_{\mf{U} \cap \mf{D}} \rho \zeta_{a,b;\varepsilon} h^\beta \nabla_\beta (\nabla_\alpha z \nabla^\alpha z) + \int_{\mf{U} \cap \mf{D}} \nabla_\alpha z \nabla^\alpha(\rho \zeta_{a,b;\varepsilon}) \cdot hz \\
& =  \int_{\mf{U} \cap \mf{D}} \rho \zeta_{a,b;\varepsilon} \nabla^\alpha h^\beta \nabla_\alpha z \nabla_\beta z - \frac{1}{2} \int_{\mf{U} \cap \mf{D}} \nabla_\beta (\rho \zeta_{a,b;\varepsilon} h^\beta ) \nabla_\alpha z \nabla^\alpha z \\
& \qquad + \frac{1}{2} \int_{\partial \mf{U} \cap \mf{D}} \rho \zeta_{a,b;\varepsilon} |\mathcal{N} z|^2 + \int_{\mf{U} \cap \mf{D}} \nabla_\alpha z \nabla^\alpha(\rho \zeta_{a,b;\varepsilon}) \cdot hz,
\end{align*}
where we used integration by parts, the Dirichlet boundary condition on \(z\), and the definition of \(h\) in the last step. Using \eqref{eq:g1} in \eqref{eq:c1}, and also writing the factor in front of the integral, we get
\begin{align*}
\numberthis \label{eq:g11}\frac{R}{2} \int_{\partial\mf{U}\cap \mf{D}}\rho\zeta_{a,b;\varepsilon} |\mathcal{N} z|^2 & = R \int_{\mf{U}\cap \mf{D}}\square z \rho\zeta_{a,b;\varepsilon} h z + R \int_{\mf{U} \cap \mf{D}} \rho \zeta_{a,b;\varepsilon} \nabla^\alpha h^\beta \nabla_\alpha z \nabla_\beta z \\
& \qquad - \frac{R}{2} \int_{\mf{U} \cap \mf{D}} \nabla_\beta (\rho \zeta_{a,b;\varepsilon} h^\beta ) \nabla_\alpha z \nabla^\alpha z + R \int_{\mf{U} \cap \mf{D}} \nabla_\alpha z \nabla^\alpha(\rho \zeta_{a,b;\varepsilon}) \cdot hz.
\end{align*}
Next, we see that
\begin{equation}
R \int_{\mf{U}\cap \mf{D}}\square z \rho\zeta_{a,b;\varepsilon} h z \lesssim \frac{1}{a} \int_{\mf{U}\cap \mf{D}} \rho \zeta_{a,b;\varepsilon} f |\square z|^2 + a R^2 \int_{\mf{U}\cap \mf{D}} \rho \zeta_{a,b;\varepsilon} f^{-1}| h z|^2. \label{eq:e}
\end{equation}
Now using Lemma \ref{thm.der_carl_est} shows that
\begin{align*}
\nabla_\beta ( \rho\zeta_{a,b;\varepsilon} h^\beta ) & = \nabla_\beta \rho \cdot \zeta_{a,b;\varepsilon} h^\beta + \rho \nabla_\beta \zeta_{a,b;\varepsilon} h^\beta + \rho \zeta_{a,b;\varepsilon} \nabla_\beta h^\beta \lesssim a R (\rho + |h^\beta\nabla_\beta  \rho|) \zeta_{a,b;\varepsilon} f^{-1},
\end{align*} 
and also that \( \nabla^\alpha(\rho \zeta_{a,b;\varepsilon}) \lesssim a R (\rho + |\nabla^\alpha \rho|) \zeta_{a,b;\varepsilon}  f^{-1} \). Using these estimates and \eqref{eq:e} in \eqref{eq:g11} shows
\begin{align*}
R & \int_{\partial\mf{U}\cap \mf{D}}\rho\zeta_{a,b;\varepsilon} |\mathcal{N} z|^2 \\
&\qquad \lesssim  \frac{1}{a} \int_{\mf{U}\cap \mf{D}} \rho \zeta_{a,b;\varepsilon} f |\square z|^2 + a R^2 \int_{\mf{U}\cap \mf{D}} \rho \zeta_{a,b;\varepsilon} f^{-1}| h z|^2 + R \int_{\mf{U} \cap \mf{D}} \rho \zeta_{a,b;\varepsilon} \nabla^\alpha h^\beta \nabla_\alpha z \nabla_\beta z \\
& \qquad \qquad  + a R^2 \int_{\mf{U} \cap \mf{D}} (\rho + |h^\beta\nabla_\beta \rho|) \zeta_{a,b;\varepsilon} f^{-1} \nabla_\alpha z \nabla^\alpha z + a R^2  \int_{\mf{U} \cap \mf{D}} (\rho + |\nabla^\alpha \rho|) \zeta_{a,b;\varepsilon} f^{-1} | \nabla_\alpha z \cdot hz| .
\end{align*}
This implies
\begin{align*}
R \int_{\partial\mf{U}\cap \mf{D}}\rho\zeta_{a,b;\varepsilon} |\mathcal{N} z|^2 & \lesssim  \frac{1}{a} \int_{\mf{U}\cap \mf{D}} \rho \zeta_{a,b;\varepsilon} f |\square z|^2 + a R^2 \int_{\mf{U}\cap \mf{D}} (\rho + |\nabla_x \rho|) \zeta_{a,b;\varepsilon} f^{-1}| \nabla_{t,x} z|^2 , 
\end{align*}
which, after writing out the space and time derivatives separately, reduces to
\begin{align*}
\numberthis \label{eq:g2} R \int_{( \R^m \times \Gamma_+ )\cap \mf{D}} \zeta_{a,b;\varepsilon} |\mathcal{N} z|^2  & \lesssim \frac{1}{a} \int_{\mf{U}\cap \mf{D}}  \zeta_{a,b;\varepsilon} f |\square z|^2 \\
& \qquad + a R^2 \int_{( \R^m \times \mathcal{O}_{\sigma/2}(\Gamma_+) )\cap \mf{D}} \zeta_{a,b;\varepsilon} f^{-1} ( |\nabla_x z|^2 + |\nabla_t z|^2 ), 
\end{align*}
where the integral region in the LHS is obtained by observing that \( \R^m \times \Gamma_+ \subset \partial\mf{U} \), and then using \eqref{eq:3.1}. Now we will estimate the term 
\[ a R^2 \int_{( \R^m \times \mathcal{O}_{\sigma/2}(\Gamma_+) )\cap \mf{D}}\zeta_{a,b;\varepsilon} f^{-1}|\nabla_x z|^2,\]
by data restricted to the interior region \(\omega\). For this purpose, define the function \[\eta(t,x):=\rho_1^2 \zeta_{a,b;\varepsilon} f^{-1},\] where  \( \rho_1\in C^2(\Bar{\Omega};[0,1]) \) is the cut-off function given by
\begin{equation}
\rho_1(x) = \begin{cases}
1, \qquad x\in\mathcal{O}_{\sigma/2}(\Gamma_+)\cap \Omega,\\
0, \qquad x\in \Omega\setminus \omega.
\end{cases} \label{eq:rho_1}
\end{equation}
Then using integration by parts, we have
\[ \int_{\mf{U}\cap \mf{D}} \eta z \square z = - \int_{\mf{U}\cap \mf{D}} z \nabla^\alpha \eta \nabla_\alpha z - \int_{\mf{U}\cap \mf{D}} \eta \nabla^\alpha z \nabla_\alpha z, \]
where the boundary terms vanish because \( z=0 \) on \( \partial \mf{U} \cap \mf{D} \), and \( \zeta_{a,b;\varepsilon} = 0 \) on \( \mf{U} \cap \partial \mf{D} \). Then, we have
\begin{align*}
\int_{\mf{U}\cap \mf{D}}\eta |\nabla_x z|^2 & = - \int_{\mf{U}\cap \mf{D}}\eta z\square z + \int_{\mf{U}\cap \mf{D}} z \nabla_t z \cdot \nabla_t \eta + \int_{\mf{U}\cap \mf{D}} \eta |\nabla_t z|^2  - \int_{\mf{U}\cap \mf{D}} z \nabla_x z \cdot \nabla_x \eta. 
\end{align*} 
From the above it follows that
\begin{align*}
\int_{\mf{U}\cap \mf{D}}\rho_1^2 \zeta_{a,b;\varepsilon} f^{-1} |\nabla_x z|^2 & \leq  \int_{\mf{U}\cap \mf{D}}|\eta z\square z| + \int_{\mf{U}\cap \mf{D}}|z \nabla_t z \cdot \nabla_t \eta| + \int_{\mf{U}\cap \mf{D}} \eta |\nabla_tz|^2 - \int_{\mf{U}\cap \mf{D}} z \nabla_x z \cdot \nabla_x \eta.
\end{align*}
We can calculate that 
\[ |\nabla_x \eta| = |\nabla_x (\rho_1^2 \zeta_{a,b;\varepsilon} f^{-1})| \lesssim \rho_1 \zeta_{a,b;\varepsilon} \left[ f^{-1} |\nabla_x \rho_1| + a R \rho_1 f^{-2} \right], \]
which implies that the latest integral estimate reduces to 
\begin{align*}
\int_{\mf{U}\cap \mf{D}} \rho_1^2 \zeta_{a,b;\varepsilon} f^{-1} |\nabla_x z|^2 & \lesssim \int_{\mf{U}\cap \mf{D}}\rho_1^2\zeta_{a,b;\varepsilon} f^{-1} | z\square z| + a R \int_{\mf{U}\cap \mf{D}} \rho_1^2 \zeta_{a,b;\varepsilon} f^{-2} | z \nabla_t z|\\
& \qquad + \int_{\mf{U}\cap \mf{D}}\rho_1^2\zeta_{a,b;\varepsilon} f^{-1} |\nabla_t z|^2 +\int_{\mf{U}\cap \mf{D}} \rho_1 \zeta_{a,b;\varepsilon} \left[ f^{-1} |\nabla_x \rho_1| + a R \rho_1 f^{-2} \right] | z \nabla_x z |,
\end{align*}
where we used \eqref{prf.der_zeta} for estimating \(\nabla_t \eta\).
Also considering the constant present in \eqref{eq:g2}
\begin{align*}
\numberthis \label{eq:g100} a R^2 & \int_{\mf{U}\cap \mf{D}}  \rho_1^2 \zeta_{a,b;\varepsilon} f^{-1} |\nabla_x z|^2 \\ 
& \lesssim  a R^2 \int_{\mf{U}\cap \mf{D}}\rho_1^2\zeta_{a,b;\varepsilon} f^{-1} | z\square z| + a^2 R^3  \int_{\mf{U}\cap \mf{D}}\rho_1^2 \zeta_{a,b;\varepsilon} f^{-2} |z \nabla_t z | \\
& \qquad + a R^2 \int_{\mf{U}\cap \mf{D}}\rho_1^2\zeta_{a,b;\varepsilon} f^{-1}|\nabla_t z|^2 + a R^2 \int_{\mf{U}\cap \mf{D}} \rho_1 \zeta_{a,b;\varepsilon} \left[ f^{-1} |\nabla_x \rho_1| + a R \rho_1 f^{-2} \right] | z \nabla_x z |.
\end{align*}
Applying the Cauchy-Schwarz inequality for the last term in the RHS of the above inequality shows 
\begin{align*}
a R^2 \int_{\mf{U}\cap \mf{D}} & \rho_1 \zeta_{a,b;\varepsilon} \left[ f^{-1} |\nabla_x \rho_1| + a R \rho_1 f^{-2} \right] | z \nabla_x z | \\
& \lesssim a R^2 \int_{\mf{U}\cap \mf{D}}  \rho_1 \zeta_{a,b;\varepsilon} f^{-1} |\nabla_x\rho_1| | z \nabla_x z | + a^2 R^3 \int_{\mf{U}\cap \mf{D}} \rho_1^2 \zeta_{a,b;\varepsilon} f^{-2} | z \nabla_x z |\\
& \lesssim a^2 R^2 \int_{\mf{U}\cap \mf{D}} |\nabla_x\rho_1|^2 \zeta_{a,b;\varepsilon} f^{-1} z^2 + R^2 \int_{\mf{U}\cap \mf{D}} \rho_1^2 \zeta_{a,b;\varepsilon} f^{-1} |\nabla_x z|^2 + a^4 R^4 \int_{\mf{U}\cap \mf{D}} \rho_1^2 \zeta_{a,b;\varepsilon} f^{-3} z^2.
\end{align*}
Substituting this into \eqref{eq:g100}, and using Cauchy-Schwarz for the other terms in the RHS of \eqref{eq:g100}, we get
\begin{align*}
\numberthis \label{eq:g123} a R^2 & \int_{ \mf{U} \cap \mf{D} } \rho_1^2 \zeta_{a,b;\varepsilon} f^{-1} |\nabla_x z|^2 \\ 
& \lesssim \frac{1}{a} \int_{\mf{U}\cap \mf{D}} \rho_1^2 \zeta_{a,b;\varepsilon} f |\square z|^2 + a^3 R^4 \int_{\mf{U}\cap \mf{D}} \rho_1^2 \zeta_{a,b;\varepsilon} f^{-3} z^2 + a^2 R^3 \int_{\mf{U}\cap \mf{D}} \rho_1^2 \zeta_{a,b;\varepsilon} f^{-2} |\nabla_t z|^2  \\
& \qquad + a^2 R^3 \int_{\mf{U}\cap \mf{D}} \rho_1^2 \zeta_{a,b;\varepsilon} f^{-2} z^2 + a R^2 \int_{\mf{U}\cap \mf{D}} \rho_1^2 \zeta_{a,b;\varepsilon} f^{-1} |\nabla_t z|^2  \\
& \qquad + a^2 R^2 \int_{\mf{U}\cap \mf{D}} |\nabla_x\rho_1|^2 \zeta_{a,b;\varepsilon} f^{-1} z^2 + a^4 R^4 \int_{\mf{U}\cap \mf{D}} \rho_1^2 \zeta_{a,b;\varepsilon} f^{-3} z^2 + R^2 \int_{\mf{U}\cap \mf{D}} \rho_1^2 \zeta_{a,b;\varepsilon} f^{-1} |\nabla_x z|^2.
\end{align*}
The last term on the RHS of the above estimate is absorbed into the LHS. 
Next, if we consider the coefficients of \( z^2 \) without \(\zeta_{a,b;\varepsilon}\), we see that
\begin{align*}
a^3 R^4 \rho_1^2 f^{-3} + & a^2 R^3 \rho_1^2 f^{-2} + a^2 R^2 |\nabla_x \rho_1|^2 f^{-1} + a^4 R^4 \rho_1^2 f^{-3} \\
& \lesssim a^4 R^4 \rho_1^2 f^{-3} + a^2 R^3 \rho_1^2 f^{-3} f + a^2 R^2 |\nabla_x \rho_1|^2  f^{-3} f^2\\
& \lesssim a^4 R^4 \left[ |\nabla_x \rho_1|^2 + \rho_1^2 \right]\text{,}
\end{align*}
where we also used \eqref{eq.uvf_bound1} and \eqref{eq.carl_int_choices}. We use similar ideas for the coefficients of \(\nabla_t z\). Then, \eqref{eq:g123} implies that
\begin{align*}
\numberthis \label{eq:g125} a R^2 \int_{ \mf{U} \cap \mf{D} } \rho_1^2 \zeta_{a,b;\varepsilon} f^{-1} |\nabla_x z|^2 & \lesssim  \frac{1}{a} \int_{\mf{U}\cap \mf{D}} \rho_1^2 \zeta_{a,b;\varepsilon} f |\square z|^2 +  a^2 R^3 \int_{\mf{U}\cap \mf{D}} \rho_1^2 \zeta_{a,b;\varepsilon} f^{-2} |\nabla_t z|^2 \\
& \qquad  +  a^4 R^4 \int_{\mf{U}\cap \mf{D}} \left[ |\nabla_x \rho_1|^2 + \rho_1^2 \right] \zeta_{a,b;\varepsilon} f^{-3} z^2.
\end{align*}
Next, we  use \eqref{eq:rho_1}  to see that \eqref{eq:g125} simplifies to
\begin{align*}
a R^2 \int_{( \R^m \times \mathcal{O}_{\sigma/2}(\Gamma_+) )\cap \mf{D} } \rho_1^2 \zeta_{a,b;\varepsilon} f^{-1} |\nabla_x z|^2 & \lesssim \frac{1}{a} \int_{\mf{U}\cap \mf{D}} \zeta_{a,b;\varepsilon} f |\square z|^2\\
& \qquad  + a^2 R^3 \int_{( \R^m \times \omega ) \cap \mf{D}} \zeta_{a,b;\varepsilon} f^{-2} |\nabla_t z|^2 \\
& \qquad + a^4 R^4  \int_{( \R^m \times \omega ) \cap \mf{D}} \zeta_{a,b;\varepsilon} f^{-3} z^2.
\end{align*}
We put together the above estimate and \eqref{eq:g2} to get
\begin{align*}
R \int_{ \R^m \times \Gamma_+ } \zeta_{a,b;\varepsilon} |\mathcal{N} z|^2 & \lesssim \frac{1}{a} \int_{\mf{U}\cap \mf{D}} \zeta_{a,b;\varepsilon} f |\square z|^2 + a^2 R^3 \int_{( \R^m \times \omega ) \cap \mf{D}} \zeta_{a,b;\varepsilon} f^{-2} |\nabla_t z|^2 \\
& \qquad + a^4 R^4  \int_{( \R^m \times \omega ) \cap \mf{D}} \zeta_{a,b;\varepsilon} f^{-3} z^2  .
\end{align*}
Finally, using this estimate to bound the boundary term in \eqref{eq.carleman_est} concludes the proof of \eqref{eq.carl_int}.
\end{proof}

\section{Observability Estimate Results}\label{sec Carl_obs}
In this section, we will prove the main observability result Theorem \ref{intro_thm_obs}. Hence, we go back to the setting of the PDE problem from Section \ref{sec Setting}, and consider the original observability problem in the space-time domain \(\R^{1+n}\). We will prove two observability estimates: exterior and interior. 

\begin{itemize}
\item \textit{Exterior observability:} This estimate is obtained by applying the Carleman estimate with the centre point lying outside the domain. This will address the case of Theorem \ref{intro_thm_obs} when \(x_0 \notin \bar{\Omega}\).

\item \textit{Interior observability:} This estimate is obtained by applying the Carleman estimate with the centre point lying inside the domain. We will see that we need to apply the Carleman estimate around two centre points. This observability estimate will help us to prove the case of Theorem \ref{intro_thm_obs} when \(x_0 \in \Omega\).
\end{itemize}
We will prove the above two observability estimates in Section \ref{ssec_extobs} and Section \ref{ssec_intobs}, respectively. Then, in Section \ref{ssec_proofofobsthm} we will present the proof of Theorem \ref{intro_thm_obs}.
Henceforth we will assume, without any loss of generality, that \(x_0 = 0 \).

\subsection{Exterior Observability}\label{ssec_extobs}
Here we will consider the case when the observation point lies outside the domain.

\begin{definition}
Define the region \(\mc{D}\) as
\begin{equation}
\mc{D} := \{ |x|^2 > t^2 \} \subset \R^{1+n}.
\end{equation}
\end{definition}
The key improvement (to the standard Carleman based observability results) in this article is related to the region \( \mc{D}\). Specifically, we will show that the observation region is restricted to the region \(\mc{D}\).

\begin{theorem}\label{thm.obs_ext}
Let \( \Omega \subset \R^n \) and consider the setting of system \eqref{eq.wave_obs}. Let \(\mc{U} \cap \mc{D}\) be bounded.  Assume that the observation point \( 0 \notin \bar{\Omega} \).  Further,
\begin{itemize}

\item Define the following constants:
\begin{align}
\label{eq.obs_ext_MR} \notag R_+ := \sup_{\Omega} r \text{,} &\qquad R_- := \inf_{ \Omega } r, \\
M_0 := \sup_{ \mc{U}  } | V | \text{,} &\qquad M_1 := \sup_{ \mc{U} } \left\{ \frac{| \mc{X}^{ t, x } |}{\sqrt{R_+}}, | \nabla_{t,x}  \mc{X}^{ t, x } | \right\} \text{,} \\
\notag  M := & \max \{ 1, M_0, M_1 \},
\end{align}
where \( \mc{X}^{t,x} \) represents the Cartesian components of \(\mc{X}\).

\item Let \(\nu\) be the outward pointing unit normal to \(\Omega\), and let
\begin{equation}
\label{eq.obs_ext_delta} \Gamma_+ :=  \{ \nu r > 0 \} = \{ x \in \partial \Omega | x\cdot \nu >0 \} \text{.}
\end{equation}

\item Let \(\sigma> 0.\) Define \(\omega \subset \Omega\) as
\begin{equation}\label{eq.obs_ext_omega}
\omega := \mc{O}_\sigma(\Gamma_+) \cap \Omega.
\end{equation}
\end{itemize}
Let \( T > R_+ \). Then, there exist constants \( C_1\) and \(C_2\) such that 
\begin{equation}\label{eq.obs_ext_est}
|| \phi_0 ||_{L^2(\Omega)}^2 + || \phi_1 ||_{H^{-1}(\Omega)}^2 \leqslant \frac{ C_1 a^4 2^{12a} M }{ \delta^2 R_+^2 } \left( \frac{ R_+}{R_-} \right)^{4a+5}  e^{ 2 C_2 M T } \int_{ ( (-T,T) \times \omega ) \cap \mc{D}} |\phi(t,x)|^2,
\end{equation}
holds for any solution \(\phi \in C^2 ( \mc{U} ) \cap C^1 ( \bar{\mc{U}}) \) of \eqref{eq.wave_obs}.
\end{theorem}

\begin{remark}
Due to the exterior observability assumption  \( (0 \notin \bar{\Omega} )\), we see that \( R_- > 0 \).
\end{remark}

\subsubsection{Energy Results:} Now, we present some results related to the energy of the system \eqref{eq.wave_obs}. This will be used towards the end of the proof of observability estimate. Let \(\prime\) denote differentiation with respect to the time component of \(\phi\). We have the following results:
\begin{lemma}\label{thm.energy1}
Let \( s_1 < s_2 < t_2 < t_1 \). Then there exists a constant  \( C_1>0, \) such that
\begin{equation}\label{eq.energy1}
\int_{s_2}^{t_2} |\phi'(t, \cdot)|^2_{H^{-1}(\Omega)} \leqslant C_1(M_0 + M_1) \int_{s_1}^{t_1}|\phi(t,\cdot)|^2_{L^2(\Omega)} dt,
\end{equation}
where \(\phi\) is the transposition solution of the system \eqref{eq.wave_obs}.
\end{lemma}
\begin{lemma}\label{thm.energy2}
Let \(\phi\) be the transposition solution of \eqref{eq.wave_obs} Define the energy of the system to be 
\begin{equation}
E(t) := \frac{1}{2} ( |\phi(t,\cdot)|^2_{L^2(\Omega)} + |\phi'(t,\cdot)|^2_{H^{-1}(\Omega)} ).
\end{equation}
Then, there is a constant \( C_2>0 \) such that the following is satisfied for all \(s, t \in \R \) 
\begin{equation}\label{eq.energy2}
E(s) \leqslant e^{C_2(M_0 + M_1) |t-s|} E(t).
\end{equation}
\end{lemma}
Lemma \ref{thm.energy1} and Lemma \ref{thm.energy2} are proved using standard energy estimate arguments, and are analogous to the ones presented in \cite{Zhang}.

\noindent\textbf{Proof of Theorem \ref{thm.obs_ext}.} Now we will present the proof of Theorem \ref{thm.obs_ext}. First, we provide an outline of the proof:
\begin{enumerate}

\item We define a new function \(z\) to which we will apply the interior Carleman estimate from Theorem \ref{thm.carl_int}. The function \(z\) will have better regularity than \(\phi\).

\item We apply the Carleman estimate to \(z\). Then we deal with absorbing the necessary terms, as well as taking care of the spatial derivatives of \(z\). 

\item We go back from \( z\) to the original function \(\phi\).

\item We use energy estimate results for \(\phi\) to prove the observability estimate \eqref{eq.obs_ext_est}.

\end{enumerate}
 
We will use a special case of the setting and the result mentioned in Theorem \ref{thm.carl_int}, namely the case when \( \boldsymbol{m=2} \). Hence, in the following discussion we have \( \mf{U} = \R \times \mc{U} = \R^2 \times \Omega \). The corresponding definitions and notations are written analogously.

\vspace{-1mm}

\subsubsection{ Defining a new function \(z\).}\label{sssec_zdef}
Let \(\phi\) be the solution of \eqref{eq.wave_obs}. Since we are using the special case \(m=2\), we get
\begin{equation*}
\mf{D} := \{f>0\} = \{ r^2 > t_1^2 + t_2^2 \} \subset \R^{2+n}.
\end{equation*}
We will make several applications of this property of \(\mf{D}\) throughout the proof.

\noindent Next, define a new function \(z\) as follows
\begin{equation}\label{eq.z_def}
z(t_1,t_2,x) := \int_{t_1}^{t_2} \phi(s,x)ds, \quad \forall\ (t_1,t_2,x) \in \mf{U}.
\end{equation}
From the above expression, we can calculate the corresponding derivatives for \(z\) as follows
\begin{align}\label{eq.z_deriv}
\notag \Delta z(t_1,t_2,x) & = \int_{t_1}^{t_2} \Delta \phi(s,x)ds,\\
z_{t_1}(t,t_2,x)= -\phi(t,x),& \qquad z_{t_2}(t_1,t,x)= \phi(t,x), \\
\notag z_{t_1t_1}(t,t_2,x)= -\phi'(t,x),& \qquad z_{t_2t_2}(t_1,t,x)= \phi'(t,x).
\end{align}
Using the wave equation satisfied by \(\phi\), we get the following equation for \(z\)
\begin{align*} \numberthis \label{eq.z_wave}
-\square z(t_1,t_2,x) & := z_{t_1t_1}(t_1,t_2,x) + z_{t_2t_2}(t_1,t_2,x) - \Delta z(t_1,t_2,x) \\
& \ =  \int_{t_1}^{t_2} ( V(s,x) z_{t_2}(t_1,s,x) + \nabla_{\mc{X}(s,x)}z_{t_2}(t_1,s,x) ) ds.
\end{align*}

\subsubsection{Application of Carleman estimate.}
Let \( \delta \ll 1 \). Then, we make the following choice of constants: choose \( a \geq (n+2)^2\) such that
\begin{equation}\label{eql.obs_ext_a} 
a \gg \delta^{-2} R_-^{-2} M^2R_+^5, \quad a \gg R_-^{-3}M R_+^2, \quad a \gg M^2 R_+^3, \quad a \gg \delta^{-1} M^2 R_+^5, \quad a \gg M R_+
\end{equation}
and choose \(\varepsilon \) and \(b\) such that
\begin{equation}\label{eql.obs_ext_b}
\varepsilon:= \delta^2 R_+^{-1}, \qquad b:= \delta R_+^{-1}.
\end{equation}
Now we present a lemma which shows that \(\zeta_{a,b;\varepsilon} f\) is a decreasing function of \(t^2\). This result will be used later, while taking care of the extra time integral that was introduced earlier.

\begin{lemma}\label{thm.f/xi_inc}
Let \(\xi:= (1 + \varepsilon u) (1 - \varepsilon v) = 1- \varepsilon r + \varepsilon^2 f\). Then we have the following property
\begin{equation}\label{eq.f/xi}
t^2 < t_1^2 \quad \Rightarrow \quad \frac{f}{\xi} (t, t_2, x) > \frac{f}{\xi}(t_1, t_2, x).
\end{equation}
Moreover, we also have
\begin{equation}\label{eq.zetaf/xi}
(\zeta_{a,b;\varepsilon} f) (t, t_2, x) > (\zeta_{a,b;\varepsilon} f) (t_1, t_2, x).
\end{equation}
\end{lemma}
\noindent \begin{proof}
Since \eqref{eq.zetaf/xi} can be directly obtained from \eqref{eq.f/xi} and the definition of \(\zeta_{a,b;\varepsilon}\), we only need to check \eqref{eq.f/xi}. For this purpose, note the following calculation
\begin{align*}
\frac{f}{\xi} = \frac{f}{1- \varepsilon r + \varepsilon^2 f} & = \frac{\left[ f + \frac{(1 - \varepsilon r ) }{ \varepsilon^2} \right] - \frac{( 1- \varepsilon r )}{\varepsilon^2}}{1- \varepsilon r + \varepsilon^2 f} = \frac{1}{\varepsilon^2} - \frac{( 1- \varepsilon r )}{\varepsilon^2} . \frac{1}{ ( 1- \varepsilon r + \varepsilon^2 f ) } .
\end{align*}
Due to \eqref{eql.obs_ext_b}, we have \( (1 - \varepsilon r ) > 0\). This proves \eqref{eq.f/xi} and completes the proof of the lemma.
\end{proof}

Now we will apply the Carleman estimate given by Theorem \ref{thm.carl_int} to the function \(z\) and the region \( \mf{U} \cap \mf{D} \), with the above chosen constants \(a, b, \varepsilon\). The previous assumptions in this section ensure that the hypotheses of Theorem \ref{thm.carl_int} are satisfied by this choice. In particular, since \(\mf{U} \cap \mf{D}\) is bounded, \eqref{eq.carl_int_domain} is satisfied with \(R:=R_+\). Using \eqref{eq.uvf_bound1} and \eqref{eql.obs_ext_b}, we see that \eqref{eq.carl_int} reduces to
\begin{align*}
\frac{ C \delta^2 }{R_+^2}  \int_{\mf{U}\cap \mf{D}} & \zeta_{a,b;\varepsilon} (|u\partial_u z |^2 +  |v\partial_v z |^2 + f g^{ab}\slashed\nabla_a z \slashed\nabla_b z - f g^{CD} \tilde\nabla_C z \tilde\nabla _D z  ) +\frac{ C \delta a^2}{R_+^2} \int_{\mf{U}\cap \mf{D}}\zeta_{a,b;\varepsilon} z^2  \\
\label{eq.carl_z} \numberthis & \leqslant \frac{1}{a}\int_{\mf{U} \cap \mf{D}} \zeta_{a,b;\varepsilon} f |\square z|^2 +  a^2 R_+^3 \int_{( (-T,T)^2 \times \omega ) \cap \mf{D} } \zeta_{a,b;\varepsilon} f^{-2} ( z_{t_1}^2 + z_{t_2}^2 )  \\
& \qquad + a^4 R_+^4 \int_{( (-T,T)^2 \times \omega ) \cap \mf{D}} \zeta_{a,b;\varepsilon} f^{-3} z^2,
\end{align*}
where the integral region for the last two terms on the RHS is obtained by noting that, since \(T> R_+\), we get
\[ ( \R^2 \times \omega  ) \cap \mf{D} = ( (-T,T)^2 \times \omega ) \cap \mf{D}.\]
First we deal with the bulk term, which due to \eqref{eq.z_wave} is as follows
\begin{equation}\label{eq.carl_square_z}
\int_{\mf{U} \cap \mf{D}} \zeta_{a,b;\varepsilon}f |\square z|^2 = \int_{\mf{U} \cap \mf{D}} \zeta_{a,b;\varepsilon}f \left( \int_{t_1}^{t_2} \{ V(s, x) z_{t_2}(t_1,s,x) + \nabla_{\mc{X}(s, x)}z_{t_2}(t_1,s,x)\} ds \right)^2.
\end{equation}
For the integral corresponding to the potential term, using \eqref{eq.z_deriv} shows 
\begin{align}
\notag \int_{\mf{U}\cap\mf{D}} \zeta_{a,b;\varepsilon}f & \left( \int_{t_1}^{t_2} V(s, x) z_{t_2}(t_1,s,x)ds \right)^2 dt_1 dt_2 dx \\
\notag & = \int_{\mf{U}\cap\mf{D}} \zeta_{a,b;\varepsilon}f \left( \int_{t_1}^0  V(s, x) z_{t_2}(t_1,s,x) ds  + \int_0^{t_2}  V(s, x) z_{t_2}(t_1,s,x) ds \right)^2 dt_1 dt_2 dx \\
\notag & = \int_{\mf{U}\cap\mf{D}} \zeta_{a,b;\varepsilon}f \left( \int_0^{t_1}  V(s, x) z_{t_1}(s,t_2,x) ds  + \int_0^{t_2}  V(s, x) z_{t_2}(t_1,s,x) ds \right)^2 dt_1 dt_2 dx\\
& \leqslant M_0^2 \int_{\mf{U}\cap\mf{D}} \zeta_{a,b;\varepsilon}f \left( \Big| \int_0^{t_1} |z_{t_1}(s,t_2,x)|^2 ds \Big| + \Big| \int_0^{t_2} |z_{t_2}(t_1,s,x)|^2 ds \Big| \right) dt_1 dt_2 dx. \label{eq.z_Vz}
\end{align}
Our next goal is to take care of the extra time integral. If we only consider the term with \\ \(|z_{t_1}(s,t_2,x)|\) in the above inequality, after an application of Lemma \ref{thm.f/xi_inc}, we get
\begin{align}
\notag  M_0^2 \int_{\mf{U}\cap\mf{D}} & \zeta_{a,b;\varepsilon}f  \Big| \int_0^{t_1} |z_{t_1}(s,t_2,x)|^2 ds \Big| dt_1 dt_2 dx  \\
\notag & =  M_0^2 \int_{\Omega} dx \int_{-r}^r dt_2 \bigg\{ \int_{- \sqrt{r^2 - t_2^2} }^0 dt_1 \bigg(  \zeta_{a,b;\varepsilon} f(t_1,t_2,x)  \int_{t_1}^0 |z_{t_1}(s,t_2,x)|^2  ds \bigg) \\
\notag & \hspace{3.2cm}  + \int_0^{\sqrt{r^2 - t_2^2}} dt_1 \bigg( \zeta_{a,b;\varepsilon} f(t_1,t_2,x)  \int_0^{t_1} |z_{t_1}(s,t_2,x)|^2  ds \bigg) \bigg\} \\
\notag & \leqslant M_0^2 \int_{\Omega} dx \int_{-r}^r dt_2 \bigg\{ \int_{-\sqrt{r^2 - t_2^2}}^0 dt_1 \bigg( \int_{t_1}^0  \zeta_{a,b;\varepsilon} f(s,t_2,x)  |z_{t_1}(s,t_2,x)|^2  ds \bigg) \\
\notag & \hspace{3.2cm} + \int_0^{\sqrt{r^2 - t_2^2}} dt_1 \bigg( \int_0^{t_1}  \zeta_{a,b;\varepsilon} f(s,t_2,x)  |z_{t_1}(s,t_2,x)|^2  ds \bigg) \bigg\} \\
\notag & \leqslant  M_0^2 \int_{\Omega} r dx \int_{-r}^r dt_2  \bigg\{ \int_{-\sqrt{r^2 - t_2^2}}^0  \zeta_{a,b;\varepsilon} f(s,t_2,x)  |z_{t_1}(s,t_2,x)|^2  ds \\
\notag & \hspace{3.2cm} + \int_0^{\sqrt{r^2 - t_2^2}}  \zeta_{a,b;\varepsilon} f(s,t_2,x)  |z_{t_1}(s,t_2,x)|^2  ds \bigg\}\\
\notag & \leqslant  M_0^2 R_+ \int_{\Omega} dx \int_{-r}^r dt_2 \int_{-\sqrt{r^2 - t_2^2}}^{\sqrt{r^2 - t_2^2}} \zeta_{a,b;\varepsilon} f(s,t_2,x)  |z_{t_1}(s,t_2,x)|^2  ds \\
& \leqslant M_0^2 R_+\int_{\mf{U}\cap\mf{D}} \zeta_{a,b;\varepsilon}f  |z_{t_1}(t_1,t_2,x)|^2 dt_1 dt_2 dx. \label{eq.z_Vz1}
\end{align}
Similarly, we can also get the same estimate for the integral term with \( |z_{t_2}(t_1,s,x) | \) in \eqref{eq.z_Vz},
\begin{equation}\label{eq.z_Vz2}
 M_0^2 \int_{\mf{U}\cap\mf{D}} \zeta_{a,b;\varepsilon}f  \Big| \int_0^{t_1} |z_{t_2}(t_1,s,x)|^2 ds \Big| \leqslant  M_0^2 R_+ \int_{\mf{U}\cap\mf{D}} \zeta_{a,b;\varepsilon}f  |z_{t_2}(t_1,t_2,x)|^2 dt_1 dt_2 dx.
\end{equation}
Combining \eqref{eq.z_Vz}-\eqref{eq.z_Vz2}, gives us 
\begin{equation}\label{eq.z_V}
 \int_{\mf{U}\cap\mf{D}} \zeta_{a,b;\varepsilon}f  \left( \int_{t_1}^{t_2} V(s, x) z_{t_2}(t_1,s,x)ds \right)^2 \leqslant M_0^2 R_+ \int_{\mf{U}\cap\mf{D}} \zeta_{a,b;\varepsilon}f (z_{t_1}^2 + z_{t_2}^2) dt_1 dt_2 dx.
\end{equation}
Next, we consider the first order term in \eqref{eq.carl_square_z}. Note that, the vector field we have is \\
\( \mc{X} = \mc{X}(s,x)\). Hence, we let \( p\) represent components in \( (s, x) \). Working similar to \eqref{eq.z_Vz}, we get
\begin{align*}
\int_{\mf{U} \cap \mf{D}} & \zeta_{a,b;\varepsilon}f \left( \int_{t_1}^{t_2} \nabla_{\mc{X}(s, x)}z_{t_2}(t_1,s,x) ds \right)^2 \\
& = \int_{\mf{U} \cap \mf{D}} \zeta_{a,b;\varepsilon}f \left( \int_{t_1}^{t_2} \mc{X}^p(s,x) \nabla_p z_{t_2}(t_1,s,x) ds \right)^2  dt_1 dt_2 dx \\
& = \int_{\mf{U}\cap\mf{D}} \zeta_{a,b;\varepsilon}f \left( \int_0^{t_1} \mc{X}^p(s,x) \nabla_p z_{t_1}(s,t_2,x) ds + \int_0^{t_2} \mc{X}^p(s,x) \nabla_p z_{t_2}(t_1,s,x) ds \right)^2 dt_1 dt_2 dx. \numberthis \label{eq.z_X1}
\end{align*}
For the \(z_{t_1} \) term in the RHS of the above equation, using integration by parts shows
\begin{align*}
\int_0^{t_1} \mc{X}^p(s,x) \nabla_p z_{t_1}(s,t_2,x) ds = -\int_0^{t_1} \partial_s \mc{X}^p(s,x) \nabla_p z (s,t_2,x) ds + \mc{X}^p(s,x) \nabla_p z (s,t_2,x) \big|_0^{t_1}.
\end{align*}
We get an analogous identity for the \(z_{t_2}\) term. Then, using both of these in \eqref{eq.z_X1} shows
\begin{align*}
\numberthis \label{eq.z_X2} \int_{\mf{U} \cap \mf{D}} & \zeta_{a,b;\varepsilon}f  \left( \int_{t_1}^{t_2} \nabla_{\mc{X}(s,x)}z_{t_2}(t_1,s,x) ds \right)^2 \\
& = \int_{\mf{U} \cap \mf{D}} \zeta_{a,b;\varepsilon}f \bigg( -\int_0^{t_1} \partial_s \mc{X}^p(s,x) \nabla_p z (s,t_2,x) ds - \int_0^{t_2} \partial_s \mc{X}^p(s,x) \nabla_p z (t_1,s,x) ds \\
& \qquad \qquad \qquad \ \qquad + \mc{X}^p(s,x) \nabla_p z (s,t_2,x) \big|_0^{t_1} + \mc{X}^p(s,x) \nabla_p z (t_1,s,x) \big|_0^{t_2} \bigg)^2 \\
&  \leqslant M_1^2 \int_{\mf{U}\cap\mf{D}} \zeta_{a,b;\varepsilon}f \left( \Big| \int_0^{t_1} |\nabla_p z(s,t_2,x)|^2 ds \Big| + \Big| \int_0^{t_2} |\nabla_p z(t_1,s,x)|^2 ds \Big| \right) \\
& \qquad + 2 M_1^2 R_+ \int_{\mf{U}\cap\mf{D}} \zeta_{a,b;\varepsilon}f |\nabla_{t_1,t_2,x} z(t_1,t_2,x)|^2 , 
\end{align*}
where we used \eqref{eq.obs_ext_MR} and \eqref{eq.z_deriv}, to note that
\begin{align*}
& \mc{X}^p(s,x) \nabla_p z(s,t_2,x) \big|_0^{t_1} +  \mc{X}^p(s,x) \nabla_p z (t_1,s,x) \big|_0^{t_2} \leqslant 2 M_1^2 R_+ |\nabla z_{t_1,t_2,x}(t_1, t_2, x)|,
\end{align*}
where the \(s=0\) terms from the two expressions cancel out.
Putting together \eqref{eq.z_X1} and \eqref{eq.z_X2}, and applying the method we used for deriving \eqref{eq.z_V}, we get
\begin{align}
\label{eq.z_X} \int_{\mf{U}\cap\mf{D}} \zeta_{a,b;\varepsilon}f \bigg| \int_{t_1}^{t_2} \nabla_\mc{X} z_{t_2}(t_1,s, x) ds \bigg|^2 \leqslant 3 M_1^2 R_+ \int_{\mf{U}\cap\mf{D}} \zeta_{a,b;\varepsilon}f |  \nabla_{t_1,t_2,x} z (t_1,t_2,x) |^2 dt_1 dt_2 dx.
\end{align}

% \\ & \quad = \mc{X}^p(t_1,x) \nabla_p z(t_1,t_2,x) - \mc{X}^p(0,x) \nabla_p z(0,t_2,x) + \mc{X}^p(t_2,x) \nabla_p z(t_1,t_2,x) - \mc{X}^p(0,x) \nabla_p z(t_1,0,x)  \\ & \quad

\noindent Using \eqref{eq.z_V} and \eqref{eq.z_X} in \eqref{eq.carl_square_z}, shows 
\begin{align}
\label{eq.z_VX} \int_{\mf{U}\cap\mf{D}} \zeta_{a,b;\varepsilon}f |\square z|^2 \leqslant  8 M^2 R_+  \int_{\mf{U}\cap\mf{D}} \zeta_{a,b;\varepsilon}f (z_{t_1}^2 + z_{t_2}^2 + |\nabla_x z|^2 ) dt_1 dt_2 dx . 
\end{align}
Then \eqref{eq.carl_z} reduces to
\begin{align*}
\frac{ C \delta^2 }{R_+^2} & \int_{\mf{U}\cap \mf{D}} \zeta_{a,b;\varepsilon} (|u\partial_u z |^2 +  |v\partial_v z |^2 + f g^{ab}\slashed\nabla_a z \slashed\nabla_b z - f g^{CD} \tilde\nabla_C z \tilde\nabla _D z  ) +\frac{ C \delta a^2}{R_+^2} \int_{\mf{U}\cap \mf{D}}\zeta_{a,b;\varepsilon} z^2  \\
& \leqslant \frac{M^2 R_+}{a}\int_{\mf{U}\cap \mf{D}} \zeta_{a,b;\varepsilon} f (z_{t_1}^2 + z_{t_2}^2 + |\nabla_x z|^2 ) + a^2 R_+^3 \int_{( (-T,T)^2 \times \omega ) \cap \mf{D} } \zeta_{a,b;\varepsilon} f^{-2} ( z_{t_1}^2 + z_{t_2}^2 )  \numberthis \label{eq.pre_obs}\\
& \qquad + a^4 R_+^4 \int_{( (-T,T)^2 \times \omega ) \cap \mf{D}} \zeta_{a,b;\varepsilon} f^{-3} z^2. 
\end{align*}
From \eqref{eq.z_deriv}, it is clear that \(z, z_{t_1}, z_{t_2}\) are expressed in \(\phi\) and a time integral of \(\phi\). We will see that these terms can be dealt with conveniently and they yield the term present in the RHS of \eqref{eq.obs_ext_est}. Hence, it is important to get rid of the other term: the spatial derivative of \(z\). Thus, we will estimate the \(\nabla_x z\) term by \(z, z_{t_1}, z_{t_2}\). We use Lemma \ref{thm.der_carl_est} to conclude that \( |\nabla^\alpha (\zeta_{a,b;\varepsilon} f)| \leqslant a K R_+ \zeta_{a,b;\varepsilon}\), for some \(K>0\). Also, due to \eqref{eq.wave_obs} and \eqref{eq.z_def}, we have \( z=0 \) on \(\partial\mf{U}\cap \mf{D}\). Then using \eqref{eq.z_wave}, integration by parts, and \eqref{eq.z_VX} shows
\begin{align*}
& \frac{M^2 R_+}{a} \int_{\mf{U}\cap \mf{D}} \zeta_{a,b;\varepsilon} f |\nabla_x z|^2 \\
& \ = - \frac{M^2 R_+}{a} \int_{\mf{U}\cap \mf{D}} \nabla_x(\zeta_{a,b;\varepsilon} f)\cdot (\nabla_x z) z - \frac{M^2 R_+}{a} \int_{\mf{U}\cap \mf{D}} \zeta_{a,b;\varepsilon} f \Delta_x z \cdot z + \frac{M^2 R_+}{a}\int_{\partial\mf{U}\cap \mf{D}} \zeta_{a,b;\varepsilon} f \mc{N} z \cdot z \\
& \ \leqslant K M^2 R_+^2 \int_{\mf{U}\cap \mf{D}} \zeta_{a,b;\varepsilon} |\nabla_x z \cdot z | - \frac{M^2 R_+}{a} \int_{\mf{U}\cap \mf{D}} \zeta_{a,b;\varepsilon} f (z_{t_1t_1} + z_{t_2t_2} ) z \\
& \ \qquad + \frac{M^2 R_+}{a}  \int_{\mf{U}\cap \mf{D}} \zeta_{a,b;\varepsilon} f z \left( \int_{t_1}^{t_2} V z_{t_2} + \nabla_{\mc{X}}z_{t_2} d\tau \right) \\
& \ \leqslant K M^2 R_+^2 \int_{\mf{U}\cap \mf{D}} \zeta_{a,b;\varepsilon} |\nabla_x z \cdot z |  + \frac{M^2 R_+}{a} \int_{\mf{U}\cap \mf{D} } [ \partial_{t_1}(\zeta_{a,b;\varepsilon} f) z_{t_1} + \partial_{t_2}(\zeta_{a,b;\varepsilon} f) z_{t_2}] z \\
& \ \qquad + \frac{M^2 R_+}{a} \int_{\mf{U}\cap \mf{D}} \zeta_{a,b;\varepsilon} f (z_{t_1}^2 + z_{t_2}^2) + \frac{ 4 M^4 R_+^2}{a} \int_{\mf{U}\cap \mf{D}} \zeta_{a,b;\varepsilon} f z^2 \\
& \ \qquad + \frac{1}{16a}\int_{\mf{U}\cap \mf{D}} \zeta_{a,b;\varepsilon} f \left( \int_{t_1}^{t_2} V z_{t_2} + \nabla_{\mc{X}}z_{t_2} d\tau \right)^2 \\
& \ \leqslant 2a K^2 M^2 R_+^3 \int_{\mf{U}\cap \mf{D}} \zeta_{a,b;\varepsilon} f^{-1} z^2 + \frac{M^2 R_+}{8 a} \int_{\mf{U}\cap \mf{D}} \zeta_{a,b;\varepsilon} f |\nabla_x z|^2 + K M^2 R_+^2 \int_{\mf{U}\cap \mf{D}} \zeta_{a,b;\varepsilon} |(z_{t_1} + z_{t_2} ) z| \\
& \ \qquad + \frac{M^2 R_+}{a} \int_{\mf{U}\cap \mf{D}} \zeta_{a,b;\varepsilon} f (z_{t_1}^2 + z_{t_2}^2) + \frac{ 4 M^4 R_+^2}{a} \int_{\mf{U}\cap \mf{D}} \zeta_{a,b;\varepsilon} f z^2 \\
& \ \qquad + \frac{ 8 M^2 R_+}{16 a} \int_{\mf{U}\cap\mf{D}} \zeta_{a,b;\varepsilon}f (z_{t_1}^2 + z_{t_2}^2 + |\nabla_x z|^2 ),
\end{align*}
where we also used H\"older's inequality. Then, first absorbing the \( \nabla_x z \) terms into the LHS and then absorbing the constant \( K \) into the inequality sign (\(\leqslant \rightarrow \lesssim\)), shows that
\begin{align*}
\frac{M^2 R_+}{a} & \int_{\mf{U}\cap \mf{D}} \zeta_{a,b;\varepsilon} f |\nabla_x z|^2 \\ 
& \lesssim  a M^2 R_+^3 \int_{\mf{U}\cap \mf{D}} \zeta_{a,b;\varepsilon}f^{-1} z^2 + \frac{M^2 R_+}{a} \int_{\mf{U}\cap \mf{D}} \zeta_{a,b;\varepsilon} f (z_{t_1}^2 + z_{t_2}^2 ) + \frac{ M^4 R_+^2}{a} \int_{\mf{U}\cap \mf{D}} \zeta_{a,b;\varepsilon} f z^2\\
 & \lesssim a M^2 R_+^3 \int_{\mf{U}\cap \mf{D}} \zeta_{a,b;\varepsilon} f^{-1} z^2 + \frac{M^2 R_+}{a} \int_{\mf{U}\cap \mf{D}} \zeta_{a,b;\varepsilon} f (z_{t_1}^2 + z_{t_2}^2 ), \numberthis \label{eql.obs_ext_1a}
\end{align*}
where we used \eqref{eql.obs_ext_a} to obtain the final constants that appear in front of the integral terms along with the fact that due to \eqref{eq.uvf_bound1}
\begin{align*}
\frac{ M^4 R_+^2}{a} \int_{\mf{U}\cap \mf{D}} \zeta_{a,b;\varepsilon} f z^2 \leqslant \frac{ M^4 R_+^2}{a} \int_{\mf{U}\cap \mf{D}} \zeta_{a,b;\varepsilon} f^2 f^{-1} z^2 & \leqslant \frac{ M^4 R_+^6}{a} \int_{\mf{U}\cap \mf{D}} \zeta_{a,b;\varepsilon} f^{-1} z^2 \\
& \ll a M^2 R_+^3 \int_{\mf{U}\cap \mf{D}} \zeta_{a,b;\varepsilon} f^{-1} z^2.
\end{align*}
Applying \eqref{eql.obs_ext_1a} to \eqref{eq.pre_obs}, gives us 
\begin{align*}
\numberthis \label{eq:obs_ext_2} \frac{ C \delta^2 }{R_+^2} \int_{\mf{U}\cap \mf{D}}  \zeta_{a,b;\varepsilon} & (|u\partial_u z |^2 +  |v\partial_v z |^2 + f g^{ab}\slashed\nabla_a z \slashed\nabla_b z - f g^{CD} \tilde\nabla_C z \tilde\nabla _D z  ) + \frac{ C \delta a^2}{R_+^2} \int_{\mf{U}\cap \mf{D}}\zeta_{a,b;\varepsilon} z^2 \\
 & \leqslant  \frac{M^2 R_+}{a} \int_{\mf{U}\cap \mf{D}} \zeta_{a,b;\varepsilon} f (z_{t_1}^2 + z_{t_2}^2 ) + a M^2 R_+^3 \int_{\mf{U}\cap \mf{D}} \zeta_{a,b;\varepsilon} f^{-1} z^2  \\
& \qquad  + a^2 R_+^3 \int_{( (-T,T)^2 \times \omega ) \cap \mf{D} } \zeta_{a,b;\varepsilon} f^{-2} ( z_{t_1}^2 + z_{t_2}^2 ) \\
& \qquad + a^4 R_+^4 \int_{( (-T,T)^2 \times \omega ) \cap \mf{D}} \zeta_{a,b;\varepsilon} f^{-3} z^2. 
\end{align*}
We will use the following notations for convenience
\begin{align*}
I_0 & := a M^2 R_+^3 \int_{\mf{U}\cap \mf{D}} \zeta_{a,b;\varepsilon} f^{-1} z^2 , \\
I_1 & :=  \frac{M^2 R_+}{a}\int_{\mf{U}\cap \mf{D}} \zeta_{a,b;\varepsilon} f (z_{t_1}^2 + z_{t_2}^2 ),  \\
I_\omega & := a^2 R_+^3 \int_{( (-T,T)^2 \times \omega ) \cap \mf{D}} \zeta_{a,b;\varepsilon} f^{-2} ( z_{t_1}^2 + z_{t_2}^2 ) + a^4 R_+^4 \int_{( (-T,T)^2 \times \omega ) \cap \mf{D}} \zeta_{a,b;\varepsilon} f^{-3} z^2.
\end{align*}

\subsubsection{Splitting the domain.} Since the weight \(  \zeta_{a,b;\varepsilon} \) vanishes near the boundary of the null cone, the LHS of \eqref{eq:obs_ext_2} becomes singular in that region and we cannot capture the \(L^2 \) norm of \( z\) any more. This will create an issue with absorption of some of the RHS terms into the LHS. To overcome this problem, we split the domain \(\mf{U} \cap \mf{D} \) as follows:
\begin{align}
\label{eql.obs_ext_Udec} \mf{U}_\leq &:= \mf{U} \cap \mf{D} \cap \left\{ \frac{ f }{ ( 1 + \varepsilon u ) ( 1 - \varepsilon v ) } \leq \frac{ R_-^2 }{ 64 } \right\} \text{,} \\
\notag \mf{U}_> &:= \mf{U} \cap \mf{D} \cap \left\{ \frac{ f }{ ( 1 + \varepsilon u ) ( 1 - \varepsilon v ) } > \frac{ R_-^2 }{ 64 } \right\} \text{.}
\end{align}
Due to \eqref{eq.uvf_bound1}, \eqref{eql.obs_ext_Udec}, and \eqref{eq.conf_comp} we have the following estimates on \(\mf{U}_>,\)
\[ v = \frac{ f}{ - u } \gtrsim \frac{ R_-^2 }{ R_+ } \text{,} \qquad - u = \frac{ f }{ v } \gtrsim \frac{ R_-^2 }{ R_+ } \text{.} \]
Let us write \( I_0\) as
\begin{equation}
I_0 = a M^2 R_+^3 \int_{\mf{U}_\leqslant} \zeta_{a,b;\varepsilon} f^{-1} z^2 + a M^2 R_+^3 \int_{\mf{U}_>} \zeta_{a,b;\varepsilon} f^{-1} z^2 =: I_{0, \leqslant} + I_{0, >},
\end{equation}
and \(I_1\) as 
\begin{equation}
I_1 = \frac{M^2 R_+}{a}\int_{\mf{U}_\leqslant} \zeta_{a,b;\varepsilon} f (z_{t_1}^2 + z_{t_2}^2 ) + \frac{M^2 R_+}{a}\int_{\mf{U}_>} \zeta_{a,b;\varepsilon} f (z_{t_1}^2 + z_{t_2}^2 ) =: I_{1, \leqslant} + I_{1,>}.
\end{equation}
After restricting the domain on the LHS of \eqref{eq:obs_ext_2}, we get
\begin{equation} \label{eq:obs_ext_2b}
C \int_{\mf{U}_>}\zeta_{a,b;\varepsilon} \left[ \frac{\delta^2 R_-^2}{R_+^3}( - u |\partial_u z |^2 +  v |\partial_v z |^2 + v g^{ab}\slashed\nabla_a z \slashed\nabla_b z - v g^{CD} \tilde\nabla_C z \tilde\nabla _D z  )  + \frac{ \delta a^2}{R_+^2} z^2 \right] \leqslant I_0 + I_1 + I_\omega.
\end{equation}
For \(I_{0,>}\), using \eqref{eql.obs_ext_a} and \eqref{eql.obs_ext_Udec} 
\begin{align*}
I_{0,>} \leqslant a M^2 R_+^3 R_-^{-2} \int_{\mf{U}_>} \zeta_{a,b;\varepsilon} z^2 \ll \frac{ \delta a^2}{R_+^2}\int_{\mf{U}_>} \zeta_{a,b;\varepsilon} z^2.
\end{align*}
Hence, this term can be absorbed into the LHS of \eqref{eq:obs_ext_2b}. For \( I_{1,>}\), using \eqref{eq.uvf_bound1},  \eqref{eql.obs_ext_a}, gives
\begin{align*}
I_{1,>} & \leqslant \frac{M^2 R_+^2}{a} \int_{\mf{U}_>}\zeta_{a,b;\varepsilon} ( - u |\partial_u z |^2 +  v |\partial_v z |^2 + v g^{ab}\slashed\nabla_a z \slashed\nabla_b z - v g^{CD} \tilde\nabla_C z \tilde\nabla _D z  ) \\
& \ll \frac{ \delta^2 R_-^2}{R_+^3} \int_{\mf{U}_>}\zeta_{a,b;\varepsilon} ( - u |\partial_u z |^2 +  v |\partial_v z |^2 + v g^{ab}\slashed\nabla_a z \slashed\nabla_b z - v g^{CD} \tilde\nabla_C z \tilde\nabla _D z  ),
\end{align*}
hence we can absorb this term into the LHS of \eqref{eq:obs_ext_2b}. Thus, we have
\begin{equation}\label{eq.shnk_dom}
C  \int_{\mf{U}_>}\zeta_{a,b;\varepsilon} \left[ \frac{ \delta^2 R_-^2}{R_+^3}( - u |\partial_u z |^2 +  v |\partial_v z |^2 + v g^{ab}\slashed\nabla_a z \slashed\nabla_b z - v g^{CD} \tilde\nabla_C z \tilde\nabla _D z  ) + \frac{ \delta a^2}{R_+^2}  z^2 \right] \leqslant I_{0, \leqslant} + I_{1,\leqslant} + I_\omega.
\end{equation}
We note that due to \eqref{eq.uv} and \eqref{eq.obs_ext_MR}
\begin{align}\label{eql.obs_ext_50}
r \geqslant R_-, \qquad -u \geqslant \frac{ 3R_- }{8}, \qquad v \geqslant \frac{3 R_- }{8},
\end{align}
for any \( \tau \in \R, \text{ with } |\tau| < \frac{1}{4} R_- \), where \(|\tau| = \sqrt{t_1^2 + t_2^2}\). Then due to \eqref{eq.carleman_weight} and \eqref{eql.obs_ext_50}, we have the following estimates for the same range of \(\tau\)
\[  \frac{ f }{ ( 1 + \varepsilon u ) ( 1 - \varepsilon v ) } \geq \frac{ R_-^2 }{ 16 } , \qquad \zeta_{ a, b; \varepsilon } \geq \left( \frac{ R_- }{4} e^\frac{ b R_- }{4} \right)^{ 4 a }.\]
This shows that
\[ \mf{U} \cap \left\{  |\tau| < \frac{ R_- }{4} \right\} \subseteq \mf{U}_> . \]
Taking this, along with an application of Fubini's theorem, shows that \eqref{eq.shnk_dom} reduces to
\begin{equation}
C \left( \frac{ R_- }{4} e^\frac{ b R_- }{4} \right)^{ 4 a } \iint_{\left\{  |\tau| < \frac{ R_- }{4} \right\} } \int_\Omega \left[ \frac{ \delta^2 R_-^3}{R_+^3} |\nabla_{t_1,t_2,x} z |^2 + \frac{ \delta a^2}{R_+^2}  z^2 \right] \leqslant I_{0, \leqslant} + I_{1,\leqslant} + I_\omega.
\end{equation}
Also, due to \eqref{eql.obs_ext_a} we get
\begin{equation}\label{eql.obs_ext_7}
 \frac{ C \delta^2 R_-^3}{R_+^3} \left( \frac{ R_- }{4} e^\frac{ b R_- }{4} \right)^{ 4 a } \iint_{\left\{  |\tau| < \frac{ R_- }{4} \right\} } \int_\Omega  ( |\nabla_{t_1,t_2,x} z |^2 + z^2)  \leqslant I_{0, \leqslant} + I_{1,\leqslant} + I_\omega.
\end{equation}
Note that due to \eqref{eq.carleman_weight}, \eqref{eq.uvf_bound1}, \eqref{eq.conf_comp}, \eqref{eql.obs_ext_b}, and \eqref{eql.obs_ext_Udec} the following bounds hold on \(\mf{U}_\leq\)
\[
\zeta_{ a, b; \varepsilon } |_{ \mf{U}_\leq } \leq \left( \frac{ R_- }{ 8 } e^\frac{ b R_- }{ 8 } \right)^{ 4 a } , \qquad \zeta_{ a, b; \varepsilon } f^{-1} |_{ \mf{U}_\leq } \leq \frac{64}{R_-^2} \left( \frac{ R_- }{ 8 } e^\frac{ b R_- }{ 8 } \right)^{ 4 a },  \qquad f |_{ \mf{U}_\leq } \leq R_-^2.
\]
Hence, we have
\begin{equation}\label{eql.obs_ext_71}
I_{0, \leqslant} + I_{1,\leqslant} \lesssim \frac{a M^2 R_+^3}{R_-^2} \left( \frac{ R_- }{ 8 } e^\frac{ b R_- }{ 8 } \right)^{ 4 a } \int_{\mf{U}_{\leqslant}} z^2  + \frac{M^2 R_+ R_-^2}{a} \left( \frac{ R_- }{ 8 } e^\frac{ b R_- }{ 8 } \right)^{ 4 a } \int_{\mf{U}_{\leqslant}} (z_{t_1}^2 + z_{t_2}^2).
\end{equation}
For \(I_\omega\), using \eqref{eql.obs_ext_b} we can estimate the coefficients as follows
\begin{align*}
\zeta_{a,b;\varepsilon} f^{-2} & \leqslant (4f)^{2a} f^{-2} \leqslant 4^{2a}f^{2a-2} \leqslant 2^{4a}R_+^{4a-4}, \\
\zeta_{a,b;\varepsilon} f^{-3} & \leqslant (4f)^{2a} f^{-3} \leqslant 4^{2a}f^{2a-3} \leqslant 2^{4a}R_+^{4a-6}.
\end{align*}
Then, we get the bound
\begin{align*}
I_\omega & \lesssim a^2 R_+^3 2^{4a}R_+^{4a-4} \int_{ ( (-T,T)^2 \times \omega )) \cap \mf{D}} ( z_{t_1}^2 + z_{t_2}^2 )  +  a^4 R_+^4 2^{4a}R_+^{4a-6} \int_{( (-T,T)^2 \times \omega ) \cap \mf{D}} z^2 \\
& \lesssim a^2 2^{4a}R_+^{4a-1} \int_{ ( (-T,T)^2 \times \omega )) \cap \mf{D}} ( z_{t_1}^2 + z_{t_2}^2 ) + a^4 2^{4a} R_+^{4a-2} \int_{ ( (-T,T)^2 \times \omega ) \cap \mf{D}} z^2 .
\end{align*}
Using this in \eqref{eql.obs_ext_7}, and also using \eqref{eql.obs_ext_71} shows
\begin{align*}
\numberthis \label{eql.obs_ext_9} \frac{ C \delta^2 R_-^3}{R_+^3} & \left( \frac{ R_- }{4} e^\frac{ b R_- }{4} \right)^{ 4 a } \iint_{\left\{  |\tau| < \frac{ R_- }{4} \right\} } \int_\Omega  ( |\nabla_{t_1,t_2,x} z |^2 + z^2) \\
& \leqslant  \frac{a^2 M^2 R_+^3}{R_-^2} \left( \frac{ R_- }{ 8 } e^\frac{ b R_- }{ 8 } \right)^{ 4 a } \int_{\mf{U}_{\leqslant}} z^2 + \frac{M^2 R_+ R_-^2}{a} \left( \frac{ R_- }{ 8 } e^\frac{ b R_- }{ 8 } \right)^{ 4 a } \int_{\mf{U}_{\leqslant}} (z_{t_1}^2 + z_{t_2}^2) \\
& \qquad + a^2 2^{4a}R_+^{4a-1} \int_{ ( (-T,T)^2 \times \omega )) \cap \mf{D}} ( z_{t_1}^2 + z_{t_2}^2 ) + a^4 2^{4a} R_+^{4a-2} \int_{ ( (-T,T)^2 \times \omega ) \cap \mf{D}} z^2.
\end{align*}

\subsubsection{Going back from \(z\) to \(\phi\).}\label{sssec_zphi}

Now we go back to the original (wave equation) function \( \phi \), using \eqref{eq.z_def}. For this purpose, note that due to \eqref{eq.z_deriv}
\begin{equation}
\int_{\mf{U}_{\leqslant}} ( z_{t_1}^2 + z_{t_2}^2 ) \leqslant \int_{\mf{U}\cap \mf{D}} ( z_{t_1}^2 + z_{t_2}^2 )  = \int_{\mf{U}\cap \mf{D}} |\phi(t_1,x)|^2 + |\phi(t_2,x)|^2.
\end{equation}
Let us look at the first term on the RHS
\begin{align*}
\int_{\mf{U}\cap \mf{D}} |\phi(t_1,x)|^2 & = \int_\Omega dx \iint_{t_1^2 + t_2^2 < r^2} dt_1 dt_2 |\phi(t_1,x)|^2 \\
& = \int_\Omega dx \iint_{t_1^2 < r^2 - t_2^2 } dt_1 dt_2 |\phi(t_1,x)|^2 \\
& \leqslant \int_\Omega dx \int_{t_1^2 < r^2 } dt_1 \int_{t_2}  dt_2 |\phi(t_1,x)|^2 \\
& \leqslant 2 R_+ \int_\Omega dx \int_{t_1^2 < r^2 } dt_1 |\phi(t_1,x)|^2\\
& = 2 R_+ \int_{\mc{U} \cap \mc{D} }|\phi(t,x)|^2.
\end{align*}
Similarly, we also get
\begin{equation}
\notag \int_{\mf{U}\cap \mf{D}} |\phi(t_2,x)|^2 \leqslant 2 R_+ \int_{\mc{U} \cap \mc{D} }|\phi(t,x)|^2.
\end{equation}
For the term containing \(z^2\) in \eqref{eql.obs_ext_9}, we have
\begin{align*}
\int_{\mf{U}_{\leqslant}} z^2 \leqslant \int_{\mf{U} \cap \mf{D}} z^2 & \leqslant \int_\Omega dx \iint_{t_1^2 + t_2^2 < r^2} dt_1 dt_2  \left( \int _{t_1}^{t_2} |\phi(s,x)| ds \right)^2 \\
& \leqslant \int_\Omega dx \iint_{t_1^2 + t_2^2 < r^2} dt_1 dt_2   \left( \int _{t_1}^{t_2} |\phi(s,x)|^2 ds \right)\\
& \leqslant \int_\Omega dx \iint_{t_1^2 + t_2^2 < r^2} dt_1 dt_2  \int_{-r}^r |\phi(s,x)|^2 ds \\
& \leqslant 4 R_+^2 \int_ \Omega dx \int_{-r}^r ds|\phi(s,x)|^2 \\
& \leqslant 4 R_+^2 \int_{\mc{U} \cap \mc{D}} |\phi(t,x)|^2.
\end{align*}
We use a similar calculation for estimating the \(\omega\)-integral term in the RHS of \eqref{eql.obs_ext_9} to get that
\begin{align*}
a^2 2^{4a}R_+^{4a-1} \int_{ ( (-T,T)^2 \times \omega )) \cap \mf{D}} ( z_{t_1}^2 + z_{t_2}^2 ) + a^4 2^{4a} & R_+^{4a-2} \int_{ ( (-T,T)^2 \times \omega ) \cap \mf{D}} z^2 \\
& \lesssim a^4 2^{4a} R_+^{4a} \int_{ ( (-T,T) \times \omega ) \cap \mc{D}} |\phi(t,x)|^2.
\end{align*}
Then using the above bounds in \eqref{eql.obs_ext_9}, shows that 
\begin{align}\label{eql.obs_ext_12}
\notag \frac{ C \delta^2 R_-^3}{R_+^3} & \left( \frac{ R_- }{4} e^\frac{ b R_- }{4} \right)^{ 4 a } \int_\Omega \iint_{ \sqrt{t_1^2+t_2^2} < \frac{ R_- }{4} } (|\phi(t_1,x)|^2 + |\phi(t_2,x)|^2 )\\
& \leqslant \frac{a^2 M^2 R_+^5}{R_-^2} \left( \frac{ R_- }{ 8 } e^\frac{ b R_- }{ 8 } \right)^{ 4 a } \int_{\mc{U} \cap \mc{D}} |\phi(t,x)|^2 +  \frac{M^2 R_+^2 R_-^2}{a} \left( \frac{ R_- }{ 8 } e^\frac{ b R_- }{ 8 } \right)^{ 4 a } \int_{\mc{U} \cap \mc{D}} |\phi(t,x)|^2 \\
\notag & \qquad + a^4 2^{4a} R_+^{4a} \int_{ ( (-T,T) \times \omega ) \cap \mc{D}} |\phi(t,x)|^2.
\end{align}
For the term in the LHS, if we only consider the first integrand, we see that
\begin{align}
\notag \int_\Omega  \iint_{\sqrt{t_1^2+t_2^2} \leqslant \frac{R_-}{4}} |\phi(t_1,x)|^2 & \geqslant \int_\Omega dx \int_{ \text{max}\{|t_1|,|t_2|\} \leqslant \frac{R_-}{64}} |\phi(t_1,x)|^2 \\
\notag & \geqslant \int_\Omega dx \int_{-\frac{R_-}{64}}^{\frac{R_-}{64}} dt_1 \int_{-\frac{R_-}{64}}^{\frac{R_-}{64}} dt_2 |\phi(t_1,x)|^2  \\
& \geqslant \frac{R_-}{32} \int_\Omega dx \int_{-\frac{R_-}{64}}^{\frac{R_-}{64}} dt_1 |\phi(t_1,x)|^2. \label{eql.obs_ext_12a}
\end{align}
Similarly, we also get
\begin{equation}
\int_\Omega  \iint_{\sqrt{t_1^2+t_2^2} \leqslant \frac{R_-}{4}} |\phi(t_2,x)|^2 \geqslant \frac{R_-}{32} \int_\Omega dx \int_{-\frac{R_-}{64}}^{\frac{R_-}{64}} dt_2 |\phi(t_2,x)|^2. \label{eql.obs_ext_12b}
\end{equation}
Substituting \eqref{eql.obs_ext_12a} and \eqref{eql.obs_ext_12b} in \eqref{eql.obs_ext_12} shows
\begin{align*}
\frac{ C \delta^2 R_-^4}{R_+^3} & \left( \frac{ R_- }{4} e^\frac{ b R_- }{4} \right)^{ 4 a } \int_\Omega dx \int_{-\frac{R_-}{64}}^{\frac{R_-}{64}} dt |\phi(t,x)|^2 \\
& \leqslant  \frac{a^2 M^2 R_+^5}{R_-^2} \left( \frac{ R_- }{ 8 } e^\frac{ b R_- }{ 8 } \right)^{ 4 a } \int_{\mc{U} \cap \mc{D}} |\phi(t,x)|^2 +  \frac{M^2 R_+^2 R_-^2}{a} \left( \frac{ R_- }{ 8 } e^\frac{ b R_- }{ 8 } \right)^{ 4 a } \int_{\mc{U} \cap \mc{D}} |\phi(t,x)|^2 \\
\notag & \qquad + a^4 2^{4a} R_+^{4a} \int_{ ( (-T,T) \times \omega ) \cap \mc{D}} |\phi(t,x)|^2.
\end{align*}
Due to the condition \eqref{eql.obs_ext_a}, the above estimate becomes
\begin{align}
\label{eql.obs_ext_14} \frac{ C \delta^2 R_-^4}{R_+^3}  \left( \frac{ R_- }{4} e^\frac{ b R_- }{4} \right)^{ 4 a } \int_\Omega dx \int_{-\frac{R_-}{64}}^{\frac{R_-}{64}} dt |\phi(t,x)|^2 & \leqslant \frac{a^2 M^2 R_+^5}{R_-^2} \left( \frac{ R_- }{ 8 } e^\frac{ b R_- }{ 8 } \right)^{ 4 a } \int_{\mc{U} \cap \mc{D}} |\phi(t,x)|^2\\
\notag & \qquad + a^4 2^{4a} R_+^{4a} \int_{ ( (-T,T) \times \omega ) \cap \mc{D}} |\phi(t,x)|^2.
\end{align}

\subsubsection{Application of energy estimate lemmas.}\label{sssec_obsenergy}
Next, we apply the energy estimate equation \eqref{eq.energy1} to \eqref{eql.obs_ext_14} to get
\begin{align*}
\frac{ C \delta^2 R_-^4}{R_+^3} & \left( \frac{ R_- }{4} e^\frac{ b R_- }{4} \right)^{ 4 a } \int_{-\frac{R_-}{100}}^{\frac{R_-}{100}} ( |\phi(t,\cdot)|^2_{L^2(\Omega)} + |\phi'(t,\cdot)|^2_{H^{-1}(\Omega)} ) dt \\
& \leqslant \frac{a^2 C_1 M^3 R_+^5}{R_-^2} \left( \frac{ R_- }{ 8 } e^\frac{ b R_- }{ 8 } \right)^{ 4 a } \int_{\mc{U} \cap \mc{D}} |\phi(t,x)|^2 + a^4 2^{4a} C_1 M R_+^{4a} \int_{ ( (-T,T) \times \omega ) \cap \mc{D}} |\phi(t,x)|^2.
\end{align*}
Using \eqref{eq.energy2} for the first term on the RHS, shows
\begin{align*}
\numberthis \label{eql.obs_ext_15} \frac{ C \delta^2 R_-^4}{R_+^3} \left( \frac{ R_- }{4} e^\frac{ b R_- }{4} \right)^{ 4 a }  \int_{-\frac{R_-}{100}}^{\frac{R_-}{100}} E(t) dt & \leqslant \frac{a^2 C_1 M^3 R_+^6}{R_-^2} \left( \frac{ R_- }{ 8 } e^\frac{ b R_- }{ 8 } \right)^{ 4 a } e^{C_2(M_0 + M_1) R_+} E(0) \\
& \qquad + a^4 2^{4a} C_1 M R_+^{4a} \int_{ ( (-T,T) \times \omega ) \cap \mc{D}} |\phi(t,x)|^2.
\end{align*}
Applying  \eqref{eq.energy2} to the LHS of the above estimate results in
\begin{align*}
\frac{ C \delta^2 R_-^4}{R_+^3} \left( \frac{ R_- }{4} e^\frac{ b R_- }{4} \right)^{ 4 a }  \int_{-\frac{R_-}{100}}^{\frac{R_-}{100}} E(t) dt & \geqslant \frac{ C \delta^2 R_-^4}{R_+^3} \left( \frac{ R_- }{4} e^\frac{ b R_- }{4} \right)^{ 4 a } \int_{-\frac{R_-}{100}}^{\frac{R_-}{100}} e^{- C_2 (M_0+M_1) R_-} E(0)\\
& \geqslant \frac{ C \delta^2 R_-^5}{R_+^3} \left( \frac{ R_- }{4} e^\frac{ b R_- }{4} \right)^{ 4 a }  e^{- C_2 (M_0+M_1) R_+} E(0).
\end{align*}
Then \eqref{eql.obs_ext_15} reduces to
\begin{align*}
\frac{ C \delta^2 R_-^5}{R_+^3} \left( \frac{ R_- }{4} e^\frac{ b R_- }{4} \right)^{ 4 a }  e^{- C_2 (M_0+M_1) R_+} E(0) & \leqslant \frac{a^2 C_1 M^3 R_+^6}{R_-^2} \left( \frac{ R_- }{ 8 } e^\frac{ b R_- }{ 8 } \right)^{ 4 a } e^{C_2(M_0 + M_1) R_+} E(0) \\
& \qquad + a^4 2^{4a} C_1 M R_+^{4a} \int_{ ( (-T,T) \times \omega ) \cap \mc{D}} |\phi(t,x)|^2.
\end{align*}
To prove the observability inequality \eqref{eq.obs_ext_est}, it is enough to show that the first term in the RHS can be absorbed into the LHS. For this purpose, we just have to show that
\begin{align*}
\frac{a^2 C_1 M^3 R_+^6}{R_-^2} \left( \frac{ R_- }{ 8 } e^\frac{ b R_- }{ 8 } \right)^{ 4 a } e^{ C_2(M_0 + M_1) R_+} \ll \frac{ C \delta^2 R_-^5}{R_+^3} \left( \frac{ R_- }{4} e^\frac{ b R_- }{4} \right)^{ 4 a }  e^{- C_2 (M_0+M_1) R_+},
\end{align*}
which is equivalent to showing
\begin{align*}
\numberthis \label{eql.obs_ext_16} \left[\frac{a^2 C_1 M^3 R_+^6}{R_-^2} \cdot \frac{R_+^3}{ C \delta^2 R_-^5}\right] \left( \frac{1}{2} \right)^{4a} \left( e^\frac{ - b R_- }{ 8 }\right)^{4a}  e^{2C_2(M_0 + M_1) R_+} \ll 1.
\end{align*}
To show the above, we see that due to \eqref{eql.obs_ext_a}
\begin{align*}
\frac{a^2 C_1 M^3 R_+^6}{R_-^2} \cdot \frac{R_+^3}{ C \delta^2 R_-^5} \ll \frac{a^3 M R_+^4}{R_-^5} \ll \frac{a^4 R_+^2}{R_-^2} \ll a^5 R_-, \qquad \left( e^\frac{ - b R_- }{ 8 }\right)^{4a} \ll 1, \qquad e^{2C_2(M_0 + M_1) R_+} \ll e^a,
\end{align*}
up to the constants \( C, C_1, C_2\). Thus, the LHS of \eqref{eql.obs_ext_16} satisfies
\begin{align*}
 \left[\frac{a^2 C_1 M^3 R_+^6}{R_-^2} \cdot \frac{R_+^3}{ C \delta^2 R_-^5}\right] \left( \frac{1}{2} \right)^{4a} \left( e^\frac{ - b R_- }{ 8 }\right)^{4a}  e^{2C_2(M_0 + M_1) R_+}  \ll a^5 R_- \left( \frac{1}{2} \right)^{4a}  e^a \ll 1,
\end{align*}
for large enough \(a\). This shows that \eqref{eql.obs_ext_16} is true. Then, applying Lemma \ref{thm.energy2} once more to go from \(E(0)\) to \(E(-T)\) completes the proof of \eqref{eq.obs_ext_est}, since the LHS of \eqref{eq.obs_ext_est} is just \(E(-T)\). \hfill \( \qed \)

\subsection{Interior Observability}\label{ssec_intobs}
In this section, we prove an observability estimate when the centre point for the Carleman estimate lies within the domain \(\mf{U}\).\\

\noindent \emph{\underline{Shifted coordinate system in \( \R^{2+n} \)}}: Before presenting the interior observability, we present a generalised version of Theorem \ref{thm.carl_int} where the centre point for the Carleman estimate is taken to be any point \( Q \in \R^{2+n}\). Note that, we have the assumption \( \boldsymbol{ m=2}\). We provide some basic definitions in the shifted coordinate system with respect to \( Q\).

\begin{remark}
We only present the shifted Carleman estimate result in the \( \R^{2+n} \) case. But, it can be analogously stated for the \( \R^{m+n}\) case as well. 
\end{remark}

\begin{definition}
Fix \( Q \in \R^{2+n}\).
\begin{itemize}

\item Define the shifted time and spatial coordinates as
\begin{equation}\label{eq.tx_shift}
t_Q := t - t(Q), \qquad x_Q := x- x(Q).
\end{equation}

\item Analogous to Definition \ref{def_setting}, we have
\begin{equation}
\numberthis \label{eq.ruvf_shift} r_Q := | x_Q | \text{,} \quad \tau_Q := | t_Q | \text{,} \qquad u_Q := \frac{1}{2} ( t_Q - r_Q ) \text{,} \qquad v_Q := \frac{1}{2} ( t_Q + r_Q ) \text{,} \qquad f_Q := - u_Q v_Q.
\end{equation}

\item And also that
\begin{align*}
\numberthis g &= - d(t_{Q,1})^2 - d(t_{Q,2})^2 + d(x_{Q,1})^2 + \cdots + d(x_{Q,n})^2 \\
& = -d\tau_Q^2 + dr_Q^2 + r_Q^2 \mathring{\gamma}_{\mathbb{S}^{n-1}}- \tau_Q^2 \mathring{\gamma}_{\mathbb{S}^1} \\
& = -4du_Q dv_Q + r_Q^2 \mathring{\gamma}_{\mathbb{S}^{n-1}} - \tau_Q^2 \mathring{\gamma}_{\mathbb{S}^1}.
\end{align*}

\item Let \( \mf{D}_Q, \) be the domain given by 
\begin{equation}
\mf{D}_Q := \{ f_Q > 0 \}.
\end{equation}

\item Let \( \partial_{\tau_Q}, \partial_{r_Q}, \partial_{u_Q}, \partial_{v_Q} \) denote the coordinate vector fields with respect to the shifted coordinate system. Also, \( \nasla{}^Q\)denotes derivatives in the spatial angular components, and \( \tilde\nabla^Q \)denotes derivatives in the temporal angular components.

\end{itemize}
\end{definition}
\noindent Then we have the following interior Carleman estimate in the shifted coordinate system.

\begin{theorem}\label{thm.carl_int_P}
Let \( \mf{U}\) be defined as follows
\[\mf{U} := \R^2 \times \Omega,
\]
and fix \( R>0 \) such that
\begin{equation}\label{eq.carl_int_domain1}
\Omega \subseteq \{ r_Q<R \} .
\end{equation}
Let \( \varepsilon,\ a,\ b >0 \), be constants, such that:
\begin{equation}\label{eq.carl_int_P_choices}
a \geqslant (n+2)^2, \qquad a \gg R, \qquad \varepsilon \ll_n b \ll R^{-1}.
\end{equation}
Let \( \Gamma_+^Q \) be defined as follows: 
\begin{equation}\label{eq.thm.carl_int_P_1}
 \Gamma_+^Q := \{\nu r_Q >0 \}.
\end{equation}
Fix \(\sigma >0 \) and define \( \omega \), a subset of \( \Omega\), as
\begin{equation}
    \omega_Q = \mathcal{O}_\sigma(\Gamma^Q_+) \cap \Omega. \label{eq_P_3}
\end{equation}
Then, there exists \(C>0\) such that for any \( z \in \mc{C}^2({\mf{U}})\cap \mc{C}^1(\bar{\mf{U}}) \) satisfying \(z|_{\partial\mf{U} \cap \mf{D}_Q } = 0,\) we have the following estimate
\begin{align*} 
\numberthis \label{eq.carl_int_P} C\varepsilon \int_{\mf{U} \cap \mf{D}_Q} & \zeta_{a,b;\varepsilon}^Q r_Q^{-1} (|u_Q\partial_{u_Q} z |^2 +  |v_Q\partial_{v_Q} z |^2 + f_Q g^{ab}\slashed\nabla^Q_a z \slashed\nabla^Q_b z - f_Q g^{CD} \tilde\nabla^Q_C z \tilde\nabla^Q_D z  ) \\
& + Cba^2\int_{\mf{U} \cap \mf{D}_Q} \zeta_{a,b;\varepsilon}^Q f_Q^{-\frac{1}{2}} z^2 \\
& \qquad \leqslant \frac{1}{a}\int_{\mf{U} \cap \mf{D}_Q} \zeta_{a,b;\varepsilon}^Q f_Q |\square z|^2  +a^2 R^3 \int_{( \R^2 \times \omega_Q ) \cap \mf{D}_Q } \zeta_{a,b;\varepsilon}^Q f_Q^{-2} ( z_{t_{Q,1}}^2 + z_{t_{Q,2}}^2 )\\
& \qquad \qquad + a^4 R^4 \int_{( \R^2 \times \omega_Q ) \cap \mf{D}_Q} \zeta_{a,b;\varepsilon}^Q f_Q^{-3} z^2,
\end{align*}
where the Carleman weight is given by
\begin{equation}
\label{eq.carleman_weight_P} \zeta_{ a, b; \varepsilon }^Q := \left\{ \frac{ f_Q }{ ( 1 + \varepsilon u_Q ) ( 1 - \varepsilon v_Q ) } \cdot \exp \left[ \frac{ 2 b f_Q^\frac{1}{2} }{ ( 1 - \varepsilon u_Q )^\frac{1}{2} ( 1 + \varepsilon v_Q )^\frac{1}{2} } \right] \right\}^{2a},
\end{equation}
\end{theorem}
The proof of the above theorem is exactly the same as that of Theorem \ref{thm.carl_int}. We just have to use the shifted coordinate system, instead of the usual coordinate system.
\smallskip

\noindent \emph{\underline{Interior Observability result}}:
Analogous to defining the shifted coordinate system in \(\R^{2+n}\), we can define the shifted coordinate system in \(\R^{1+n}\) as well. Let \( P \in \R^{1+n} \). We will attach the letter \(P\) to the corresponding geometric quantities' symbols to denote them in the shifted coordinates. In particular,
\[ \mc{D}_P = \{ |x_P|^2 > |t_P|^2 \} \subset \R^{1+n}. \]

\begin{theorem} \label{thm.obs_int}
Let \( \Omega \subset \R^n \). Consider the setting of \eqref{eq.wave_obs}. Fix $P_1, P_2 \in \mc{U}$ with
\begin{equation}
\label{eq.obs_int_Pi} P_1 \neq P_2 \text{,} \qquad t ( P_1 ) = t ( P_2 ) := 0 \text{.}
\end{equation}
Also, assume $\mc{U} \cap ( \mc{D}_{ P_1 } \cup \mc{D}_{ P_2 } )$ is bounded. Further,
\begin{itemize}
\item Define the constants
\begin{align*}
\numberthis \label{eq.obs_int_MR} R_+ := \max_{ i = 1, 2 } \sup_{ \Omega } r_{ P_i } \text{,} &\qquad R_- := \frac{1}{2} | x ( P_2 ) - x ( P_1 ) | \\
M_0 := \sup_{ \mc{U} } | V | \text{,} &\qquad M_1 :=  \sup_{ \mc{U} } \left\{ \frac{| \mc{X}^{ t, x } |}{\sqrt{R_+}}, | \nabla_{t,x}  \mc{X}^{ t, x } | \right\}  \text{,} \\
 M &:= \max \{1, M_0, M_1\}  \text{.}
\end{align*}

\item Let \(\nu\) be the outward-pointing unit normal to \(\Omega\), and define, for $i \in \{ 1, 2 \}$,
\begin{equation}
\label{eq.obs_int_delta} \qquad \Gamma^i_+ := \{ \nu r_{ P_i } > 0 \} \text{.}
\end{equation}

\item Let \(\sigma>0\). For $i \in \{ 1, 2 \}$, define the interior subsets of \(\Omega\)
\begin{equation}\label{eq.obs_int_omega}
\omega_i = \mc{O}_\sigma(\Gamma^i_+) \cap \Omega.
\end{equation}

\end{itemize}
If \( T > R_+ \), then there exist constants \(C_1\) and \(C_2\), depending on $\mc{U}$, such that
\begin{equation}
\label{eq.obs_int_est} E(-T) = || \phi_0 ||_{L^2(\Omega)}^2 + || \phi_1 ||_{H^{-1}(\Omega)} \leqslant \frac{ C_1 a^4 2^{12a} M }{ \delta^2 R_+^2 } \left( \frac{ R_+}{R_-} \right)^{4a+5}  e^{ 2 C_2 M T } \sum_{ i = 1 }^2 \int_{ ( (-T,T) \times \omega_i ) \cap \mc{D}} |\phi(t,x)|^2,
\end{equation}
holds true for any solution $\phi \in C^2 ( \mc{U} ) \cap C^1 ( \bar{\mc{U}} )$ of \eqref{eq.wave_obs} that also satisfies $\phi |_{ \partial \mc{U} \cap ( \mc{D}_{ P_1 } \cup \mc{D}_{ P_2 } ) } = 0$.
\end{theorem}

\begin{remark}
From the statement of Theorem \ref{thm.obs_int}, it is clear that we now need to apply the Carleman estimate around two observation points, \(P_1, P_2\). This is due to the fact that the weights in the estimate \eqref{eq.carl_int_P} vanish at the observation point, which is now inside the domain. Hence, we must consider the estimate for two interior points and add them to get the contribution from the whole domain.
\end{remark}

\noindent\textbf{Proof of Theorem \ref{thm.obs_int}.}
The proof of Theorem \ref{thm.obs_int} is similar to that of Theorem \ref{thm.obs_ext}. Hence, we omit most of the details here and provide an outline of the proof, only giving details where necessary. 

\subsubsection{Defining a new function z}
Similar to the case of the exterior observability, here we define the function \(z\) to which we apply the Carleman estimate. We define \(z\) as in \eqref{eq.z_def} and it satisfies the same wave equation \eqref{eq.z_wave}.

\subsubsection{Application of Carleman estimate}
Fix $i \in \{ 1, 2 \}$. Let \( Q_i = (0,P_i) \). Then, due to \eqref{eq.obs_int_MR} we get
\begin{equation}
\label{eql.obs_int_Ui} \mf{U}^i := \mf{U} \cap \mf{D}_{ Q_i } \subseteq \mf{D}_{ Q_i } \cap \{ r_{ Q_i } < R_+ \} \text{.}
\end{equation}
Let \(\delta \ll 1\), and choose \(a \geq (n+2)^2\) large enough satisfying
\begin{equation}\label{eql.obs_int_a}
a \gg \delta^{-2} R_-^{-2} M^2R_+^5, \quad a \gg R_-^{-3}M R_+^2, \quad a \gg M^2 R_+^3, \quad a \gg \delta^{-1} M^2 R_+^5.
\end{equation}
Also, choose \(\varepsilon \) and \(b\) such that
\begin{equation}\label{eql.obs_int_b}
\varepsilon:= \delta^2 R_+^{-1}, \qquad b:= \delta R_+^{-1}.
\end{equation}
The previous assumptions in this section ensure that the hypotheses of Theorem \ref{thm.carl_int_P} are satisfied. Applying \eqref{eq.carl_int_P} to \(\mf{U}\) and \(Q_i\), with the above \(a, b, \varepsilon\), and also using  \eqref{eq.uvf_bound1} and \eqref{eql.obs_int_b}, gives us 
\begin{align*}
\numberthis \label{eq.carl_int_z} \frac{ C \delta^2 }{R_+^2} & \int_{\mf{U}\cap \mf{D}_{Q_i}}  \zeta_{a,b;\varepsilon}^{Q_i} (|u_{Q_i} \partial_{u_{Q_i}} z |^2 +  |v_{Q_i} \partial_{v_{Q_i}} z |^2 + f_{Q_i} g^{ab}\slashed\nabla^{Q_i}_a z \slashed\nabla^{Q_i}_b z - f_{Q_i} g^{CD} \tilde\nabla^{Q_i}_C z \tilde\nabla^{Q_i}_D z  ) \\ 
& \quad + \frac{ C \delta a^2}{R_+^2} \int_{\mf{U}\cap \mf{D}_{Q_i}}\zeta_{a,b;\varepsilon}^{Q_i} z^2 \\
& \qquad \qquad\leqslant \frac{1}{a}\int_{\mf{U} \cap \mf{D}_{Q_i}} \zeta_{a,b;\varepsilon}^{Q_i} f_{Q_i} |\square z|^2 +  a^2 R_+^3 \int_{( (-T,T)^2 \times \omega_i ) \cap \mf{D}_{Q_i} } \zeta_{a,b;\varepsilon}^{Q_i} f_{Q_i}^{-2} ( z_{t_1}^2 + z_{t_2}^2 )  \\
& \qquad \qquad \qquad + a^4 R_+^4\int_{( (-T,T)^2 \times \omega_i ) \cap \mf{D}_{Q_i}} \zeta_{a,b;\varepsilon}^{Q_i}  f_{Q_i}^{-3}  z^2.
\end{align*}
Now we use an argument similar to the one used to go from \eqref{eq.carl_square_z} to \eqref{eq:obs_ext_2}, to conclude that
\begin{align*}
\frac{ C \delta^2 }{R_+^2} \int_{\mf{U}\cap \mf{D}_{Q_i}} \zeta_{a,b;\varepsilon}^{Q_i} (|u_{Q_i} \partial_{u_{Q_i}} z |^2 & +  |v_{Q_i} \partial_{v_{Q_i}} z |^2 + f_{Q_i} g^{ab} \slashed\nabla^{Q_i}_a z \slashed\nabla^{Q_i}_b z - f_{Q_i} g^{CD} \tilde\nabla^{Q_i}_C z \tilde\nabla^{Q_i}_D z  ) \\ 
+ \frac{ C \delta a^2}{R_+^2} \int_{\mf{U}\cap \mf{D}_{Q_i}} \zeta_{a,b;\varepsilon}^{Q_i} z^2 & \leqslant  \frac{M^2 R_+}{a} \int_{\mf{U}\cap \mf{D}_{Q_i}} \zeta_{a,b;\varepsilon}^{Q_i} f_{Q_i} (z_{t_1}^2 + z_{t_2}^2 ) + a M^2 R_+^3 \int_{\mf{U}\cap \mf{D}_{Q_i}} \zeta_{a,b;\varepsilon}^{Q_i} f_{Q_i}^{-1} z^2 \numberthis \label{eq:obs_int_2} \\
 & \qquad  + a^2 R_+^3 \int_{( (-T,T)^2 \times \omega_i ) \cap \mf{D}_{Q_i} } \zeta_{a,b;\varepsilon}^{Q_i} f^{-2} ( z_{t_1}^2 + z_{t_2}^2 ) \\
& \qquad + a^4 R_+^4 \int_{( (-T,T)^2 \times \omega_i ) \cap \mf{D}_{Q_i}} \zeta_{a,b;\varepsilon}^{Q_i} f_{Q_i}^{-3} z^2. 
\end{align*}
We use the notations \( I_1^i, I_0^i \) to denote the first and second terms on the RHS of the above estimate and \( I_\omega^i \) to denote the terms with integral over \(\omega_i\). Partitioning \(\mf{U}_i\) into
\begin{align}
\label{eql.obs_int_Udec} \mf{U}^i_\leq &:= \mf{U}^i \cap \left\{ \frac{ f_{Q_i} }{ ( 1 + \varepsilon u_{Q_i} ) ( 1 - \varepsilon v_{Q_i} ) } \leq \frac{ R_-^2 }{ 64 } \right\} \text{,} \\
\notag \mf{U}^i_> &:= \mf{U}^i \cap \left\{ \frac{ f_{Q_i} }{ ( 1 + \varepsilon u_{Q_i} ) ( 1 - \varepsilon v_{Q_i} ) } > \frac{ R_-^2 }{ 64 } \right\} \text{,} 
\end{align}
and denoting by \( I^i_{0, \leqslant}, I^i_{1,\leqslant} \) and \( I^i_{0, >}, I^i_{1, >} \) the parts of \(I^i_0, I^i_1 \) over the regions \( \mf{U}^i_\leq \) and \( \mf{U}^i_> \) respectively, shows that 
\begin{align*}
C \int_{\mf{U}^i_>} \zeta_{a,b;\varepsilon}^{Q_i} \bigg[ \frac{ \delta^2 R_-^2}{R_+^3}( - u_{Q_i} |\partial_{u_{Q_i}} z |^2 + v_{Q_i} |\partial_{v_{Q_i}} z |^2 + & v_{Q_i} g^{ab}\slashed\nabla^{Q_i}_a z \slashed\nabla^{Q_i}_b z - v_{Q_i} g^{CD} \tilde\nabla^{Q_i}_C z \tilde\nabla^{Q_i}_D z  ) + \frac{ \delta a^2}{R_+^2}  z^2 \bigg] \\
\numberthis \label{eq.shnk_int_dom} & \leqslant I^i_{0, \leqslant} + I^i_{1,\leqslant} + I^i_\omega.
\end{align*}
Now, on the region \( |\tau| \leqslant \frac{R_-}{4} \), we use the notation
\begin{align*}
\mc{V}^i := \Omega \cap \left\{ r_{ P_i } > \frac{ 3 R_- }{4} \right\} \text{,} \qquad \mc{W}^i := \mc{V}^i \times \left\{ (t_1,t_2): |\tau| \leqslant \frac{R_-}{4} \right\} \text{.}
\end{align*}
Then, due to \eqref{eq.ruvf_shift} and \eqref{eq.carleman_weight_P}, we have the estimates
\begin{align}
\label{eql.obs_int_50} - u_{ P_i } |_{ \mc{W}^i } \geq \frac{ R_- }{4} \text{,} &\qquad v_{ P_i } |_{\mc{W}^i} \geq \frac{ R_- }{4} \text{,} \\
\notag \frac{ f_{ P_i } }{ ( 1 + \varepsilon u_{ P_i } ) ( 1 - \varepsilon v_{ P_i } ) } \bigg|_{ \mc{W}^i } \geq \frac{ R_-^2 }{ 16 } \text{,} &\qquad \zeta_{ a, b; \varepsilon }^{ P_i } \big|_{ \mc{W}^i } \geq \left( \frac{ R_- }{4} e^\frac{ b R_- }{4} \right)^{ 4a } \text{.} 
\end{align}
Hence, we get
\begin{align*}
\mc{W}^i \subseteq \mf{U}^i_> \text{.}
\end{align*}
Similar to the proof in the exterior observability case, we get
\begin{equation}\label{eql.obs_int_7}
 \frac{ C \delta^2 R_-^3}{R_+^3} \left( \frac{ R_- }{4} e^\frac{ b R_- }{4} \right)^{ 4 a } \iint_{\left\{  |\tau | < \frac{ R_- }{4} \right\} } \int_{\mc{V}^i}  ( |\nabla_{t_1,t_2,x} z |^2 + z^2)  \leqslant I^i_{0, \leqslant} + I^i_{1,\leqslant} + I^i_\omega.
\end{equation}
Next, due to \eqref{eq.obs_int_MR}, we get
\begin{align*}
\mc{V}^1 \cup \mc{V}^2 = \Omega \text{,}
\end{align*}
and then summing over \(i\) shows that
\begin{equation}
 \frac{ C \delta^2 R_-^3}{R_+^3} \left( \frac{ R_- }{4} e^\frac{ b R_- }{4} \right)^{ 4 a } \iint_{\left\{  |\tau| < \frac{ R_- }{4} \right\} } \int_\Omega  ( |\nabla_{t_1,t_2,x} z |^2 + z^2)  \leqslant \sum_{i=1}^2 (I^i_{0, \leqslant} + I^i_{1,\leqslant} + I^i_\omega).
\end{equation}
Then proceeding in the same way as we did in the proof of Theorem \ref{thm.obs_ext}, we find estimates for the weights inside the integral to take them outside of the integral, and get that
\begin{align*}
 \label{eql.obs_int_9}\numberthis & \frac{ C \delta^2 R_-^3}{R_+^3} \left( \frac{ R_- }{4} e^\frac{ b R_- }{4} \right)^{ 4 a } \iint_{\left\{  |\tau| < \frac{ R_- }{4} \right\} } \int_\Omega  ( |\nabla_{t_1,t_2,x} z |^2 + z^2) \\
& \qquad\leqslant \sum_{i=1}^2 \bigg[ \frac{a^2 M^2 R_+^3}{R_-^2} \left( \frac{ R_- }{ 8 } e^\frac{ b R_- }{ 8 } \right)^{ 4 a } \int_{\mf{U}^i_{\leqslant}} z^2 + \frac{M^2 R_+ R_-^2}{a} \left( \frac{ R_- }{ 8 } e^\frac{ b R_- }{ 8 } \right)^{ 4 a } \int_{\mf{U}^i_{\leqslant}} (z_{t_1}^2 + z_{t_2}^2) \\
& \qquad \qquad \ \ + a^2 2^{4a} R_+^{4a-1} \int_{ ( (-T,T)^2 \times \omega^i ) \cap \mf{D}_{Q_i}} ( z_{t_1}^2 +  z_{t_2}^2 ) + a^4 2^{4a} R_+^{4a-2} \int_{ ( (-T,T)^2 \times \omega^i ) \cap \mf{D}_{Q_i}} z^2 \bigg].
\end{align*}

\subsubsection{Going from \(z\) to \(\phi\).}
We use the same derivation as done in Section \ref{sssec_zphi}, to go back to the original function \(\phi\) and use the fact that \( Q_i = (0, P_i) \), to obtain
\begin{align}
\notag \frac{ C \delta^2 R_-^4}{R_+^3}  \left( \frac{ R_- }{4} e^\frac{ b R_- }{4} \right)^{ 4 a } \int_\Omega dx \int_{- \frac{R_-}{64}}^{\frac{R_-}{64}} dt |\phi(t,x)|^2 & \leqslant \sum_{i=1}^2 \bigg[ \frac{a^2 M^2 R_+^5}{R_-^2} \left( \frac{ R_- }{ 8 } e^\frac{ b R_- }{ 8 } \right)^{ 4 a } \int_{\mc{U} \cap \mc{D}_{P_i}} |\phi(t,x)|^2\\
\label{eql.obs_int_14} &  \qquad + a^4 2^{4a} R_+^{4a} \int_{ ( (-T,T) \times \omega^i ) \cap \mc{D}_{P_i}} |\phi(t,x)|^2 \bigg].
\end{align}

\subsubsection{Application of energy estimate lemmas}
The term on the LHS is the same as we had for the exterior observability and will be dealt with exactly the same way as before. Since we want to prove \eqref{eq.obs_int_est}, the only term that we need to consider is the first term in the RHS of \eqref{eql.obs_int_14}.  For this term, using an analogous version of \eqref{eq.energy2} (with the choice of constants from \eqref{eq.obs_int_MR}), shows
\begin{align*}
\frac{a^2 M^2 R_+^5}{R_-^2} \left( \frac{ R_- }{ 8 } e^\frac{ b R_- }{ 8 } \right)^{ 4 a } \sum_{i=1}^2 \int_{\mc{U} \cap \mc{D}_{P_i}} |\phi(t,x)|^2 \leqslant \frac{a^2 C_1 M^3 R_+^6}{R_-^2} \left( \frac{ R_- }{ 8 } e^\frac{ b R_- }{ 8 } \right)^{ 4 a } e^{C_2(M_0 + M_1) R_+} E(0). 
\end{align*}
The rest of the proof proceeds similar to Section \ref{sssec_obsenergy}, that is we absorb the constants analogous to \eqref{eql.obs_ext_16}, and conclude that 
\begin{align*}
E(0) \leqslant C_1 a^4 2^{4a} M R_+^{4a} \sum_{i=1}^2 \int_{ ( (-T,T) \times \omega^i ) \cap \mc{D}_{P_i}} |\phi(t,x)|^2,
\end{align*}
which after another application of Lemma \ref{thm.energy2} (to go from \(E(0)\) to \(E(-T)\)) completes the proof of the interior observability estimate \eqref{eq.obs_int_est}. \hfill \(\qed\)

\medskip

\subsection{ Conclusion }\label{ssec_proofofobsthm} Now we are all set to present the proof of the main observability result, that is, Theorem \ref{intro_thm_obs}.

\begin{proof}[Proof of Theorem \ref{intro_thm_obs}]
We assume, without loss of generality, that the observation point is the origin, that is, let \( x_0 =0 \).
To prove this theorem, we divide the proof into different cases. Each case depends on the location of the origin with respect to the domain.

For the first case, let \( 0 \notin \bar{\Omega}\). Then, using the exterior observability result Theorem \ref{thm.obs_ext} directly gives us \eqref{eq.intro_thm_obs}.

For the second case, we first prove the result if \( 0 \in \Omega\). We take two distinct points \( x_1, x_2 \in \Omega\), and let \( P_i = (0, x_i)\) for \( i \in \{1,2\}\). Now, we can choose \(P_1, P_2\) sufficiently close enough to each other as well as to the origin, such that the observation regions from \eqref{eq.obs_int_omega}-\eqref{eq.obs_int_est} satisfy
\[ \left\{((-T,T) \times \omega_1 ) \cap \mc{D} \right\} \cup \left\{((-T,T) \times \omega_2 ) \cap \mc{D}\right\} \subset W. \]
Then, applying the interior observability result Theorem \ref{thm.obs_int} to this choice shows that \eqref{eq.intro_thm_obs2} holds for this case.

Finally, let \( 0 \in \partial \Omega \). We choose a point \(x_1 \notin \bar\Omega\) that is close enough to \( 0 \) such that the observation region, corresponding to \(x_1\), obtained from \eqref{eq.obs_ext_omega}-\eqref{eq.obs_ext_est} is a subset of \( W \). Then applying Theorem \ref{thm.obs_ext} to this choice shows that \eqref{eq.intro_thm_obs2} is true, and this concludes the proof of the theorem.
\end{proof}

\end{document}